\documentclass[reqno]{amsart}

\usepackage{tikz}

\usepackage{enumitem}
\usepackage{nicefrac}
\usepackage{amssymb}
\usepackage[perpage,symbol*]{footmisc}
\usepackage{bbm}
\usepackage{color}
\usepackage{etoolbox}
\usepackage{hyperref}

\DefineFNsymbols{daggers}{\dag\ddag{\dag\dag}{\ddag\ddag}}
\setfnsymbol{daggers}

\newcommand*{\mailto}[1]{\href{mailto:#1}{\nolinkurl{#1}}}

\newcommand{\doi}[1]{\href{http://dx.doi.org/#1}{doi:#1}}
\newcommand{\msc}[1]{\href{http://www.ams.org/msc/msc2020.html?t=&s=#1}{#1}}

\newcommand{\qbinom}[2] {{#1 \brack #2}_q}

\newtheorem{theorem}{Theorem}[section]
\newtheorem{lemma}[theorem]{Lemma}
\newtheorem{proposition}[theorem]{Proposition}
\newtheorem{corollary}[theorem]{Corollary}

\theoremstyle{definition}
\newtheorem{definition}[theorem]{Definition}

\newtheorem{example}[theorem]{Example}

\theoremstyle{remark}
\newtheorem{remark}[theorem]{Remark}
\AtEndEnvironment{remark}{\hfill $\lozenge$}

\newcommand{\R}{{\mathbb R}}
\newcommand{\N}{{\mathbb N}}
\newcommand{\Z}{{\mathbb Z}}
\newcommand{\C}{{\mathbb C}}

\newcommand{\mlam}{\boldsymbol{\lambda}}

\newcommand{\mgam}{\boldsymbol{\gamma}}
\newcommand{\mulk}{\boldsymbol{{\rm k}}}

\newcommand{\wand}{\mathrm{D}}

\newcommand{\Wr}{\mathsf{w}}

\newcommand{\E}{\mathrm{e}}
\newcommand{\I}{\mathrm{i}}
\newcommand{\sgn}{\mathrm{sgn}}
\newcommand{\supp}{\mathrm{supp}}

\newcommand{\loc}{{\mathrm{loc}}}
\newcommand{\cc}{{\mathrm{c}}}

\newcommand{\wt}{\tilde}
\newcommand{\eps}{\varepsilon}
\newcommand{\cI}{\mathcal{I}}

\newcommand{\OO}{\mathcal{O}}
\newcommand{\oo}{o}

\newcommand{\ledot}{\,\cdot\,}
\newcommand{\redot}{\cdot\,}

\newcommand{\id}{{\mathbbm 1}}

\newcommand{\qd}{{[1]}}

\newcommand{\dip}{\upsilon}

\newcommand{\SM}{\mathcal{R}}
\newcommand{\Peakons}{\mathcal{P}}

\newcommand{\CHdom}{\mathcal{D}}

\newcommand{\NLz}{(z,0-)}
\newcommand{\NL}{(0-)}

\newcommand{\hyp}[5]{\,\mbox{}_{#1}F_{#2}\!\left(
  \genfrac{}{}{0pt}{}{#3}{#4};#5\right)}

\makeatletter
\@namedef{subjclassname@2020}{%
  \textup{2020} Mathematics Subject Classification}
  
\newcommand{\Hm}[1]{\leavevmode{\marginpar{\tiny%
			$\hbox to 0mm{\hspace*{-0.5mm}$\leftarrow$\hss}%
			\vcenter{\vrule depth 0.1mm height 0.1mm width \the\marginparwidth}%
			\hbox to
			0mm{\hss$\rightarrow$\hspace*{-0.5mm}}$\\\relax\raggedright #1}}}  
  
\makeatletter
\newcommand{\dlmf}[1]{%
\cite[%
 \def\nextitem{\def\nextitem{, }}%
 \@for \el:=#1\do{\nextitem\expandafter\dlmf@eq@href\el...\end}%
]{dlmf}%
}
\def\dlmf@eq@href#1.#2.#3.#4\end{%
  \href{http://dlmf.nist.gov/#1.#2.E#3}{(#1.#2.#3)}}
\makeatother  

\newcommand{\eat}[1]{}

\numberwithin{equation}{section}

\begin{document}

\title[Infinite-peakon solutions]{Infinite-peakon solutions of the Camassa--Holm equation}

\author[X.-K.\ Chang]{Xiang-Ke Chang}
\address{SKLMS \& ICMSEC, Academy of Mathematics and Systems Science, Chinese Academy of Sciences, Beijing 100190, P.\ R.\ China, and School of Mathematical Sciences, University of Chinese Academy of Sciences, Beijing 100049, P.\ R.\ China}
\email{\mailto{changxk@lsec.cc.ac.cn}}
\thanks{X.-K.\ Chang was supported by the National Natural Science Foundation of China (NSFC) under Grants No.~12222119 and~12288201, and by the Youth Innovation Promotion Association CAS}

\author[J.\ Eckhardt]{Jonathan Eckhardt}
\address{Department of Mathematical Sciences\\ Loughborough University\\ Epinal Way\\ Loughborough\\ Leicestershire LE11 3TU \\ UK}
\email{\mailto{J.Eckhardt@lboro.ac.uk}}

\author[A.\ Kostenko]{Aleksey Kostenko}
\address{Faculty of Mathematics and Physics\\ University of Ljubljana\\ Jadranska ul.\ 19\\ 1000 Ljubljana\\ Slovenia\\ and 
Institute for Analysis and Scientific Computing\\ Vienna University of Technology\\ Wiedner Hauptstra\ss e 8-10/101\\1040 Vienna\\ Austria}
\email{\mailto{Aleksey.Kostenko@fmf.uni-lj.si}}
\thanks{\ \ \ \ A.\ Kostenko was supported by the Slovenian Research Agency (ARIS) under Grants No.~N1-0137 and~P1-0291 and by the Austrian Science Fund (FWF) under Grant No.~I4600.}


\keywords{Conservative Camassa--Holm flow, solitons, multi-peakons, inverse spectral transform, Hamburger moment problem}
\subjclass[2020]{Primary \msc{37K15}, \msc{34L05}; Secondary \msc{34A55}, \msc{35Q51}}

\begin{abstract}
We study a class of (conservative) low regularity solutions to the Camassa--Holm equation on the line by exploiting the classical moment problem (in the framework of generalized indefinite strings) to develop the inverse spectral transform method. 
In particular, we identify explicitly the solutions that are amenable to this approach, which include solutions made up of infinitely many peaked solitons (peakons). 
We determine which part of the solution can be recovered from the moments of the underlying spectral measure and provide explicit formulas.
We show that the solution can be recovered completely if the corresponding moment problem is determinate, in which case the solution is a (potentially infinite) superposition of peakons. However, we also explore the situation when the underlying moment problem is indeterminate.

  As an application, our results are then used to investigate the long-time behavior of solutions. 
 We will demonstrate this on three exemplary cases of solutions with: (i) discrete underlying spectrum (one may choose for this initial data corresponding to Al-Salam--Carlitz II polynomials; notice that they correspond to an indeterminate moment problem); (ii) step-like initial data associated with the Laguerre polynomials, and (iii) asymptotically eventually periodic initial data associated with the Jacobi polynomials.
 \end{abstract}

\maketitle

{\scriptsize{\tableofcontents}}

\section{Introduction}
 
\subsection{The Camassa--Holm equation and multi-peakons}

The Camassa--Holm equation is a nonlinear wave equation given by 
  \begin{equation}\label{eqnCH}
   u_{t} -u_{xxt}  = 2u_x u_{xx} - 3uu_x + u u_{xxx}.
  \end{equation}
It was first noticed by B.\ Fuchsteiner and A.\ S.\ Fokas~\cite{fofu81} to be formally integrable with Hamiltonians given by
 \begin{align}\label{eqnCHHam}
  H_1 & = \int u^2 + u_x^2\, dx, & H_2 & =  \frac{1}{2}\int u^3 + uu_x^2\, dx,
   \end{align}
   where integration is either over $\R$ or $\mathbb{T}$. 
 Its formulation as a non-local conservation law
\begin{align}
\label{eqnCHweak}
  u_t + u u_x + P_x  = 0, 
\end{align}
 where the source term $P$ is defined (for suitable functions $u$) as the convolution 
 \begin{align}\label{eq:PequCHdef}
 P = \frac{1}{2}\E^{-|\cdot|} \ast \biggl(u^2 + \frac{1}{2} u_x^2\biggr), 
 \end{align} 
is reminiscent of the three-dimensional incompressible Euler equation. However, its intensive study began only in the 1990s after the paper~\cite{caho93} by R.~Camassa and D.~D.~Holm, who, in particular, observed that~\eqref{eqnCH} is completely integrable, that is, enjoys a Lax pair structure. Among many interesting features (like finite time blow-up of smooth solutions that resembles wave-breaking to some extent~\cite{coes98,coes98b,mc04}), the Camassa--Holm equation exhibits soliton interaction, including peaked ones, called {\em peakons}, which is another fundamental discovery of~\cite{caho93}. 
More specifically, \eqref{eqnCHweak}--\eqref{eq:PequCHdef} admits so-called multi-peakon solutions of the form
\begin{align}\label{eqnMP}
  u(x,t) = \frac{1}{2}\sum_{n=1}^N p_n(t)\, \E^{-|x-q_n(t)|}, 
 \end{align}
 where the coefficients $p_n$ and $q_n$ satisfy the  system of ordinary differential equations 
 \begin{align}\label{eqnMPsys}
  \dot{q}_n  & = \frac{1}{2}\sum_{k=1}^N p_k\, \E^{-|q_n-q_k|}, &
  \dot{p}_n  & = \frac{1}{2}\sum_{k=1}^N p_n p_k\, \sgn(q_n - q_k)\, \E^{-|q_n-q_k|}.
 \end{align}
 We note that the system in~\eqref{eqnMPsys} is Hamiltonian, that is,  
 \begin{align}\label{eqnMPsysH}
  \dot{q}_n & = \frac{\partial H(p,q)}{\partial p_n}, & 
  \dot{p}_n & = - \frac{\partial H(p,q)}{\partial q_n}, 
 \end{align}
 with the Hamiltonian given by 
 \begin{align}\label{eq:H}
  H(p,q)= \frac{1}{4} \sum_{n,k=1}^N p_n p_k \, \E^{-|q_n-q_k|} =  2\|u\|^2_{H^1(\R)}.
 \end{align}
 
The isospectral problem for the Camassa--Holm equation is a Sturm--Liouville problem of the form 
  \begin{align}\label{eq:Spec}
  - f'' + \frac{1}{4}f=z\, \omega f,
  \end{align}
 where $\omega = u - u_{xx}$ (which will have to be understood in a distributional sense in general), is known as the {\em momentum}. 
 In particular, for $u$ given by~\eqref{eqnMP}, one has
  \begin{align}\label{eqnMPomega}
  \omega(x,t) = \sum_{n=1}^N p_n(t)\, \delta_{q_n(t)}(x), 
 \end{align}
 where $\delta_q$ is the Dirac delta-measure located at $q\in \R$.
 This spectral problem is connected via a simple change of variables with {\em Krein strings}, an object introduced by M.\ G.\ Krein in the 1950s~\cite{kr52,kakr74}.   The work of Krein can be viewed as a far reaching generalization of investigations of T.\ J.\ Stieltjes on the moment problem~\cite{sti}. In particular, the framework of Krein strings provides a mechanical interpretation of Stieltjes' work (see~\cite[Appendix]{akh} and~\cite[Section~13]{kakr74}) and thus one may use the latter to integrate multi-peakon solutions. This has been done by R.\ Beals, D.\ H.\ Sattinger and J.\ Szmigielski in~\cite{besasz00}, who found explicit expressions for the coefficients of multi-peakons using classical formulas of Stieltjes. 
 
 The Stieltjes moment problem (or, more generally, the setting of Krein strings) requires $\omega$ to be a positive Borel measure, so that all peakon heights $p_n$ in~\eqref{eqnMP} must be of one sign.
 This entails that all peakons move in the same direction and, moreover, that the positions $q_n$ of the peakons stay distinct for all times. 
  On the one hand, the formulas of Stieltjes are purely algebraic and hence continue to work for sign indefinite heights, but on the other hand, something peculiar happens -- it is not always possible to solve the corresponding inverse spectral problem in this case. 
  This means that for certain spectral data there is no measure $\omega$ such that the problem~\eqref{eq:Spec} has the given spectral data. In fact, this happens exactly at times when peakons collide. Namely, unlike the Korteweg--de Vries equation, the Camassa--Holm equation has solitons that travel in either direction. When allowing peakons with heights of different sign, one sees that peakons with positive heights move to the right whereas peakons with negative heights move to the left, which leads to collisions. The most illustrative example of such a situation is a symmetric peakon-antipeakon interaction, where the solution $u$ is given by  
\begin{align}\label{eq:PeakPM}
u(x,t) = \frac{1}{2}\left(p(t)\E^{-|x+q(t)|} - p(t)\E^{-|x-q(t)|}\right).
\end{align}
Explicit formulas for $p$ and $q$ can be found in~\cite{caho93} for example. If both $p(0)$ and $q(0)$ are positive, then at a certain time $t^\times>0$ the peakons meet and $u$ tends pointwise to zero as $t\uparrow t^\times$. Moreover, all this happens in such a way that the solution $u$ stays uniformly bounded but its derivative develops a singularity at the point where the two peakons collide; for $u$ given by~\eqref{eq:PeakPM}, one has $u_x(0,t)\to -\infty$ as $t\uparrow t^\times$ (see Figure~\ref{fig:peak-antipeak}). 
\begin{figure}[h]
  \centering
  \includegraphics[width=1\textwidth]{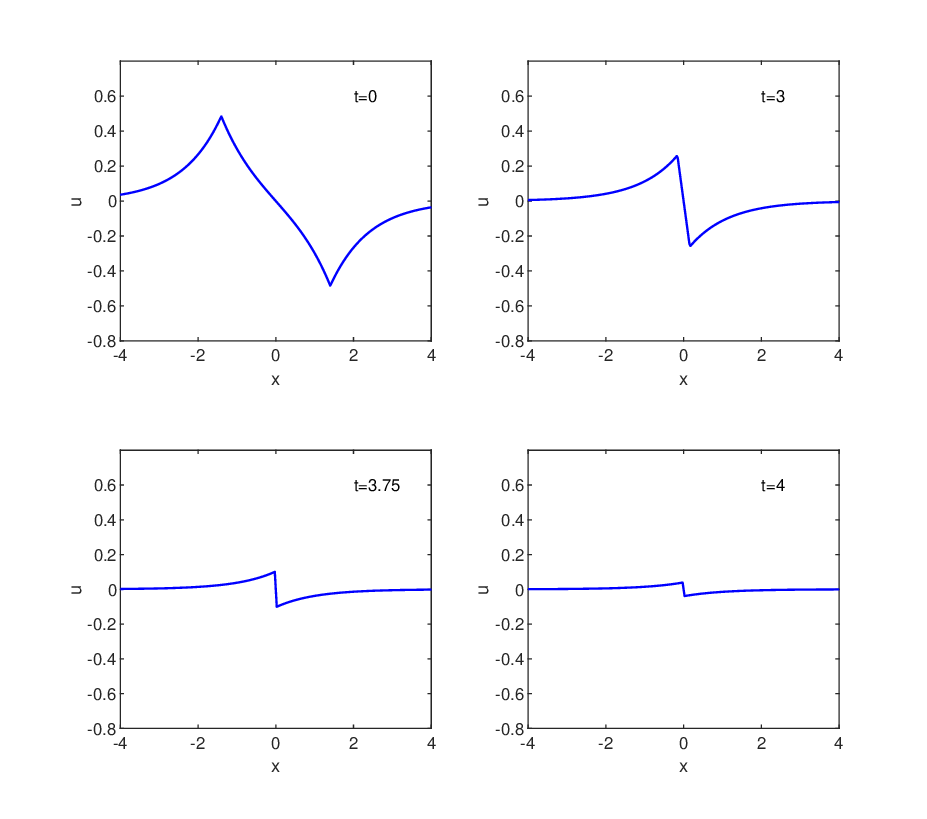}
  \caption{{\small  Peakon-antipeakon interaction:\\ Snapshots of $u(x,t)$ with $p(0)=\nicefrac{65}{63}$ and $q(0)=\log\nicefrac{65}{16}$;
  the collision time is $t^\times=\log 64 \approx4.16$.}}
  \label{fig:peak-antipeak} 
  \end{figure}
 Even for general multi-peakons, such collisions always happen in pairs, that is, it is not possible for three or more peakons to collide at a single point (see~\cite{besasz00,besasz01,ConservMP} for example).

It was noticed independently in~\cite{brco07} and~\cite{hora07} that when solutions blow up, the energy concentrates on sets of Lebesgue measure zero (in the peakon-antipeakon interaction described above, the whole energy concentrates at the point of collision). 
This led to the notion of {\em conservative solutions}, which are weak solutions $(u,\mu)$ of the system 
  \begin{align}
 \begin{split}\label{eqnOurCH}
  u_t + u u_x + P_x & = 0, \\
  \mu_t + (u\mu)_x & = (u^3 - 2Pu)_x, 
 \end{split}
 \end{align}
 where the auxiliary function $P$ satisfies
 \begin{align}\label{eq:PequDef}
  P - P_{xx} & = \frac{u^2+ \mu}{2}.
 \end{align} 
We will call system~\eqref{eqnOurCH} the {\em two-component Camassa--Holm system} because it includes the Camassa--Holm equation~\eqref{eqnCH} (set $\mu = u^2 + u_x^2$) as well as its two-component generalization (see \cite{chlizh06, coiv08, esleyi07, hoiv11})
 \begin{align}
 \begin{split}\label{eqn2CH}
    u_{t} -u_{xxt} & = 2u_x u_{xx} - 3uu_x + u u_{xxx} - \varrho \varrho_x, \\
   \varrho_t & = -u_x \varrho - u\varrho_x,
 \end{split}
 \end{align}
where one just needs to set $\mu = u^2 + u_x^2 + \varrho^2$. The role of the additional positive Borel measure $\mu$ is to control the loss of energy at times of blow-up. 
Solutions of this kind have been constructed by a generalized method of characteristics that relies on a transformation from Eulerian to Lagrangian coordinates and was accomplished for various classes of initial data in~\cite{brco07, hora07, grhora12, grhora12b}.
 
 From our point of view, such conservative solutions are of special interest because they preserve the integrable structure of the Camassa--Holm equation and can be obtained by employing the inverse spectral transform method~\cite{ConservMP,ConservCH,Eplusminus}. 
This approach is based on the solution of an inverse problem, which is equivalent to that for so-called {\em generalized indefinite strings}~\cite{IndefiniteString}, a spectral problem that generalizes Krein strings and originated from work of M.\ G.\ Krein and H.\ Langer~\cite{krla79} on the indefinite moment problem.
More specifically, it was shown in~\cite{ConservMP} that replacing~\eqref{eq:Spec} with  
\begin{align}\label{eqnISP}
 -f'' + \frac{1}{4} f = z\, \omega f + z^2 \dip f, 
\end{align}
where $\dip$ is a positive Borel measure, one is able to integrate conservative multi-peakon solutions. The role of the additional coefficient $\dip$ is to control the concentration of energy on Lebesgue measure zero sets. More precisely, for a given pair $(u,\mu)$, the measure $\dip$ is defined by 
\begin{align}
 \mu(B) = \dip(B) + \int_B u(x)^2 + u'(x)^2\, dx
\end{align}
for all Borel subsets $B\subseteq \R$. 
For instance, in the peakon-antipeakon interaction mentioned above, at the time of the blow-up, the corresponding momentum $\omega$ vanishes and $\dip$ is simply $c\delta_0$, where $c$ is the total energy of the system and $\delta_0$ is the Dirac delta-measure at $x=0$~\cite[Appendix~A]{ConservMP}. These observations have been developed further in~\cite{ConservCH,Eplusminus}. In particular, in~\cite{Eplusminus}, two of us extended the conservative multi-peakon flow to a rather wide phase space $\CHdom$ of initial data that requires strong decay at one endpoint but only mild boundedness-type conditions at the other endpoint (see Definition~\ref{defPS}). There we were able to solve the inverse spectral problem for~\eqref{eqnISP} with coefficients in $\CHdom$, which allowed us to employ the inverse spectral transform method to show that the two-component Camassa--Holm system is well-posed on $\CHdom$ in a certain sense (see Theorem~\ref{thm:consCHweak} and Remark~\ref{rem:wellposed} below). 

The main goal of the current paper is to investigate conservative solutions constructed in~\cite{Eplusminus} with the help of the classical moment problem\footnote{By this we always mean the {\em Hamburger moment problem} if it is not explicitly stated otherwise.} (see~\cite{akh,sc17} for example). 
Since the moment problem can be included into the spectral theory of generalized indefinite strings (see~\cite{IndMoment}), all the necessary prerequisites are at our disposal. 
On the one hand, the moment problem approach clearly has a narrower range of applicability since it requires the existence of finite moments of the spectral measure. 
On the other hand, its obvious advantage is the explicit procedure to solve the inverse spectral problem. 
The latter not only allows to provide explicit formulas for solutions, but also to investigate these solutions both numerically and analytically. 
Let us emphasize that in this paper we are mostly interested in the non-degenerate case of the moment problem, that is, when the support of the measure is an infinite set (the case of a finite set always leads to  multi-peakons and was studied in the works~\cite{besasz00,ConservMP} mentioned above).

\subsection{Overview of the main results}
 
Section~\ref{secPre} is of preliminary nature. Here we first recall the main notions and facts from~\cite{Eplusminus} needed in this article: The definition and properties of the phase spaces $\CHdom$ (Definition~\ref{defPS}) and $\CHdom^+$, the associated spectral problem~\eqref{eqnDE} as well as the solution of the inverse problem for both $\CHdom$ (Theorem~\ref{thm:ISP}) and $\CHdom^+$ (Corollary~\ref{corIP+}); construction of conservative solutions by means of the inverse spectral transform (Definition~\ref{def:CHspecflow} and Theorem~\ref{thm:consCHweak}). We also briefly recall in this section the notion of a generalized indefinite string and connect it to the spectral problem~\eqref{eqnDE} (for the reader's convenience, a more detailed discussion of generalized indefinite strings and their relation to the classical moment problem is included in Appendix~\ref{appMoment}). Moreover, we describe the sets of pairs $(u,\mu)$, which are of particular interest to us in this article: Multi-peakon profiles (Definition~\ref{def:MP}) and pairs that are of multi-peakon form from the left, which in particular include infinite multi-peakon profiles (Remark~\ref{rem:exmpls}).

Section~\ref{secMoment} and Section~\ref{sec:conssolMP} form the core of this article. 
As it was mentioned already, application of the classical moment problem requires the existence of finite moments of the underlying spectral measure. 
Therefore, we begin by characterizing all pairs $(u,\mu)$ in $\CHdom$ such that the spectral measure $\rho$ associated with the problem~\eqref{eqnISP} has at least $2K+1$ finite moments $s_0, s_1, \ldots, s_{2K}$ for some $K\in \N\cup\{0\}$, where  
\begin{align}
s_k = \int_\R \lambda^k \rho(d\lambda).
\end{align}
It turns out (see Theorem~\ref{thmMomK}) that for this it is necessary and sufficient that $(u,\mu)$ is of multi-peakon form to the left of some point $\ell\in\R$, that is,  
 \begin{align}\label{eq:MPell}
 \omega|_{(-\infty,x_N)} & = \sum_{n=1}^{N-1} \omega_n \delta_{x_n}, & \dip|_{(-\infty,x_N]} & =  \sum_{n=1}^{N} \dip_n \delta_{x_n},
 \end{align}
 for some $N\in\N\cup\{0\}$, increasing points $x_1,\ldots,x_N$, real weights $\omega_1,\ldots,\omega_N$ and non-negative weights $\dip_1,\ldots,\dip_N$. 
 The relationship between $N$ and $K$ is made explicit in~\eqref{eq:KviaN}, from which it follows that all moments exist exactly when~\eqref{eq:MPell} takes place with $N=\infty$. 
 One can recover this multi-peakon part of the pair $(u,\mu)$ explicitly from the moments by using the corresponding Hankel determinants (see Corollary~\ref{corStrKForm} and also Remark~\ref{rem:UatInfty}).  A relevant question in this regard is whether one is able to recover the pair $(u,\mu)$ completely from the corresponding moments (which are understood to also include the first negative moment $s_{-1}$). Of course, if the support of $\rho$ is an infinite set, for this to happen it is necessary to have all the moments of $\rho$ finite. However, still we would be able to recover $(u,\mu)$ only up to a certain point $L$, where all the $x_n$-s accumulate.  Clearly, if $L=\infty$, then we obtain $(u,\mu)$ on all of $\R$ out of the moments. The question for which $\rho$ the corresponding $L$ equals infinity can be answered with the help of the moment problem. More precisely, {\em $L=\infty$ if and only if the corresponding moment problem is determinate}, that is, $\rho$ is the unique non-negative Borel measure on $\R$ whose moments coincide with the sequence $(s_k)_{k=0}^\infty$. To a certain extent this answer should not be too surprising since it has its roots in the work of Krein on the Stieltjes moment problem. 
Now it remains to clarify whether one would be able to recover the pair $(u,\mu)$ completely if $L<\infty$, that is, if the corresponding moment problem is indeterminate. The latter means that there are infinitely many measures having the same moments. It is somehow natural (at least from the perspective of Krein's parameterization of solutions to the Stieltjes moment problem in the indeterminate case) that the moments of $\rho$ do not contain any information about $(u,\mu)$ in the interval $(L,\infty)$, which means that both $\omega$ and $\dip$ are not supported on $(L,\infty)$. 
It turns out the latter can also be characterized by means of the moment problem: If $L<\infty$, then the pair $(u,\mu)$ has no support in the interval $(L,\infty)$ and $\dip(\{L\})=0$ if and only if the associated spectral measure is an {\em N-extremal solution} of the corresponding moment problem (this will be explained in detail in Section~\ref{sec:DiscrSpec}). 

 In Section~\ref{sec:conssolMP}, we apply the above results to the conservative Camassa--Holm flow. The remarkable fact is that the properties like determinacy/indeterminacy of the moment problem and N-extremality are preserved under the flow since on the spectral side the evolution is simply given by  
   \begin{align}\label{eqnSMEvoIntro}
    \rho(d\lambda;t) = \E^{-\frac{t}{2\lambda}} \rho_0(d\lambda),
  \end{align}
and for every $t\in\R$ the function $\lambda\mapsto \E^{-\frac{t}{2\lambda}}$ is bounded on the support of every $\rho_0$ arising from $(u,\mu)\in \CHdom$ (see Theorem~\ref{thm:ISP}). In particular, this implies that if $\rho_0$ has finite moments up to order $2K$, then so does the measure~\eqref{eqnSMEvoIntro} for all $t\in\R$, which means that the corresponding solution $ (u(\ledot,t),\mu(\ledot,t))$ is of multi-peakon form to the left of some point for some $t$ if and only it is of this form for all $t$. Moreover, the evolution of this multi-peakon part can be recovered from the evolution of the moments (Theorem~\ref{thm:MPevolt}).   Because the explicit formulas for the peakon part of a solution are the same as for usual multi-peakons in~\cite{besasz00}, the dynamics are alike too (Proposition~\ref{prop:dynamics}). In particular, under the additional positivity and determinacy assumptions, we show that the corresponding infinite multi-peakon solution satisfies an infinite dimensional analog of~\eqref{eqnMPsys} (see Corollary~\ref{cor:ODEforCH}). Let us also stress that these results can be seen as a manifestation of the finite propagation speed of conservative solutions. Namely, it is known that classical solutions to the Camassa--Holm equation enjoy the finite propagation speed property in the following sense~\cite{co05}: {\em If $u$ is a smooth solution to~\eqref{eqnCH} and $\omega(\ledot,0)$ has compact support, then $\omega(\ledot,t)$ has compact support for all $t$}. Setting $K=0$, our results immediately imply that if $(u_0,\mu_0)$ is such that $\omega_0$ and $\dip_0$ has no support near $-\infty$, then for the corresponding conservative solution $(u(\ledot,t),\mu(\ledot,t))$ both $\omega(\ledot,t)$ and $\dip(\ledot,t)$ also have no support near $-\infty$ for all $t\in\R$ (see Remark~\ref{rem:finspeed}).    
 
Since the set $\CHdom$ is sufficiently rich and diverse (at least from the spectral perspective), any further analysis of solutions to the two-component Camassa--Holm system would require a restriction to certain subclasses of initial data. We decided to restrict to three particular cases and our choices will be explained below. Let us also mention that all our long-time asymptotcs results are established under the additional assumption that the initial data $(u_0,\mu_0)$ belongs to the space $\CHdom^+$, which is the same as to assume that the corresponding spectral measure is supported on $(0,\infty)$. This assumption is mainly of technical nature and the results can be extended to a wider setting. 

As our first choice, we decided to consider spectral measures supported on discrete subsets of $\R$ and this is the content of Section~\ref{sec:DiscrSpec}. There are several reasons for this choice. On the one hand, for a given measure to be an N-extremal solution of an indeterminate moment problem it is necessary to be supported on a discrete subset of $\R$.  On the other hand, it is known that discreteness of the spectrum of~\eqref{eqnISP} is related to decay of the function $u$ and the measure $\dip$ (for instance, this holds true if $u\in H^1(\R)$, compare with~\eqref{eqnCHHam}, and $\dip\equiv 0$). 
In fact, the recent results in~\cite{DSpec} provide a complete characterization of pairs $(u,\mu)$ in $\CHdom$ such that the spectrum of~\eqref{eqnISP} is a discrete set (see Remark~\ref{rem:DSpec}). One of the main results in Section~\ref{sec:DiscrSpec} is Theorem~\ref{thm:longt_dis}, which establishes long-time asymptotics for the positions and heights of the individual peakons in this case. While we get a rather detailed asymptotics as $t\to-\infty$ (see~\eqref{asym_neg}), which can also be considered as an extension of Theorem~6.4 from~\cite{besasz00} as well as of Theorem~5.2 from~\cite{li09}, the results when $t\to\infty$ might seem somewhat unsatisfactory and even to a certain extent misleading (see Remark~\ref{rem:PeakDiscr+}). 
On the other hand, these asymptotics are rather natural since all the peakons move to the right and accumulate either at $\infty$ (when the moment problem is determinate) or at some finite point $L(t)$ (since $L$ also changes with time), which slows them down. 
Let us also stress that the results of this section extend and complement the results from~\cite{li09} (see Corollary~\ref{cor:Li} and Remark~\ref{rem:Li01}(c)), where a different approach was used.  More specifically, the main focus in~\cite{li09} is on pairs $(u,\mu)$ in $\CHdom^+$ such that $u$ has pure infinite multi-peakon form
\begin{align}\label{eq:PureMPintro}
u(x) =  \frac{1}{2}\sum_{n=1}^\infty \omega_n \E^{-|x-x_n|}, 
\end{align}
with positions $x_1<x_2<\dots<x_n<\dots$ accumultating at some finite point $L$ and with heights satisfying certain summability assumptions (see Remark~\ref{rem:Li01}). Under these rather restrictive assumptions, \cite{li09} contains, on the one hand, a slightly stronger claim regarding the asymptotic behavior as $t\to\infty$ (see Theorem~4.3 there), but on the other hand, it is not clear to us whether the approach of~\cite{li09} would allow to establish~\eqref{eq:LasympLi}, the asymptotic behavior of the accumulation point of peakons' positions. Indeed, the approach of~\cite{li09} differs from ours in several crucial aspects. First of all, the corresponding weak solutions are constructed in~\cite{li09} with the help of the infinite dimensional version of~\eqref{eqnMPsysH}--\eqref{eq:H}. However, most importantly the long-time analysis is based on the use of Toda flows on Hilbert--Schmidt operators~\cite{dlt85,li08} and the underlying operators are nothing but the inverses of Jacobi (tri-diagonal) matrices (in the multi-peakon case, these relationships are explained in~\cite{besasz01} and they are based on a rather widely known connection between positive Jacobi matrices and Krein--Stieltjes strings, see~\cite[Appendix]{akh}, \cite{kakr74a} and also~\cite{IndMoment}). However, as we mentioned above, a finite accumulation point for positions implies that the corresponding spectral measure is a solution of an indeterminate moment problem and therefore this limits the applicability of the approach from~\cite{li09} (indeed, comparing with generalized indefinite strings, the framework of Jacobi matrices is rather inconvenient to work with when describing solutions to the indeterminate moment problem; see Theorem~\ref{th:KreinParam} for example). Let us make one more comment in this regard. Our analysis shows that the spectral measures that arise from the class considered in~\cite{li09} are specific N-extremal solutions to the Stieltjes moment problem. It was a notoriously difficult problem to provide an explicit construction of such a measure, which was resolved only in the 1960s. More specifically, despite the fact that examples of measures that give rise to indeterminate moment problems were known already to Stieltjes,  
 the first example of an N-extremal measure was found much later by T.~S.~Chihara~\cite{chi68}. 
 Indeed,  cases of indeterminate moment problems are present within the $q$-analog of the Askey-scheme and the corresponding construction involves the Al-Salam--Carlitz polynomials of type II (see also~\cite{bech20}). In particular, we exploit the Al-Salam--Carlitz polynomials to construct an explicitly solvable pair from the class considered in~\cite{li09} (see Proposition~\ref{prop:ASCpoln}). One of the corresponding solutions is depicted in Figure~\ref{fig:ASC3} on page~\pageref{fig:ASC3}. 
 
 In Section~\ref{sec:LaguerrePeaks}, we perform a detailed study of the conservative solution associated with the Laguerre polynomials. More specifically, as initial data we take the pair $(u_0,\mu_0)\in\CHdom^+$ such that its spectral measure is explicitly given by
 \begin{align}
 \rho_{\gamma,\alpha}(d\lambda) =  \frac{1}{\Gamma(\gamma+1)}\id_{(\alpha,\infty)}(\lambda)(\lambda-\alpha)^\gamma\E^{-(\lambda-\alpha)} d\lambda,
\end{align}
where $\gamma>-1$ and $\alpha>0$ are parameters. It is well known that the corresponding moment problem is determinate. Moreover, taking into account that the support of $ \rho_{\gamma,\alpha}$ is contained in $(0,\infty)$, $\dip\equiv 0$ and hence all the information about $(u,\mu)$ is contained in $u$, which has the form~\eqref{eq:PureMPintro},
where $\omega_n>0$ for all $n\in \N$ and $x_n\uparrow \infty$. The coefficients of $u$ can be expressed with the help of the Laguerre polynomials and one can find a rather precise asymptotic behavior of  $\omega_n$ and $x_n$ for large  $n$ (see Lemma~\ref{lem:asympLaginn}). Our interest in studying this particular case stems from the fact that $u$ is close asymptotically to a step-like function. More specifically, it follows from~\cite[Theorem~13.5]{ISPforCH} that $u - (4\alpha)^{-1} \in H^1[0,\infty)$ (see also Lemma~\ref{lem:killipsimon}). Furthermore, the latter property as well as the form~\eqref{eq:PureMPintro} are preserved under the flow, that is, the corresponding solution $(u(\ledot,t),\mu(\ledot,t))$ satisfies $u(\ledot,t) - (4\alpha)^{-1} \in H^1[0,\infty)$ and 
\begin{align}\label{eq:PuretMPintro}
u(x,t) =  \frac{1}{2}\sum_{n=1}^\infty \omega_n(t) \E^{-|x-x_n(t)|}, 
\end{align}
for all $t\in\R$. One might think of this situation as of a regime with dispersion for~\eqref{eqnCH} only at $+\infty$. Our main result in this section is a peakon-wise long-time asymptotic behavior of $u(\ledot,t)$ as $|t|\to \infty$ (Theorem~\ref{thm:LaguerreMPs}). 
In sharp contrast with the purely discrete spectrum situation, the analysis shows that all peakons move asymptotically with the same speed. However, their velocities slow down as $t\to\infty$, which might be considered as an effect of dispersion at $+\infty$, but on the other hand, as $t\to-\infty$, they move to the left at asymptotically constant speed (one might think of this as the absence of dispersion on the left) and their speed is determined by the size of the spectral gap, $\alpha$. However, in both cases the distances between peaks grow logarithmically as $|t|\to\infty$. Our numerics confirm the long-time analysis (see Figure~\ref{fig:lague7} on page~\pageref{fig:lague7} depicting $u(x,t)$ with the parameters $\gamma=0$ and $\alpha=\nicefrac{1}{2}$ and~\cite{animations} for an animation).   
The analysis extends to a more general class of measures (see Theorem~\ref{th:LagPertrbd}). However, it seems to be an interesting and nontrivial problem to investigate the long-time behavior of solutions with initial data $(u_0,\mu_0)\in \CHdom^+$ having pure infinite multi-peakon form and such that $u_0-c\in H^1[0,\infty)$ with some constant $c>0$. Take into account that the Sobolev space $H^1$ is intimately related with the Camassa--Holm equation (see~\eqref{eqnCHHam}). Furthermore, it was shown recently in~\cite{ISPforCH} that this class is invariant under the flow~\eqref{eqnSMEvoIntro} and, moreover,~\cite{ISPforCH} contains a complete spectral characterization of this class.  

In the final Section~\ref{sec:JacobiPeaks}, we focus on the case when the spectral measure at the initial time is given by the Jacobi-type weight~\eqref{eq:JacobWeight} (notice that we can easily reduce a more general situation with weights given by~\eqref{eq:JacobWeightGen} to~\eqref{eq:JacobWeight} by employing very simple transformations described in Remark~\ref{rem:scaling}). Using the asymptotics of the Jacobi polynomials, one can show (see Lemma~\ref{lem:asympJacobinn2}) that the corresponding initial profile $u_0$ has a pure infinite multi-peakon form, which is asymptotically eventually periodic, that is, as $x\to \infty$, it is close to the profile 
\begin{align}
\tilde{u}_0(x) = \frac{1}{2}\sum_{n=1}^\infty a \E^{-|x-cn-\ell|}, 
\end{align}
with certain constants $a>0$, $c\in \R$, and the latter partially explains our interest in this example. Theorem~\ref{th:JacobiPeakons} shows that from the long-time peakon-wise perspective, the flow preserves the eventually periodic nature, that is, as $|t|\to\infty$ all peaks move in the same direction at an asymptotically constant speed. However, at the same time the distance between peaks grows as $\log(t^2)$. As in the previous cases, the analysis is based on the study of long-time behavior of the corresponding Hankel determinants. Our analysis is confirmed by numerical evidence; see Figure~\ref{fig:jacNew} on page~\pageref{fig:jacNew} with the snapshots of $u(x,t)$ and~\cite{animations} for an animation. 
As in the Laguerre case, the analysis can be extended to a wider class of measures, but it remains open how far it can be pushed.

Let us conclude this lengthy introduction with one more remark. One of the most prominent features of multi-peakons is the fact that many qualitative properties of solutions to the Camassa--Holm equation can already be observed for multi-peakons. Solutions constructed as superposition of an infinite number of solitons is the most natural next development (the idea to consider a closure of multi-soliton solutions in a suitable topology is not new and in the context of the Korteweg--de Vries equation it goes back at least to the work of V.\ A.\ Marchenko~\cite{mar91} in the 1990s, see also~\cite{gkz92}; let us also mention the growing interest in soliton gases~\cite{elka,zak}), however, in the context of the Camassa--Holm equation we are only aware of the work of L.-C.~Li~\cite{li09}. Furthermore, in~\cite{Eplusminus} the flow on the space $\CHdom$ was constructed as an extension by continuity (in a suitable sense) from conservative multi-peakons. Also, the space $\CHdom$ is sufficiently rich even despite a very restrictive assumption~\eqref{eqnMdef-}. More precisely, one can get arbitrary types of spectral behavior for both~\eqref{eq:Spec} and~\eqref{eqnISP} with the coefficients $\omega$ and $\dip$ from the spaces $\CHdom^+$ and, respectively, $\CHdom$. In this paper, we have touched only three particular situations, however, many more cases remain to be explored. It is conceivable that our study may inspire further research on ``peakon gases".

 \subsection*{Notation} 
 For an interval $I\subseteq\R$, we denote with $H^1_{\loc}(I)$, $H^1(I)$ and $H^1_{\cc}(I)$ the usual Sobolev spaces defined by 
\begin{align}
 H^1_{\loc}(I) & =  \lbrace f\in AC_{\loc}(I) \,|\, f'\in L^2_{\loc}(I) \rbrace, \\
 H^1(I) & = \lbrace f\in H^1_{\loc}(I) \,|\, f,\, f'\in L^2(I) \rbrace, \\ 
 H^1_{\cc}(I) & = \lbrace f\in H^1(I) \,|\, \supp(f) \text{ compact in } I \rbrace.
\end{align}
 The space of distributions $H^{-1}_{\loc}(I)$ is the topological dual space of $H^1_{\cc}(I)$. 
 A distribution in $H^{-1}_{\loc}(I)$ is said to be {\em real} if it is real for real-valued functions in $H^1_{\cc}(I)$. 

  For integrals of a function $f$ that is locally integrable with respect to a Borel measure $\mu$ on an interval $I$, we will employ the notation 
\begin{align}\label{eqnDefintmu}
 \int_x^y f\, d\mu = \begin{cases}
                                     \int_{[x,y)} f\, d\mu, & y>x, \\
                                     0,                                     & y=x, \\
                                     -\int_{[y,x)} f\, d\mu, & y< x, 
                                    \end{cases} 
\end{align}
 rendering the integral left-continuous as a function of $y$. 

For two functions $f$ and $g$, we write $f = \OO(g)$ if $f \le Cg$ for all sufficiently large values of the variables of the two functions, where $C>0$ is an absolute positive constant. If the limit of the ratio $f/g$ tends to zero as the variables of the functions tend to infinity, we write $f = \oo(g)$. Finally, $f \sim g$ denotes that $f = (1 + \oo(1))g$, that is, $f/g$ tends to $1$ when the variables tend to infinity.

Throughout the paper we shall use the following multi-index notation:
\begin{itemize}[leftmargin=*, widest=x]
\item For $\mlam = (\lambda_1,\dots,\lambda_n)$,  we shall denote 
$\rho(d\mlam) = \rho(d\lambda_1)\dots\rho(d\lambda_n)$. 
\item For a connected subset $\cI\subseteq\R$ we shall denote 
\begin{align}
\int_\cI\dots\int_\cI F(\lambda_1,\dots,\lambda_n)\rho(d\lambda_1)\dots\rho(d\lambda_n) = \int_{\cI^n}F(\mlam)\rho(d\mlam).
\end{align}
\item $\wand_n(\mlam)$ is the square of the Vandermonde determinant, that is, $\wand_1\equiv 1$ and $\wand_n(\mlam) = \prod_{1\leq i< j\leq n}|\lambda_i-\lambda_j|^2$ for $n\ge 2$. 
\item Similarly, for any index set $J = \{k_1,\dots, k_n\}\subset \N$ with $k_1<k_2<\dots<k_n$, $\mlam_J = (\lambda_{k_1},\dots,\lambda_{k_n})$ and $F(\mlam_J) = F(\lambda_{k_1},\dots,\lambda_{k_n})$.
\end{itemize}

\section{Preliminaries}\label{secPre}
 
  We are first going to consider the phase space $\CHdom$ for the two-component Camassa--Holm system~\eqref{eqnOurCH}--\eqref{eq:PequDef}, first introduced in~\cite{Eplusminus} as follows:
  
 \begin{definition}\label{defPS}
 The set $\CHdom$ consists of all pairs $(u,\mu)$ such that $u$ is a real-valued function in $H^1_{\loc}(\R)$ and $\mu$ is a positive Borel measure on $\R$ with
\begin{align}\label{eqnmuac}
   \int_B u(x)^2 + u'(x)^2\, dx \leq  \mu(B)
\end{align}
for every Borel set $B\subseteq\R$, satisfying the asymptotic growth restrictions  
\begin{align}
 \label{eqnMdef-}   \int_{-\infty}^0 \E^{-s} \bigl(u(s)^2 + u'(s)^2 \bigr) ds +  \int_{-\infty}^{0} \E^{-s} \dip(ds) & < \infty,  \\
 \label{eqnMdef+}   \limsup_{x\rightarrow\infty}\, \E^{x} \biggl(\int_{x}^{\infty}\E^{-s}(u(s) + u'(s))^2ds + \int_{x}^{\infty}\E^{-s}\dip(ds)\biggr) & < \infty,
\end{align}
where $\dip$ is the positive Borel measure on $\R$ defined such that 
\begin{align}\label{eqndipdef}
 \mu(B) = \dip(B) + \int_B u(x)^2 + u'(x)^2\, dx. 
\end{align}
\end{definition}
 
\begin{remark} 
A couple of remarks are in order:
\begin{enumerate}[label=(\alph*), ref=(\alph*), leftmargin=*, widest=e]
    \item
We prefer to work with pairs $(u,\mu)$ and the unusual condition~\eqref{eqnmuac} instead of the simpler definable pairs $(u,\dip)$ for various reasons. 
For example,  in the context of the conservative Camassa--Holm flow, the measure $\mu$ satisfies the transport equation in~\eqref{eqnOurCH} and represents the (local) energy of a solution. Moreover, the measure $\mu$ is more natural when considering suitable notions of convergence on $\CHdom$; see~\cite[Section~2]{Eplusminus} for details.
\item 
 Condition~\eqref{eqnMdef-} in this definition requires strong decay of both, the function $u$ and the measure $\dip$, at $-\infty$ (in particular, the condition on $u$ in~\eqref{eqnMdef-} is equivalent to the function $x\mapsto\E^{\nicefrac{-x}{2}}u(x)$ belonging to $H^1$ near $-\infty$), whereas condition~\eqref{eqnMdef+} on the growth near $+\infty$ is rather mild and satisfied as soon as $u+u'$ is bounded and $\dip$ is a finite measure for example. Despite its cumbersome looking form, after a simple change of variables, condition~\eqref{eqnMdef+} turns into a boundedness condition under the action of the classical Hardy operator (see~\cite{Eplusminus} for further details and also Remark~\ref{rem:Hardy} below). 
  \end{enumerate}
 \end{remark} 

In the following, we are going to introduce the spectral quantities that linearize the conservative Camassa--Holm flow on $\CHdom$. 
To this end, let us fix a pair $(u,\mu)$ in $\CHdom$ and define the real distribution $\omega$ in $H^{-1}_{\loc}(\R)$ by
\begin{align}\label{eqnDefomega}
 \omega(h) = \int_\R u(x)h(x)dx + \int_\R u'(x)h'(x)dx, 
\end{align}
so that one has $\omega = u - u''$ in a weak sense. 
We also mention that the positive Borel measure $\dip$ on $\R$ is defined by~\eqref{eqndipdef} and that it is always possible to uniquely recover the pair $(u,\mu)$ from the distribution $\omega$ and the measure $\dip$.

Associated with the pair $(u,\mu)$ is the ordinary differential equation 
\begin{align}\label{eqnDE}
 - f'' + \frac{1}{4} f = z\, \omega f + z^2 \dip f, 
\end{align}
where $z$ is a complex spectral parameter. 
Due to the low regularity of the coefficients, this differential equation has to be understood in a distributional sense in general; see \cite{ConservCH, IndefiniteString, gewe14, sash03}:
  A solution of~\eqref{eqnDE} is a function $f\in H^1_{\loc}(\R)$ such that 
 \begin{align}\label{eqnDEweakform}
   \int_{\R} f'(x) h'(x) dx + \frac{1}{4} \int_\R f(x)h(x)dx = z\, \omega(fh) + z^2 \int_\R f(x) h(x) \dip(dx) 
 \end{align} 
 for every function $h\in H^1_\cc(\R)$.
 We note that the derivative of such a solution $f$ is in general only defined almost everywhere. 
 However, there is always a unique left-continuous function $f^\qd$ on $\R$ such that 
 \begin{align}\label{eqnfqpm} 
     f^\qd = f' +\frac{1}{2} f - z (u +u') f 
\end{align} 
 almost everywhere on $\R$ (see \cite[Lemma~A.2]{ConservCH}), called the {\em quasi-derivative} of $f$. 

 The main consequence of the strong decay restriction on the pair $(u,\mu)$ at $-\infty$ in condition~\eqref{eqnMdef-} is the existence of a particular fundamental system $\phi(z,\redot)$, $\theta(z,\redot)$ of solutions to the differential equation~\eqref{eqnDE} with the asymptotics 
\begin{align}\label{eqnphiasym}
  \phi(z,x) & \sim \E^{\frac{x}{2}}, & \theta(z,x) & \sim \E^{-\frac{x}{2}},   \\
\label{eqnthetaasym}
  \phi^\qd(z,x) & \sim \E^{\frac{x}{2}}, &  \theta^\qd(z,x) & = \oo\bigl(\E^{\frac{x}{2}}\bigr),   
\end{align}
as $x\rightarrow-\infty$; see~\cite[Theorem~3.1]{Eplusminus}. 
With these solutions, we are able to define the {\em Weyl--Titchmarsh function} $m$ on $\C\backslash\R$ via 
\begin{align}\label{eq:m-funct-def}
  m(z) = -\lim_{x\rightarrow\infty} \frac{\theta(z,x)}{z\phi(z,x)},
\end{align}
which is a Herglotz--Nevanlinna function. 
Even more, according to~\cite[Proposition~3.3]{Eplusminus} it has a particular integral representation of the form 
\begin{align}\label{eqnWTmIntRepZero}
  m(z) = \int_\R \frac{z}{\lambda(\lambda-z)} \rho(d\lambda)  
\end{align}
for some positive Borel measure $\rho$ on $\R$ with  
\begin{align}\label{eq:Poisson}
 \int_\R \frac{\rho(d\lambda)}{1+\lambda^2} < \infty
\end{align}
and no mass in a small enough neighborhood of zero.
The measure $\rho$ can be seen to be a {\em spectral measure} for a self-adjoint linear relation associated with~\eqref{eqnDE}; we only refer to~\cite{CHPencil} for more details. 
One of the main properties of the {\em (direct) spectral transform} $(u,\mu)\mapsto\rho$ that will linearize the conservative Camassa--Holm flow on $\CHdom$ has been proved in~\cite[Theorem~4.1]{Eplusminus}:

\begin{theorem}\label{thm:ISP}
The mapping $(u,\mu)\mapsto\rho$ is a bijection between $\CHdom$ and the set $\SM_0$ of all positive Borel measures $\rho$ on $\R$ satisfying~\eqref{eq:Poisson} and whose topological support does not contain zero. 
\end{theorem}

\begin{remark}
In fact, upon equipping $\CHdom$ and $\SM_0$ with suitable topologies, the spectral transform in Theorem~\ref{thm:ISP} becomes homeomorphic; see~\cite[Proposition~4.5]{Eplusminus}.  
 \end{remark}

We shall also need the following important subclass of spectral measures: Denote by $\SM_0^+$ the set of all measures in $\SM_0$ whose support is contained in $(0,\infty)$. The analog of Theorem~\ref{thm:ISP} in this case was established in~\cite[Corollary~4.2]{Eplusminus}: 

 \begin{corollary}\label{corIP+}
 The mapping $(u,\mu)\mapsto \rho$ is a bijection between the set $\CHdom^+$ of all pairs $(u,\mu)$ in $\CHdom$ such that $\omega$ is a positive Borel measure on $\R$ and $\dip$ vanishes identically, and $\SM_0^+$. 
 \end{corollary} 

\begin{remark}\label{rem:IP+}
In fact, every pair $(u,\mu)$ in the set $\CHdom^+$ is uniquely determined by the function $u$ and hence $\CHdom^+$ can be identified with a subset of $H^1_{\loc}(\R)$. 
For further details we refer to~\cite[Lemma~4.3]{Eplusminus}. 
 \end{remark}

By means of the correspondence in Theorem~\ref{thm:ISP}, we next introduce a flow on $\CHdom$.

\begin{definition}\label{def:CHspecflow} 
The {\em conservative Camassa--Holm flow}  $\Phi$ on $\CHdom$ is a mapping
 \begin{align}
   \Phi\colon  \CHdom\times\R \rightarrow\CHdom
 \end{align}
defined as follows: Given a pair $(u,\mu)$ in $\CHdom$ with associated spectral measure $\rho$ and some $t\in\R$, the corresponding image $\Phi^t(u,\mu)$ under $\Phi$ is defined to be the unique pair in $\CHdom$ for which the associated spectral measure is given by    
   \begin{align}\label{eqnSMEvo}
    B \mapsto \int_B \E^{-\frac{t}{2\lambda}} \rho(d\lambda)
  \end{align}
  on the Borel subsets of $\R$. 
  \end{definition}
  
 We note that $\Phi^t(u,\mu)$ is well-defined since the measure given by~\eqref{eqnSMEvo} belongs to $\SM_0$ whenever so does $\rho$ and hence the existence of a unique corresponding pair in $\CHdom$ is guaranteed by Theorem~\ref{thm:ISP}. 
  The definition of this flow is of course motivated by the well-known time evolution of spectral data for spatially decaying classical solutions of the Camassa--Holm equation as well as multi-peakons; see~\cite[Section~6]{besasz98}.  
  
   \begin{definition}\label{def:ConsSolution}
A {\em global conservative solution} of the two-component Camassa--Holm system~\eqref{eqnOurCH} with initial data $(u_0,\mu_0)\in\CHdom$ is a continuous curve 
\begin{align}
  \gamma\colon t\mapsto(u(\ledot,t),\mu(\ledot,t))
\end{align}
from $\R$ to $\CHdom$ with $\gamma(0)=(u_0,\mu_0)$ that satisfies~\eqref{eqnOurCH} in the sense that for every test function $\varphi\in C_\cc^\infty(\R\times\R)$ one has  
 \begin{align}
 & \int_\R \int_\R u(x,t) \varphi_t(x,t) + \biggl(\frac{u(x,t)^2}{2} + P(x,t) \biggr) \varphi_x(x,t) \,dx \,dt = 0, \\
 \begin{split} 
 &  \int_\R \int_\R \varphi_t(x,t) + u(x,t) \varphi_x(x,t) \,\mu(dx,t) \,dt  \\ 
 &   \qquad\qquad\qquad\qquad = 2\int_\R \int_\R u(x,t)\biggl(\frac{u(x,t)^2}{2} - P(x,t) \biggr) \varphi_x(x,t) \,dx \,dt,
 \end{split}
 \end{align}
 where the function $P$ on $\R\times\R$ is given by 
 \begin{align}
  P(x,t) =  \frac{1}{4} \int_\R \E^{-|x-s|} u(s,t)^2 ds +  \frac{1}{4} \int_\R \E^{-|x-s|} \mu(ds,t).
 \end{align}
\end{definition}   
  
  In~\cite[Theorem~5.4]{Eplusminus}, it was shown that the conservative Camassa--Holm flow $\Phi$ indeed gives rise to these kinds of global conservative solutions.
  
   \begin{theorem}\label{thm:consCHweak}
  For every pair $(u_0,\mu_0)\in\CHdom$, the integral curve $t\mapsto \Phi^t(u_0,\mu_0)$ is a global conservative solution of the two-component Camassa--Holm system~\eqref{eqnOurCH} with initial data $(u_0,\mu_0)$.
 \end{theorem} 
 
\begin{remark}\label{rem:wellposed}
It is widely known that classical solutions to the Camassa--Holm equation~\eqref{eqnCH} may develop singularities in finite time. On the other hand, it was proved that~\eqref{eqnCH} possesses global weak solutions~\cite{xizh00}, which are not unique however and hence continuation of solutions after blow-up is a delicate matter. A particular kind of weak solutions are so-called {\em conservative} solutions, the notion of which was suggested independently in~\cite{brco07} and~\cite{hora07}. Solutions of this kind have been constructed by a generalized method of characteristics that relied on a transformation from Eulerian to Lagrangian coordinates and was accomplished for various classes of initial data in~\cite{brco07, grhora12, grhora12b, hora07, hora07c}. The question about uniqueness of conservative weak solutions to the Camassa--Holm equation and its two-component generalization is a subtle one. 
 Uniqueness of conservative weak solutions to the Camassa--Holm equation has been established in~\cite{bcz15} (see also~\cite{bre16}) under the assumption that the initial data $u_0$ belongs to $H^1(\R)$. 
 However, the notion of weak solution employed in~\cite{bcz15} is stronger than ours, so that this uniqueness result does not apply in our case. 
On the other hand, in combination with continuous dependence on the initial data established in~\cite[Proposition~5.2]{Eplusminus}, Theorem~\ref{thm:consCHweak} leads to a well-posedness result for the two-com\-po\-nent Camassa--Holm system~\eqref{eqnOurCH} on $\CHdom$ (for further details see~\cite[Remark~13.21]{ISPforCH}).  
 \end{remark}

Later on, we are also going to use that the measure $\rho$ associated with a pair $(u,\mu)$ in $\CHdom$ is the spectral measure of a corresponding generalized indefinite string.

\begin{definition}\label{def:GIS}
A {\em generalized indefinite string} is a triple $(L,\tilde{\omega},\tilde{\dip})$ such that $L\in(0,\infty]$, $\tilde{\omega}$ is a real distribution in $H^{-1}_{\loc}[0,L)$ and $\tilde{\dip}$ is a positive Borel measure on the interval $[0,L)$.
The unique function $\tilde{\Wr}$ in $L^2_{\loc}[0,L)$ such that 
 \begin{align}
  \tilde{\omega}(h) & = - \int_0^L \tilde{\Wr}(x)h'(x)dx 
 \end{align}
 for all functions $h\in H^1_{\cc}[0,L)$ is called the {\em normalized anti-derivative} of $\tilde{\omega}$. 
\end{definition}

Associated with such a generalized indefinite string $(L,\tilde{\omega},\tilde{\dip})$ is the ordinary differential equation  
  \begin{align}\label{eqnGISODE}
  -y'' = z\, \tilde{\omega}y + z^2 \tilde{\dip}y
 \end{align}
 on the interval $[0,L)$, where $z$ is a complex spectral parameter. 
 Spectral problems of this form go back at least to work of M.\ G.\ Krein and H.\ Langer from the 1970s on indefinite analogues of the classical moment problem~\cite{krla79,la76}.
 In the generality above, they were introduced in~\cite{IndefiniteString}, where it was proved that they serve as a canonical model for self-adjoint operators with simple spectrum. 

 In order to make the connection with~\eqref{eqnDE} more precise, with a pair $(u,\mu)$ in $\CHdom$ we shall associate the generalized indefinite string $(\infty,\tilde{\omega},\tilde{\dip})$, where $\tilde{\omega}$ is defined via its normalized anti-derivative $\tilde{\Wr}$ by    
 \begin{align}\label{eqnDefa}
     \tilde{\Wr}(x) & = -\frac{u(\log x) + u'(\log x)}{x}  
 \end{align}
 and the positive Borel measure $\tilde{\dip}$ on $(0,\infty)$ is defined by 
 \begin{align}\label{eqnDefbeta}
  \tilde{\dip}(B)   = \int_{\log(B)} \E^{- x} \dip(dx) 
 \end{align}
 for every Borel set $B\subseteq(0,\infty)$. 
 This is well-defined because the function $\tilde{\Wr}$ turns out to belong to $L^2[0,\infty)$ and the measure $\tilde{\dip}$ turns out to be finite so that it can be extended to a Borel measure on $[0,\infty)$ by defining that it has no point mass at zero; see~\cite[Lemma~2.1]{Eplusminus}.
 Furthermore, these coefficients satisfy the asymptotic condition
  \begin{align}\label{eqnMdefTilde}
  \limsup_{x\rightarrow\infty}\,  x \int_{x}^\infty  \wt{\Wr}(s)^2ds  + x \int_x^\infty\wt{\dip}(ds) < \infty.
\end{align}
 In fact, the properties above characterize all generalized indefinite strings that arise in this way (see~\cite[Lemma~2.2]{Eplusminus}) and the pair $(u,\mu)$ can be recovered from the generalized indefinite string $(\infty,\tilde{\omega},\tilde{\dip})$ via 
  \begin{align}\label{eqnuitofa}
   u(x) & = - \frac{1}{\E^{x}} \int_0^{\E^x}  \wt{\Wr}(s) s\, ds, & \dip(B) & =   \int_{\E^B} x\, \tilde{\dip}(dx), 
  \end{align} 
  and then using relation~\eqref{eqndipdef}.
 Now the measure $\rho$ introduced above is precisely the spectral measure of the generalized indefinite string $(\infty,\tilde{\omega},\tilde{\dip})$ in the sense of~\cite{IndefiniteString}. 

Particular kinds of pairs in $\CHdom$ correspond to multi-soliton solutions of the conservative Camassa--Holm flow:

\begin{definition}\label{def:MP}
We say that a pair $(u,\mu)$ in $\CHdom$ is a {\em multi-peakon profile} if it is of the form 
\begin{align}
  u(x) & = \frac{1}{2} \sum_{n=1}^N \omega_n \E^{-|x-x_n|}, & \dip & = \sum_{n=1}^N \dip_n \delta_{x_n}, 
\end{align}
for some $N\in\N\cup\{0\}$, increasing points $x_1,\ldots,x_N$ in $\R$, real weights $\omega_1,\ldots,\omega_N$ and non-negative weights $\dip_1,\ldots,\dip_N$. The collection of all multi-peakon profiles will be denoted by $\Peakons$. 
\end{definition}

If $(u,\mu)$ is a multi-peakon profile, then the distribution $\omega$ is a measure supported on a finite set and given by 
\begin{align}
    \omega = \sum_{n=1}^N \omega_n \delta_{x_n}.
\end{align} 
For the corresponding generalized indefinite string $(\infty,\tilde{\omega},\tilde{\dip})$, the distribution $\tilde{\omega}$ and the measure $\tilde{\dip}$ are supported on a finite set as well and we compute that  
 \begin{align}\label{eq:GISviaMP}
   \tilde{\omega} & = \omega_0\delta_0 + \sum_{n=1}^N \frac{\omega_n}{\E^{x_n}}\delta_{\E^{x_n}}, & \tilde{\dip} & = \sum_{n=1}^N \frac{\dip_n}{\E^{x_n}}\delta_{\E^{x_n}},
 \end{align}
 for some constant $\omega_0\in\R$, which can be determined explicitly in terms of the other constants.
 In fact, for almost all $x<x_1$ one has 
 \begin{align}
   \omega_0 = \tilde{\Wr}(\E^x) = -\frac{u(x) + u'(x)}{\E^x} = -\sum_{n=1}^N \frac{\omega_n}{\E^{x_n}}.
 \end{align}
 
\begin{remark}\label{rem:MP=rational} 
The importance of multi-peakon profiles among all pairs in the set $\CHdom$ stems from the fact that they correspond precisely to spectral measures $\rho$ in $\SM_0$ that are supported on a finite set (see Proposition~\ref{prop:MP} in the next section or~\cite{besasz00, ConservMP}). In particular, the corresponding Weyl--Titchmarsh functions are precisely the rational Herglotz--Nevanlinna functions that are analytic at zero and admit the integral representation~\eqref{eqnWTmIntRepZero}.  
 \end{remark}

 More generally, in this article we will be concerned with pairs $(u,\mu)$ in $\CHdom$ that are of multi-peakon form merely to the left of some point $\ell\in\R$, that is, such that 
 \begin{align}\label{eqnMPtox}
   \omega|_{(-\infty,\ell)} & = \sum_{n=1}^N \omega_n \delta_{x_n}, & \dip|_{(-\infty,\ell)} & = \sum_{n=1}^N \dip_n \delta_{x_n}, 
 \end{align}
 for some $N\in\N\cup\{0\}$, increasing points $x_1,\ldots,x_N$ in $(-\infty,\ell)$, real weights $\omega_1,\ldots,\omega_N$ and non-negative weights $\dip_1,\ldots,\dip_N$. 
 In this case, the distribution $\tilde{\omega}$ and the measure $\tilde{\dip}$ of the corresponding generalized indefinite string $(\infty,\tilde{\omega},\tilde{\dip})$ are supported on a finite set when restricted to $[0,\E^\ell)$. 
 More precisely, one has that 
 \begin{align}\label{eqnStrViaMP}
   \tilde{\omega}|_{[0,\E^\ell)} & = \omega_0\delta_0 + \sum_{n=1}^N \frac{\omega_n}{\E^{x_n}}\delta_{\E^{x_n}}, & \tilde{\dip}|_{[0,\E^\ell)} & = \sum_{n=1}^N \frac{\dip_n}{\E^{x_n}}\delta_{\E^{x_n}},
 \end{align}
 for some constant $\omega_0\in\R$. 
 The converse holds true as well; if the generalized indefinite string $(\infty,\tilde{\omega},\tilde{\dip})$ is of the form in~\eqref{eqnStrViaMP}, then the pair $(u,\mu)$ is of the form in~\eqref{eqnMPtox}. The significance of this class of pairs $(u,\mu)$ stems from the fact that it admits a complete spectral characterization and, moreover, this class is preserved under the conservative Camassa--Holm flow. We are going to substantiate all this in the following section.
 
\begin{remark}\label{rem:exmpls}
 Particular examples of pairs $(u,\mu)$ as above are {\em infinite multi-peakon profiles} of the form  
\begin{align}\label{eqnMPinf}
\omega & = \sum_{n=1}^\infty \omega_n \delta_{x_n}, & \dip & = \sum_{n=1}^\infty\dip_n\delta_{x_n},
\end{align}
for increasing points $x_1,x_2,\ldots$ in $\R$ with $x_n\rightarrow\infty$, real weights $\omega_1,\omega_2,\ldots$ and non-negative weights $\dip_1,\dip_2,\ldots$. 
Such a pair $(u,\mu)$ clearly satisfies~\eqref{eqnMPtox} for all points $\ell\in\R$ (where $N$ is allowed to vary with $\ell$).
These examples in particular include infinite multi-peakon profiles that are eventually periodic or almost periodic.

Another class of pairs of the form~\eqref{eqnMPinf}, however, with a bounded sequence $x_n$, was studied in~\cite{li09}. More specifically, the main object in~\cite{li09} are pairs $(u,\mu)$ such that $\dip$ vanishes identically and $u$ is given by 
\begin{align}
u(x) = \frac{1}{2}\sum_{n=1}^\infty \omega_n \E^{-|x-x_n|},
\end{align}
where the strictly increasing sequence of points $x_1,x_2,\dots$ accumulates at $\ell\in\R$ and positive weights $\omega_1,\omega_2,\ldots$ satisfy certain decay assumptions (for further details we refer to Section~\ref{sec:DiscrSpec}). 
 \end{remark} 

\begin{remark}\label{rem:Hardy}
If $(u,\mu)$ is an infinite multi-peakon profile, then the measure $\dip$, whose support is bounded from below and discrete, satisfies~\eqref{eqnMdef+} exactly when
\begin{align}\label{eq:HardyForV}
\sup_{n\in\N} \E^{x_n} \sum_{k\ge n}\frac{\dip_k}{\E^{x_k}} <\infty.
\end{align} 
For the function $u$ the corresponding condition in~\eqref{eqnMdef+} is much more involved since $\omega_n$ may take values of both signs.

It is not difficult to show that~\eqref{eq:HardyForV} is equivalent to 
\begin{align}\label{eq:HardyForValt}
\sup_{n\to \infty} \sum_{x_k \in [n,{n+1})}\dip_k <\infty.
\end{align}
However, it seems a rather difficult task to simplify the above characterization without further assumptions on the growth of the $x_n$-s. For instance, if $x_n = n$ for all $n\in\N$ (or, more generally, if $(x_n)_{n\in\N}$ is a restriction of a Delone set to $(\ell,\infty)$ for some $\ell\in\R$), then the above condition holds true if and only if $\sup_n \dip_n < \infty$. Clearly, the latter is necessary, however, it is in general far from being sufficient (for example, take $x_n = \log n$ for all $n\in\N$).
Let us also mention that one may relate~\eqref{eq:HardyForV} to the action of a weighted discrete Hardy operator. Namely, for a fixed strictly increasing sequence $X=(x_n)_{n\in\N}$ consider the operator $H_X$ acting on sequences $(f_n)_{n\in\N}$ by
\begin{align}
H_X f(n) = \E^{x_n} \sum_{k\ge n} \frac{f_k}{\E^{x_k}}.
\end{align}
From this perspective, condition~\eqref{eq:HardyForV} for the discrete measure $\dip$ is nothing but a description of the pre-image of $\ell^\infty(\N)$ under the action of the Hardy-type operator $H_X$ in the cone of all positive bounded sequences.  
\end{remark}
 

\section{Spectral measures with finite moments}\label{secMoment}

 As in the previous section, we suppose that $(u,\mu)$ is an arbitrary pair in $\CHdom$ and let $\rho \in \SM_0$ denote the corresponding spectral measure. 
 The distribution $\omega$ and the measure $\dip$ are defined in the same way as in Section~\ref{secPre} as well.
 We are next going to characterize (in terms of the coefficients $\omega$ and $\dip$) when the spectral measure $\rho$ has finite moments up to a certain order $2K$. 
  It will turn out that this means precisely that our pair $(u,\mu)$ is of multi-peakon form~\eqref{eqnMPtox} to the left of some point. 
 
 Before we are able to state our main results, it is necessary to introduce some more notation: 
 We will denote with $s_{-1},s_0,s_1,\ldots$ the moments of the measure $\rho$ as long as they exist, that is, we set 
 \begin{align}
   s_k & = \int_\R \lambda^k \rho(d\lambda). 
 \end{align}
  Furthermore, for $l\in\Z$ and $k\in \N\cup\{0\}$ we will use the Hankel determinants  
   \begin{align}  \label{eqnHankel1,2}
  \Delta_{l,k} & = \begin{vmatrix} s_l & s_{l+1} & \cdots & s_{l+k-1} \\ s_{l+1} & s_{l+2} & \cdots & s_{l+k} \\ \vdots & \vdots & \ddots & \vdots \\ s_{l+k-1} & s_{l+k} & \cdots & s_{l+2k-2} \end{vmatrix}
  \end{align}
 as far as they are well-defined (when $k$ is zero, these determinants should be interpreted as equal to one), where we set the $s_{k}$ equal to zero for $k<-1$. 
 For future reference, we also state the useful relation 
  \begin{align} 
   \label{eqnDeltaRel}   \Delta_{l+1,k}  \Delta_{l-1,k}  - \Delta_{l+1,k-1} \Delta_{l-1,k+1} & = \Delta_{l,k}^2,          
  \end{align} 
 which follows from Sylvester's determinant identity \cite{ba68, ga} and holds as long as the involved determinants are defined.
 
 \begin{remark}\label{remDeltas}
 Later on, we shall need the well-known multiple integral expression for the Hankel determinants (see~\cite{sze} for example) 
\begin{align}\label{eq:HankelviaIntegral}
\Delta_{l,k} = \frac{1}{k!}\int_{\R^k} \mlam^l \wand_k(\mlam)\rho(d\mlam), 
\end{align}
where $\mlam^l = \lambda_1^l\lambda_2^l\cdots\lambda_k^l$, $\rho(d\mlam) = \rho(d\lambda_1)\cdots \rho(d\lambda_k)$ and $\wand_k$ is the square of the Vandermonde determinant.  
  \end{remark}

Let us suppose for now that $\rho$ is supported on a finite set with $D$ points, so that all moments of the measure $\rho$ exist. 
 In this case, the determinants $\Delta_{0,0},\ldots,\Delta_{0,D}$ are positive, but $\Delta_{0,k}$ is zero when $k>D$.
 Similarly, one has that $\Delta_{1,k}$ is zero when $k>D$, but $\Delta_{1,D}$ is not zero in view of~\eqref{eqnDeltaRel} with $l=1$ and $k=D$ because the support of $\rho$ does not contain zero and hence $\Delta_{2,D}>0$. 
  Now let $N\in\N$ be the number of non-zero elements of the sequence $\Delta_{1,0},\ldots,\Delta_{1,D}$ and define the increasing function 
 \begin{align}\label{eqnkapparat}
   \kappa\colon\{1,\dots,N\}\rightarrow\{0,\ldots,D\}
 \end{align}
  such that $\Delta_{1,\kappa(1)},\ldots,\Delta_{1,\kappa(N)}$ enumerates all non-zero members of the sequence $\Delta_{1,0},\ldots,\Delta_{1,D}$.  
 Since this sequence does not have any consecutive zeros, we could alternatively also define $\kappa$ recursively via 
  \begin{align}\label{eqnkapparec}
     \kappa(1) & = 0, &   \kappa(n+1) & = \begin{cases} \kappa(n)+1, & \Delta_{1,\kappa(n)+1}\not=0, \\ \kappa(n)+2, & \Delta_{1,\kappa(n)+1}=0. \end{cases}
  \end{align} 
 Our first result is essentially contained in~\cite{ConservMP} and is based on the work of M.~G.~Krein and H.~Langer~\cite{krla79}. 

\begin{proposition}\label{prop:MP}
  The spectral measure $\rho$ is supported on a finite set if and only if the pair $(u,\mu)$ is a multi-peakon profile. 
  In this case, the pair $(u,\mu)$ has the form 
              \begin{align}\label{eqnGISfin}
         u(x) & = \frac{1}{2} \sum_{n=1}^{N-1} \omega_n \E^{-|x-x_n|}, & \dip & = \sum_{n=1}^{N-1} \dip_n \delta_{x_n}, 
      \end{align}
       where
       \begin{align}
       N & = \#\{k\in\{0,\ldots,D\}\,|\,\Delta_{1,k} \not= 0\}, & D & =\#\supp(\rho),
       \end{align}
        the increasing points $x_1,\ldots,x_{N-1}$ in $\R$, the real weights $\omega_1,\ldots,\omega_{N-1}$ and the non-negative weights $\dip_1,\ldots,\dip_{N-1}$ are given by 
      \begin{subequations}\label{eqnGISDel}
      \begin{align}
                x_n & = \log\frac{\Delta_{2,\kappa(n)}}{\Delta_{0,\kappa(n)+1}}, \label{eqnGISDelx} \\
        	  \omega_n & = \biggl(\frac{\Delta_{-1,\kappa(n)+1}}{\Delta_{1,\kappa(n)}} - \frac{\Delta_{-1,\kappa(n+1)+1}}{\Delta_{1,\kappa(n+1)}}\biggr)\frac{\Delta_{2,\kappa(n)}}{\Delta_{0,\kappa(n)+1}}, \label{eqnGISDelomega} \\
	 \dip_n & = \biggl(\frac{\Delta_{-2,\kappa(n)+2}}{\Delta_{0,\kappa(n)+1}} -  \frac{\Delta_{-2,\kappa(n+1)+1}}{\Delta_{0,\kappa(n+1)}}\biggr)\frac{\Delta_{2,\kappa(n)}}{\Delta_{0,\kappa(n)+1}}. \label{eqnGISDeldip}
      \end{align}
           \end{subequations}
\end{proposition}

\begin{proof}
 By employing the explicit transformation to generalized indefinite strings described in Section~\ref{secPre}, this can be inferred immediately from the corresponding result for generalized indefinite strings (see~\cite[Proposition~2.2]{StieltjesType}, \cite[Section~5.3]{IndMoment}). 
\end{proof}

\begin{remark}\label{rem:formulas3.9}
A few remarks are in order:
\begin{enumerate}[label=(\alph*), ref=(\alph*), leftmargin=*, widest=e]
\item\label{rem:3.3a}
The expressions for the weights in~\eqref{eqnGISDel} can also be put in different forms. 
For example, relation~\eqref{eqnDeltaRel} with $l=0$ shows that one has 
\begin{align}\label{eqnGISDelomegaB}
  \omega_n = & \frac{\Delta_{0,\kappa(n)+1}\Delta_{2,\kappa(n)}}{\Delta_{1,\kappa(n)}\Delta_{1,\kappa(n)+1}} \not= 0,  & \dip_n & = 0, 
\end{align}
 as long as $\Delta_{1,\kappa(n)+1}$ is not zero.
 On the other side, if $\Delta_{1,\kappa(n)+1}$ is zero, then relation~\eqref{eqnDeltaRel} with $l=-1$ allows us to write 
 \begin{align}
   \dip_n = \frac{\Delta_{-1,\kappa(n)+2}^2 \Delta_{2,\kappa(n)}^{\,}}{\Delta_{0,\kappa(n)+1}^2\Delta_{0,\kappa(n)+2}^{\,}} >0.
 \end{align}
 In particular, these considerations make it clear that $\dip_n+|\omega_n|>0$ and that the weight $\dip_n$ is not zero if and only if $\Delta_{1,\kappa(n)+1}$ vanishes.
\item \label{rem:3.3b}
If the spectral measure is supported on the positive half-line so that $\rho\in \SM_0^+$, then the function $\kappa$ defined above simplifies to
    \begin{align}\label{eq:kappan-positive}
      \kappa(n) = n-1
    \end{align}
    because vanishing of the Hankel determinants $\Delta_{1,k}$ is restricted and the expressions in~\eqref{eqnGISDel} turn into
      \begin{align}\label{eq:3.9positive}
                x_n & = \log\frac{\Delta_{2,n-1}}{\Delta_{0,n}}, &
        	  \omega_n & = \frac{\Delta_{0,n}\Delta_{2,n-1}}{\Delta_{1,n-1}\Delta_{1,n}}, &
	  \dip_n & = 0. 
      \end{align}
   These formulas are essentially (up to a slightly different choice of spectral measure) the same as the ones in~\cite[Theorem~5.6]{besasz00}. 
   \end{enumerate}
\end{remark}

Our main goal is to extend the above formulas to a much larger class of pairs in the set $\CHdom$. We begin with the following important observation.

\begin{theorem}\label{thmMomK}
 Suppose that the spectral measure $\rho$ is not supported on a finite set and let $K\in\N\cup\{0\}$. 
 Then the moments of $\rho$ exist up to order $2K$ if and only if there is an integer $N\in\N$, increasing points $x_1,\ldots,x_{N}$ in $\R$, real weights $\omega_1,\ldots,\omega_{N-1}$ and non-negative weights $\dip_1,\ldots,\dip_{N}$ with $\dip_n+|\omega_n|>0$ for all $n\in\{1,\ldots,N-1\}$ and 
    \begin{align}\label{eq:KviaN}
       N - 1 + \#\{n\in\{1,\ldots,N\}\,|\,\dip_n\not=0\} \geq K
    \end{align} 
    such that the pair $(u,\mu)$ has the form  
    \begin{align}\label{eqnstrK}
      \omega|_{(-\infty,x_N)} & = \sum_{n=1}^{N-1} \omega_n \delta_{x_n}, & \dip|_{(-\infty,x_N]} & =  \sum_{n=1}^{N} \dip_n \delta_{x_n}.
    \end{align}
\end{theorem}

 \begin{proof}
   This follows again from the corresponding result for generalized indefinite strings in~\cite[Theorem~2.4]{StieltjesType} together with the explicit transformation from generalized indefinite strings described in Section~\ref{secPre}, see~\eqref{eqnMPtox} and~\eqref{eqnStrViaMP}. 
 \end{proof}
 
  Even though the moments of the spectral measure $\rho$ do not uniquely determine the measure itself and hence the pair $(u,\mu)$ in this case, we are still able to provide explicit formulas for the multi-peakon part in terms of the Hankel determinants. 
  To this end, let us suppose that $\rho$ is not supported on a finite set and that its moments exist up to order $2K$ for some $K\in\N\cup\{0\}$. 
  Under these assumptions, the determinants $\Delta_{0,k}$ are well-defined and positive for all $k\in\{0,\ldots,K+1\}$, whereas the determinants $\Delta_{1,k}$ are well-defined for $k\in\{0,\ldots,K\}$. 
  Now let $N$ be the number of non-zero elements of the sequence $\Delta_{1,0},\ldots,\Delta_{1,K}$ and define the increasing function 
  \begin{align}
    \kappa\colon \{1,\ldots,N\}\rightarrow\{0,\ldots,K\}
  \end{align}
   such that $\Delta_{1,\kappa(1)},\ldots,\Delta_{1,\kappa(N)}$ enumerates all non-zero members of the sequence $\Delta_{1,0},\ldots,\Delta_{1,K}$. 
  Alternatively, the function $\kappa$ could again be defined via the recursion in~\eqref{eqnkapparec}.  

   \begin{corollary}\label{corStrKForm}
     If the spectral measure $\rho$ is not supported on a finite set and its moments exist up to order $2K$ for some $K\in\N\cup\{0\}$, then the pair $(u,\mu)$ has the form 
         \begin{align}\label{eqnomegadiptoxN}
      \omega|_{(-\infty,x_N)} & =  \sum_{n=1}^{N-1} \omega_n \delta_{x_n}, & \dip|_{(-\infty,x_N)} & =  \sum_{n=1}^{N-1} \dip_n \delta_{x_n},
    \end{align} 
   where the increasing points $x_1,\ldots,x_{N}$ in $\R$, the real weights $\omega_1,\ldots,\omega_{N-1}$ and the non-negative weights $\dip_1,\ldots,\dip_{N-1}$ are given by~\eqref{eqnGISDel}. 
    Furthermore, if the determinant $\Delta_{1,K}$ is zero, then $\kappa(N)=K-1$ and one has  
    \begin{align}
      \dip(\{x_N\}) = \frac{\Delta_{-1,\kappa(N)+2}^2 \Delta_{2,\kappa(N)}}{\Delta_{0,\kappa(N)+1}^2\Delta_{0,\kappa(N)+2}} >0. 
    \end{align}
   \end{corollary}

   \begin{proof}
    This follows again from the corresponding result for generalized indefinite strings in~\cite[Corollary~2.5]{StieltjesType}.
   \end{proof}
  
  \begin{remark}
 Taking $K=0$ in Theorem~\ref{thmMomK}, we conclude that the spectral measure has finite total mass if and only if there is $x_1\in\R$ such that $\omega|_{(-\infty,x_1)} = 0$  and $\dip|_{(-\infty,x_1)} = 0$. Moreover, the maximal such $x_1$ is explicitly given by
  \begin{align}\label{eq:x_1vias0}
  x_1 = -\log s_0. 
  \end{align} 
Applying Theorem~\ref{thmMomK} once again, however, this time with $K=1$, we conclude that the second moment $s_2$ is infinite if and only if $\dip(\{x_1\}) = 0$ and for every $\eps>0$ the support of either $\omega$ or $\dip$ in $(x_1,x_1+\eps)$ is infinite. 
\end{remark}

It is tempting to connect existence of the first moment $s_1$ with the inequality $\omega(\{x_1\})\neq 0$, however, this question is rather subtle and it does not seem plausible to answer it in the general setting. On the other hand, this can be done under the additional positivity assumption. 

\begin{lemma}\label{lem:s1}
Assume that the pair $(u,\mu)\in \CHdom_0^+$ is such that $\omega|_{(-\infty,x_1)} = 0$ for some $x_1\in \R$ and $\omega|_{[x_1,x_1+\eps)} \neq 0$ for all $\eps>0$. Then the corresponding spectral measure $\rho$ has finite first moment if and only if $\omega(\{x_1\})\neq 0$. Moreover, 
\begin{align}
\omega(\{x_1\}) = \frac{s_0}{s_1}.
\end{align}
\end{lemma}

\begin{proof}
Let us consider the corresponding generalized indefinite string $(\infty,\tilde{\omega},\tilde{\dip})$. Then, by Remark~\ref{rem:mcontfracApp}, its Weyl--Titchmarsh function admits the continued fraction expansion~\eqref{eq:mcontfracApp}, where $\omega_0 = -s_{-1}$, $\ell = \E^{x_1}$, and $m_\ell$ is the Weyl--Titchmarsh function of the truncated string $(\infty,\tilde{\omega}_\ell,\tilde{\dip}_\ell)$. However, our assumptions  imply that it is in fact a Krein string and hence $m_\ell$ is a Stieltjes function, that is, it admits the integral representation (see Remark~\ref{rem:KreinSting})
\begin{align*}
m_\ell(z) = c + \int_0^\infty \frac{\rho_\ell(d\lambda)}{\lambda-z},
\end{align*}
where $c$ is a non-negative constant and $\rho_\ell$ is a positive Borel measure on $[0,\infty)$ satisfying~\eqref{eq:mStieltjesRho}. It is well known that $c = \tilde{\omega}_\ell(\{0\})$ and
\begin{align*}
c = \lim_{y\uparrow \infty} m_\ell(\I y).
\end{align*}
However, using~\eqref{eq:mcontfracApp}, we get
\begin{align*}
z^2 \left(m(z) - \omega_0 + \frac{1}{\ell z}\right) = \frac{z}{\ell(1-\ell z m_\ell(z))}, 
\end{align*}
and the right-hand side is bounded as $z\to\I\infty$ only if $c>0$. In this case one has
\begin{align*}
 m(z) = \omega_0 - \frac{s_0}{z} - \frac{s_1}{z^2} + \oo\left(\frac{1}{|z|^2}\right),
 \end{align*}
$z\to\I\infty$, which proves the final equality after simple straightforward calculations.
\end{proof}

\begin{remark}
It is not difficult to extend the above result to a wider setting. For instance, the argument in the proof of Lemma~\ref{lem:s1} applies if the Weyl--Titchmarsh function $m_\ell$ of a truncated string is an $R_1$-function in terminology of~\cite{kakr74a}, that is, when the corresponding spectral measure $\rho_\ell$ satisfies
\begin{align}\label{eq:R1cond}
\int_\R \frac{\rho_\ell(d\lambda)}{1+|\lambda|}<\infty.
\end{align}
However, we are unaware of a characterization of generalized indefinite strings (or canonical systems) enjoying this condition. 
 On the other hand, \cite[Section~6]{AsymCS} contains a characterization of generalized indefinite strings such that $m(z) = c + \oo(1)$ with some $c\in \overline{\C_+}$ in any closed sector in $\C_+$. 

Perhaps, a much more transparent explanation why the problem of the finite first moment is subtle can be given by looking at the integral representation of Herglotz--Nevanlinna functions. Namely, if $m$ is the Weyl--Titchmarsh function of a generalized indefinite string $(L,\tilde{\omega},\tilde{\dip})$, then 
\begin{align}
m(z) = \tilde{\dip}(\{0\}) z + c_2 - \frac{1}{Lz} + \int_{\R\setminus\{0\}}\frac{1}{\lambda-z} - \frac{\lambda}{1+\lambda^2}\rho(d\lambda). 
\end{align}
It is tempting to connect $c_2$ with $\tilde{\omega}(\{0\})$, however, it is not true without further assumptions on the spectral measure $\rho$ (for example, \eqref{eq:R1cond} would be sufficient).
\end{remark}
  
  From the explicit expressions for the positions and weights in~\eqref{eqnGISDel}, one can derive the identities  
    \begin{align}
     \label{eqnomegasum} \int_{-\infty}^{x_n} \E^{-x}\omega(dx)  & = \frac{\Delta_{-1,1}}{\Delta_{1,0}} -\frac{\Delta_{-1,\kappa(n)+1}}{\Delta_{1,\kappa(n)}},  \\ 
     \label{eqndipsum}  \int_{-\infty}^{x_n} \E^{-x} (u(x) + u'(x))^2dx + \int_{-\infty}^{x_n} \E^{-x}\dip(dx) & = -\frac{\Delta_{-2,\kappa(n)+2}}{\Delta_{0,\kappa(n)+1}}.
    \end{align} 
  In fact, they also follow readily from the corresponding identities in~\cite[Corollary~5.6]{IndMoment} for generalized indefinite strings (see also~\cite[Corollary~3.4]{chang2014generalized}).   
Let us point out that knowledge of the distribution $\omega$ near $-\infty$ does not allow one to recover the function $u$ near $-\infty$. However, it turns out that the additional knowledge of the first negative moment $s_{-1}$ would suffice for this purpose.
    
\begin{corollary}\label{cor:UviaMoments} 
Under the assumptions of Corollary~\ref{corStrKForm}, the function $u$ is given by 
\begin{align}\label{eqnunearminusinf}
u(x) = \frac{s_{-1}}{2} \E^{x},
\end{align}
for all $x\in (-\infty,x_1]$ and then recursively by 
 \begin{align}\label{form_u_interval}
      u(x) = u(x_n)\cosh(x-x_n) + (u'(x_n-) - \omega_n)\sinh(x-x_n), 
    \end{align}  
for all $x\in [x_n,x_{n+1})$ with $n\in \{1,\dots,N-1\}$. Here the increasing points $x_1,\ldots,x_{N}$ in $\R$ and the real weights $\omega_1,\ldots,\omega_{N-1}$ are given by~\eqref{eqnGISDel}.    
\end{corollary}   
   
\begin{proof}
Taking into account the first formula in~\eqref{eqnuitofa}, it follows from~\eqref{eqnStrViaMP} that it suffices to know $\omega_0$ in order to recover $u$ from $\omega$ in the interval $(-\infty,x_N)$. Indeed, the normalized anti-derivative $\tilde{\Wr}$ is equal to $\omega_0$ almost everywhere on $[0,\E^{x_1})$, and hence we conclude from~\eqref{eqnuitofa} that  
   \begin{align*}
      u(x) = - \frac{1}{\E^{x}} \int_0^{\E^x} \omega_0 s\, ds = -\frac{\omega_0}{2} \E^{x}
    \end{align*}
    for all $x\leq x_1$ in this case. 
However, under the assumptions of Corollary~\ref{corStrKForm}, the weight $\omega_0$ in~\eqref{eqnStrViaMP} for the corresponding generalized indefinite string $(\infty,\tilde{\omega},\tilde{\dip})$ is given by 
    \begin{align*}
      \omega_0 = - \frac{\Delta_{-1,1}}{\Delta_{1,0}} = -s_{-1} = - \int_\R \frac{\rho(d\lambda)}{\lambda}
    \end{align*}
    in view of~\cite[Corollary~2.5]{StieltjesType}. 
It then remains to notice that in each interval $(x_n,x_{n+1})$ the function $u$ is a linear combination of hyperbolic functions and at every $x_n$ it is continuous with 
\begin{align*}
 u'(x_n+) & = u'(x_n-) - \omega_n. \qedhere
\end{align*}
\end{proof}   

\begin{remark}\label{rem:UatInfty}
The coefficients $u(x_n)$ and $u'(x_n-)$ can be obtained via the recursion 
     \begin{align}\label{eqnuu'rec}
   \begin{pmatrix} u(x_{n+1}) \\ u'(x_{n+1}-) \end{pmatrix} =   \begin{pmatrix} C_n & S_n \\ S_n & C_n \end{pmatrix}  \begin{pmatrix} u(x_{n}) \\ u'(x_{n}-) \end{pmatrix} - \omega_n \begin{pmatrix} S_n \\ C_n \end{pmatrix}
 \end{align}    
    with the constants $C_n$ and $S_n$ given by 
 \begin{align}
   C_n & = \cosh(x_{n+1}-x_n)= \frac{1}{2}\biggl( \frac{\Delta_{2,\kappa(n+1)}\Delta_{0,\kappa(n)+1}}{\Delta_{0,\kappa(n+1)+1}\Delta_{2,\kappa(n)}} + \frac{\Delta_{0,\kappa(n+1)+1}\Delta_{2,\kappa(n)}}{\Delta_{2,\kappa(n+1)}\Delta_{0,\kappa(n)+1}}\biggr), \\
   S_n & = \sinh(x_{n+1}-x_n)= \frac{1}{2}\biggl( \frac{\Delta_{2,\kappa(n+1)}\Delta_{0,\kappa(n)+1}}{\Delta_{0,\kappa(n+1)+1}\Delta_{2,\kappa(n)}} - \frac{\Delta_{0,\kappa(n+1)+1}\Delta_{2,\kappa(n)}}{\Delta_{2,\kappa(n+1)}\Delta_{0,\kappa(n)+1}}\biggr). \label{form_S}
 \end{align}      
 For later use, we also note that the values $u(x_n)$ satisfy the simpler recursion 
 \begin{align}\label{eqnuxnRecursion}
   u(x_{n+1}) & = u(x_n) \frac{\Delta_{0,\kappa(n+1)+1}\Delta_{2,\kappa(n)}}{\Delta_{2,\kappa(n+1)}\Delta_{0,\kappa(n)+1}} + S_n\frac{\Delta_{2,\kappa(n)}\Delta_{-1,\kappa(n+1)+1}}{\Delta_{0,\kappa(n)+1}\Delta_{1,\kappa(n+1)}},
 \end{align}
 which follows from the recursion in~\eqref{eqnuu'rec} after taking into account the relation 
 \begin{align}\label{eqnu+u'Delta}
   u(x_n) + u'(x_n-) = -\E^{x_n}\tilde{\Wr}(\E^{x_n}-) = \frac{\Delta_{2,\kappa(n)}\Delta_{-1,\kappa(n)+1}}{\Delta_{0,\kappa(n)+1}\Delta_{1,\kappa(n)}},
 \end{align}
 where we used~\eqref{eqnGISDelx} and the first identity in~\cite[Remark~2.3]{StieltjesType}.
  \end{remark}

 We are now going to turn to the situation when the spectral measure $\rho$ is not supported on a finite set and all its moments exist. 
  In this case, all the Hankel determinants are well-defined and the function $\kappa$ extends to an increasing function $\kappa\colon\N\rightarrow\N\cup\{0\}$ such that $\Delta_{1,\kappa(1)},\Delta_{1,\kappa(2)},\ldots$ enumerates all non-zero members of the sequence $\Delta_{1,0},\Delta_{1,1},\ldots$ (of which there are infinitely many as the latter sequence does not contain any consecutive zeros).

    \begin{corollary}\label{corStrinf}
        Suppose that the spectral measure $\rho$ is not supported on a finite set. 
    Then all moments of $\rho$ exist if and only if there are increasing points $x_1,x_2,\ldots$ in $\R$, real weights $\omega_1,\omega_2,\ldots$ and non-negative weights $\dip_1,\dip_2,\ldots$ with $\dip_n+|\omega_n|>0$ for all $n\in\N$ such that the pair $(u,\mu)$ has the form 
    \begin{align}\label{eqnstrinf}
      \omega|_{(-\infty,L)} & =  \sum_{n=1}^{\infty} \omega_n \delta_{x_n}, & \dip|_{(-\infty,L)} & =  \sum_{n=1}^{\infty} \dip_n \delta_{x_n},  & L & = \lim_{n\rightarrow\infty} x_n. 
    \end{align}
       In this case, the increasing points $x_1,x_2,\ldots$ in $\R$, the real weights $\omega_1,\omega_2,\ldots$ and the non-negative weights $\dip_1,\dip_2,\ldots$ are given by~\eqref{eqnGISDel}.
  \end{corollary}

  \begin{proof}
   This follows immediately from Theorem~\ref{thmMomK} and Corollary~\ref{corStrKForm}. 
  \end{proof} 
  
  \begin{remark}
Let us mention that the function $u$ can be recovered on the interval $(-\infty,L)$ by following the procedure described in Corollary~\ref{cor:UviaMoments} and Remark~\ref{rem:UatInfty}.  
  \end{remark}
  
  \begin{remark}\label{rem:scaling}
 For future use in the following sections, let us state a few simple transformation rules for the correspondence $(u,\mu)\mapsto\rho$.
\begin{enumerate}[label=(\alph*), ref=(\alph*), leftmargin=*, widest=e]
\item For a scaled spectral measure $\tilde{\rho}$ given by 
\begin{align}
\tilde{\rho}(B) = c\rho(B)
\end{align}
for all Borel subsets of $\R$, where $c>0$ is a positive constant, the corresponding moments and Hankel determinants are clearly connected by 
\begin{align}
s_k(\tilde{\rho}) & = c s_k(\rho), & \Delta_{l,k}(\tilde{\rho}) = c^k \Delta_{l,k}(\rho).
\end{align}
Therefore, the formulas in~\eqref{eqnGISDel} imply that the corresponding multi-peakon profiles are related by 
\begin{align}
x_k(\tilde{\rho}) & = x_k(\rho) - \log c, & \omega_{k}(\tilde{\rho}) & = \omega_{k}(\rho), & \dip_{k}(\tilde{\rho}) & =\dip_k(\rho),
\end{align}
that is, the scaling only leads to a shift in the spatial variable.
\item When the spectral parameter is scaled instead, that is, if for some positive constant $c>0$ one has 
\begin{align}
\int_{(a,b)}\tilde\rho(d\lambda) = \int_{(ca,cb)} \rho(d\lambda)
\end{align}
for all $(a,b)\subseteq \R$, then  the corresponding moments and Hankel determinants are clearly connected by 
\begin{align}
s_k(\tilde{\rho}) & = c^{-k} s_k(\rho), & \Delta_{l,k}(\tilde{\rho}) = c^{-k(k+l-1)} \Delta_{l,k}(\rho).
\end{align}
Using the formulas in~\eqref{eqnGISDel} once again, we conclude that the corresponding multi-peakon profiles are related by
\begin{align}
x_k(\tilde{\rho}) & = x_k(\rho), & \omega_{k}(\tilde{\rho}) & = c\omega_{k}(\rho), & \dip_{k}(\tilde{\rho}) & =c^2\dip_k(\rho),
\end{align}
that is, this only leads to a scaling of the masses. 
\end{enumerate}
\end{remark}
  
\begin{remark}\label{rem:OPs}
   If all moments of the spectral measure $\rho$ exist, then the formulas in~\eqref{eqnGISDel} can also be expressed in terms of the corresponding orthogonal polynomials. 
More specifically (see~\cite{akh} for example), the orthogonal polynomials $P_k$ of the first kind are given by
   \begin{align}\label{eq:Pol1stkind}
    P_k(z) & = \frac{1}{\sqrt{\Delta_{0,k}\Delta_{0,k+1}}}\begin{vmatrix} 
s_0 & s_1 & \dots &s_{k} \\ 
s_1 & s_2 & \dots & s_{k+1} \\ 
\vdots & \vdots & \ddots & \vdots \\ 
s_{k-1} & s_k & \dots & s_{2k-1} \\
1 & z &  \dots & z^k \end{vmatrix}, & k & \in \N\cup\{0\}.  
   \end{align}
They satisfy the recurrence relations
\begin{align}
zP_k(z) = b_k P_{k+1}(z) + a_k P_k(z) + b_{k-1}P_{k-1}(z),
\end{align}
for all $k\in \N$, where the coefficients $(a_k)$ and $(b_k)$ are usually referred to as {\em Jacobi parameters} and which can be expressed by means of certain Hankel determinants. 
Evaluating~\eqref{eq:Pol1stkind} at zero, 
we see that~\eqref{eqnGISDelx} turns into\footnote{The remaining formulas in~\eqref{eqnGISDel} require polynomials of the second kind, which is why we omit them here.}   
      \begin{align}
                x_n & = \log\Biggl(\sum_{k=0}^{\kappa(n)} P_k(0)^2 \Biggr).
      \end{align}       
The latter establishes a close relationship with the {\em Christoffel--Darboux kernel},    
a major object in the theory of orthogonal polynomials 
 (see~\cite{sim09} for example). Its diagonal is given by
\begin{align}\label{eq:CDkernel}
K_n(z,z) = \sum_{k=0}^n P_k(z)^2 = -\frac{1}{\Delta_{0,n+1}}\begin{vmatrix} 0 & 1 & z & \cdots & z^n \\ 
1& s_0 & s_1 & \cdots & s_{n} \\ z& s_1 & s_2 & \cdots & s_{n+1} \\ \vdots & \vdots& \vdots & \ddots & \vdots \\ z^n & s_{n} & s_{n+1} & \cdots & s_{2n} \end{vmatrix}.
\end{align}
We also recall the (confluent) Christoffel--Darboux formula
\begin{align}\label{eq:CDflaGen}
 K_n(z,z) = b_{n} (P_{n+1}'(z)P_{n}(z) - P_{n+1}(z)P_{n}'(z)).
\end{align}
In particular, if $\Delta_{1,n} \neq 0$ for all $n\in \N$ (the latter holds when $\rho \in \SM_0^+$ for example), then we get the formulas 
\begin{subequations}\label{eqnCFOP}
           \begin{align}
\omega_n &   = \frac{P_{n-1}'(0)}{P_{n-1}(0)} - \frac{P_{n}'(0)}{P_{n}(0)} ,\label{eq:omegaviaCD}\\ 
 \dip_n & = 0,\\[1mm]
\E^{x_n} &   = b_{n-1}(P_n'(0)P_{n-1}(0) - P_n(0)P_{n-1}'(0)). \label{eq:XnviaCD}
\end{align}
\end{subequations}
 \end{remark}
 
  Under the conditions in Corollary~\ref{corStrinf}, the limit $L$ can be clearly written as  
  \begin{align}\label{eqnLviaDelta}
    L = \lim_{n\rightarrow\infty} x_n = \lim_{n\rightarrow\infty} \log\frac{\Delta_{2,\kappa(n)}}{\Delta_{0,\kappa(n)+1}}, 
  \end{align}    
  providing an explicit relation between $L$ and the spectral data in form of the Hankel determinants. 
  Thus, by adding the condition 
     \begin{align}\label{eqnCondStrL}
               \lim_{k\rightarrow\infty} \frac{\Delta_{0,k+1}}{\Delta_{2,k}} =0
     \end{align}    
  on the spectral side in Corollary~\ref{corStrinf}, we can make sure that the limit $L$ is always infinite and that~\eqref{eqnstrinf} holds unrestricted. 
  
  \begin{remark}
  Notice that the sequence in~\eqref{eqnCondStrL} converges to zero if and only if so does its subsequence indexed by $\kappa(n)$. Indeed, the function $\kappa$ enumerates all non-zero elements in the sequence of Hankel determinants $\Delta_{1,k}$ and, by~\eqref{eqnkapparec}, $\kappa(n)\neq \kappa(n+1)-1$ exactly when $\Delta_{1,\kappa(n)+1} = 0$. However, by~\eqref{eqnDeltaRel} we then get 
\begin{align}
    \frac{\Delta_{0,\kappa(n)+1}}{\Delta_{2,\kappa(n)}} = \frac{\Delta_{0,\kappa(n)+2}}{ \Delta_{2,\kappa(n)+1}},
\end{align}    
and it only remains to notice that in this case $\kappa(n+1) = \kappa(n) + 2$.      
  \end{remark}
  
  Instead of condition~\eqref{eqnCondStrL}, one can alternatively also require determinacy of the associated Hamburger moment problem as our next result shows.

   \begin{corollary}\label{corStrinfpure}
        Suppose that the spectral measure $\rho$ is not supported on a finite set. 
    Then all moments of $\rho$ exist and the corresponding Hamburger moment problem is determinate if and only if the pair $(u,\mu)$ is an infinite multi-peakon profile, that is, there are increasing points $x_1,x_2,\ldots$ in $\R$ with $x_n\rightarrow\infty$, real weights $\omega_1,\omega_2,\ldots$ and non-negative weights $\dip_1,\dip_2,\ldots$ with $\dip_n+|\omega_n|>0$ for all $n\in\N$ such that the pair $(u,\mu)$ has the form  
    \begin{align}\label{eqnstrinfpure}
      \omega & =  \sum_{n=1}^{\infty} \omega_n \delta_{x_n}, & \dip & =  \sum_{n=1}^{\infty} \dip_n \delta_{x_n}. 
    \end{align}
       In this case, the increasing points $x_1,x_2,\ldots$ in $\R$, the real weights $\omega_1,\omega_2,\ldots$ and the non-negative weights $\dip_1,\dip_2,\ldots$ are given by~\eqref{eqnGISDel}.  
  \end{corollary}

  \begin{proof}
     This follows from Corollary~\ref{corStrinf} and~\cite[Theorem~5.7]{IndMoment} applied to the corresponding generalized indefinite string $(\infty,\tilde{\omega},\tilde{\dip})$. 
  \end{proof}
  
  \begin{remark}
It goes back to the work of H.\ L.\ Hamburger (see~\cite[Addenda and Problems~II.8]{akh} for example) that the moment problem is determinate if and only if 
 \begin{align}\label{eqnHambStrL}
               \lim_{k\rightarrow\infty} \frac{\Delta_{0,k+1}}{\Delta_{4,k-1}} = 0.
     \end{align} 
   Let us also mention that it was shown in~\cite{bci02} that the moment problem is determinate if and only if the smallest eigenvalue of the Hankel matrix $(s_{i+j})_{0\le i,j\le k}$ tends to zero as $k\to \infty$. 
\end{remark}  
  
  In fact, the knowledge of moments allows to recover an infinite multi-peakon profile completely. To simplify considerations, we state the next result under the additional positivity assumption.
  
  \begin{corollary}\label{corStrinfpure+}
     Suppose that the spectral measure $\rho\in\SM_0^+$ is not supported on a finite set. Then all moments of $\rho$ exist and the corresponding Hamburger moment problem is determinate if and only if there are increasing points $x_1,x_2,\ldots$ in $\R$ with $x_n\rightarrow\infty$, positive weights $\omega_1,\omega_2,\ldots$ such that the pair $(u,\mu)\in \CHdom^+$ has the form 
      \begin{align}\label{eqnstrinfpureU}
      u(x) & =  \frac{1}{2}\sum_{n=1}^{\infty} \omega_n \E^{-|x-x_n|}, & \dip & \equiv 0.
       \end{align}
In this case, the coefficients of $u$ are given by~\eqref{eq:3.9positive} and the series converges locally uniformly on $\R$.  
  \end{corollary}   
  
  \begin{proof}
Taking into account Corollary~\ref{corIP+}, we only need to show that $u$ indeed has the desired form. Since $\rho\in\SM_0^+$, condition~\eqref{eqnMdef+} is equivalent to $u+u'\in L^\infty(\R)$,  see~\cite[Remark~3.8]{Eplusminus}. However, the latter means that 
 $\omega$ satisfies~\eqref{eq:HardyForV}, that is,
\begin{align*}
 \sum_{k\ge n}\frac{\omega_k}{\E^{x_k}} \le C \E^{-x_n},
\end{align*}
for all $n\in\N$, where $C>0$ is a positive constant independent of $n\in\N$. Therefore, 
\begin{align*}
\sum_{n=1}^{\infty} \omega_n \E^{-|x-x_n|} & 
       = \sum_{x_n<x} \omega_n \E^{x_n - x} +\sum_{x_n\ge x} \omega_n \E^{x-x_n} \\
       & = \E^{-x}\sum_{x_n<x} \omega_n \E^{x_n} + \E^x\sum_{x_n\ge x} \frac{\omega_n}{ \E^{x_n}}\\
       &\le C+\sum_{x_n<x} \omega_n \\
       &\le C+C_1 |x|,
\end{align*}
where the last estimate follows from~\eqref{eq:HardyForValt}. This immediately implies that the series in~\eqref{eqnstrinfpureU} converges locally uniformly on $\R$. Moreover, it is clear that $u$ given by~\eqref{eqnstrinfpureU} is a distributional solution to $u-u'' = \omega$. However, by~\cite[Lemma~2.9]{Eplusminus}, this equation admits a unique solution satisfying $u(x) = \oo(\E^{|x|})$ as $|x|\to\infty$. 
  \end{proof} 

The classical orthogonal polynomials (Chebyshev--Hermite, Jacobi, Laguerre--Sonin) are associated with determinate moment problems, whereas 
indeterminate moment problems are present within the $q$-analog of the Askey-scheme (see, e.g., \cite{bech20} for a detailed account and further references). Historically, the first example of an indeterminate (Stieltjes) moment problem is due to T.\ J.\ Stieltjes.  We are going to discuss this example next.

\begin{example}\label{rem:SW}
The (shifted) Stieltjes--Wigert weight, which is also called ``log-normal" for obvious reasons,
\begin{align}\label{eq:SWweight}
\rho(d\lambda) =\frac{\kappa}{\sqrt{\pi}}\id_{(\alpha,\infty)}(\lambda)\E^{-\kappa^2\log^2 (\lambda-\alpha)}d\lambda,
\end{align}
where $\alpha\ge 0$ and $\kappa>0$, is the classical and, apparently, historically the first example of a measure that leads to an indeterminate moment problem (see~\cite[Chapter~II.2.7]{sze} and also \cite{chr03,ww06}).
 If $\alpha=0$, then the corresponding moments are 
\begin{align}\label{eq:SWmoments}
s_n & = \E^{\frac{(n+1)^2}{4\kappa^2}} = q^{-\frac{(n+1)^2}{2}}, & q & := \E^{\nicefrac{-1}{2\kappa^2}},
\end{align}
for all $n\in \N\cup\{0\}$.
The polynomials of the first kind (when $\alpha=0$) are given by
\begin{align}\label{eq:SWpolynom}
P_n(z) = (-1)^n \frac{q^{\frac{2n+1}{4}}}{\sqrt{(q;q)_{n}}} \sum_{j=0}^n \qbinom{n}{j}q^{j^2}(-\sqrt{q}z)^j, 
\end{align}
where the $q$-factorial and $q$-binomial coefficient are defined by
\begin{align}
(a;q)_n & 
 = \prod_{j=1}^n(1-aq^{j-1}), & \qbinom{n}{j} &  = \frac{(q;q)_n}{(q;q)_j(q;q)_{n-j}} .
\end{align}
Let $(u,\mu)$ be the pair in $\CHdom^+$ corresponding to the measure~\eqref{eq:SWweight} with $\alpha>0$. Taking into account the corresponding Jacobi parameters 
\begin{align}\label{eq:SWjacobi}
a_n & = \alpha + q^{-\frac{4n+3}{2}} + q^{-\frac{4n+1}{2}} - q^{-\frac{2n+1}{2}}, & b_{n-1} & = q^{-2n}\sqrt{1-q^{n}},
\end{align}
and using~\eqref{eq:omegaviaCD} and~\eqref{eq:XnviaCD} below, one can express the heights and positions of the peakons to the left of $L$ by means of the Stieltjes--Wigert polynomials and then 
may try to employ the asymptotic expansion of the latter as $n\to \infty$, which was obtained rather recently in~\cite[Theorem~1]{ww06} (see also~\cite{won18}), in order to deduce the asymptotic behavior of the corresponding heights and positions.

Let us emphasize that the weight~\eqref{eq:SWweight} satisfies the assumptions of \cite[Theorem~13.5]{ISPforCH} 
 and hence $u- (4\alpha)^{-1}\in H^1(0,\infty)$. Even more, in view of \cite[Theorem~13.9]{ISPforCH}, the momentum $\omega = u-u_{xx}$ has a nontrivial absolutely continuous part near infinity:  $\varrho - (2\sqrt{\alpha})^{-1} \in L^2(0,\infty)$, where $\varrho^2$ is the Radon--Nikod\'ym derivative of $\omega$ with respect to the Lebesgue measure. However, the refined structure of $u$ on the interval $[L,\infty)$ (e.g., whether $u$ contains any peakons on $[L,\infty)$?) together with the exact value of $L$ remains unclear to us.

In conclusion of this example, let us also mention the following observation of T.\ S.\ Chihara~\cite{chi70}. For every positive $c>0$ and non-negative $\alpha\ge 0$, the measure 
\begin{align}\label{eq:SWchihara}
\rho_c = \frac{1}{\sqrt{q}M(c)}\sum_{n\in \Z} c^nq^{\frac{n^2+2n}{2}} \delta_{cq^n+\alpha},
\end{align}
upon choosing a suitable norming constant $M(c)>0$, also has moments~\eqref{eq:SWmoments} when $\alpha=0$. This, in particular, implies that $\rho_c$ has the same moments as $\rho$ given by~\eqref{eq:SWweight} for all $\alpha>0$ and therefore the corresponding pair $(u_c,\mu_c)\in\CHdom^+$ coincides with $(u,\mu)$ on $(-\infty,L)$.  However, $\rho_c$ is pure point with support accumulating only at $\alpha$ and at $\infty$. Again, the refined structure of the corresponding pair $(u,\mu)\in\CHdom^+$ on $(L,\infty)$ remains unclear to us. We can only claim that $\omega$ has infinite support in $[L,\infty)$ in this case and that it is not of the multi-peakon from the right of any $\ell>L$ (the latter follows from the remark below).
 \end{example} 
 
  \begin{remark}\label{rem:InftyMinToPlus}
  We finish this section with a couple of remarks. First of all, notice that the conditions at $-\infty$ and $+\infty$ in Definition~\ref{defPS} could be switched due to the symmetry
\begin{align}
  (x,t) \mapsto (-x,-t)
\end{align}
of the two-component Camassa--Holm system~\eqref{eqnOurCH}--\eqref{eq:PequDef}. In particular, all the results of this section can easily be extended, upon appropriate modifications, to the pairs $(u,\mu)$ which are of multi-peakon form from the right of some point $\ell\in\R$, that is, such that 
    \begin{align}
   \omega|_{(\ell,\infty)} & = \sum_{n=1}^N \omega_n \delta_{x_n}, & \dip|_{(\ell,\infty)} & = \sum_{n=1}^N \dip_n \delta_{x_n}, 
 \end{align}
  for some $N\in\N\cup\{0\}$, decreasing points $x_1,\ldots,x_N$ in $(\ell,\infty)$, real weights $\omega_1,\ldots,\omega_N$ and non-negative weights $\dip_1,\ldots,\dip_N$. 

However, if $(u,\mu)$ is of multi-peakon form from both sides, then it satisfies the strong decay restriction~\eqref{eqnMdef-} at $-\infty$ and also an analogous decay condition at $+\infty$. This in particular entails that the spectral measure $\rho$ is supported on a discrete set $\sigma\subset \R$. Notice that this class contains pairs $(u,\mu)$ in $\CHdom$ that are of the form
       \begin{align}\label{eqnstrinf2side}
      \omega|_{\R\backslash\{L\}} & =  \sum_{n=1}^{\infty} \omega_n \delta_{x_n}, & \dip|_{\R\backslash\{L\}} & =  \sum_{n=1}^{\infty} \dip_n \delta_{x_n},
    \end{align}
     for some points $x_1,x_2,\ldots$ in $\R\backslash\{L\}$ that only accumulate at $L$, real weights $\omega_1,\omega_2,\ldots$ and non-negative weights $\dip_1,\dip_2,\ldots$, where $L$ is a given point in $\R$.      As will be seen, a subclass of this type of profiles can be associated with $n$-canonical solutions of indeterminate moment problems.
     We postpone this discussion to Section~\ref{sec:DiscrSpec}.   
 \end{remark}
 
\section{Global conservative solutions and the moment problem}\label{sec:conssolMP}
 
Now we are in position to discuss our results' relevance for the conservative Camassa--Holm flow $\Phi$ on $\CHdom$. 
Since it is clear that the existence of moments of the spectral measure $\rho$ is invariant under the flow defined by~\eqref{eqnSMEvo}, the corresponding conditions on the pairs $(u,\mu)$ in Theorem~\ref{thmMomK} and Corollary~\ref{corStrinf} are preserved as well.
In this case, the multi-peakon part of the solution 
\begin{align}
  (u(\ledot,t),\mu(\ledot,t)) = \Phi^t(u_0,\mu_0)
\end{align}
near $-\infty$ can be written down explicitly in terms of the spectral measure $\rho_0$ at initial time zero. 

\begin{theorem}\label{thm:MPevolt}
    Suppose that the spectral measure $\rho_0$ at initial time zero is not supported on a finite set and its moments exist up to order $2K$ for some $K\in\N\cup\{0\}$. 
  Then for each time $t\in\R$, the solution $(u(\ledot,t),\mu(\ledot,t))$ has the form~\eqref{eqnomegadiptoxN}, where the coefficients are given by~\eqref{eqnGISDel}, computed using the time evolved moments 
\begin{align}\label{eq:momentInT}
  s_k(t) = \int_\R \lambda^k \E^{-\frac{t}{2\lambda}} \rho_0(d\lambda).
\end{align}
If all moments of $\rho_0$ exist and the corresponding Hamburger moment problem is determinate, then for each time $t\in\R$, the solution $(u(\ledot,t),\mu(\ledot,t))$ has the form
    \begin{align}\label{eqnPurePeakon}
      \omega(\ledot,t) & =  \sum_{n=1}^{\infty} \omega_n(t) \delta_{x_n(t)}, & \dip(\ledot,t) & = \sum_{n=1}^{\infty} \dip_n(t) \delta_{x_n(t)},
    \end{align}
  where the increasing points $x_1(t),x_2(t),\ldots$ in $\R$ with $x_n(t)\rightarrow\infty$, the real weights $\omega_1(t),\omega_2(t),\ldots$ and the non-negative weights $\dip_1(t),\dip_2(t),\ldots$ are given by~\eqref{eqnGISDel}, computed using the time evolved moments $s_k(t)$.
\end{theorem}

\begin{proof}
In view of Corollary~\ref{corStrKForm} and Corollary~\ref{corStrinfpure}, we only need to prove that determinacy of the Hamburger moment problem is preserved under the conservative Camassa--Holm flow. 
However, this follows from~\cite[Exercise~6.2]{sc17} for example. 
\end{proof}

\begin{remark}\label{rem:finspeed}
A few remarks are in order:
\begin{enumerate}[label=(\alph*), ref=(\alph*), leftmargin=*, widest=e]
\item Let us point out that the necessary function $\kappa$ in~\eqref{eqnGISDel} is not independent of time, although it will only change when some of the determinants $\Delta_{1,k}$ become either zero or non-zero. Let us also stress that $\Delta_{1,k}$ is an entire function in $t$ since so are the moments in~\eqref{eq:momentInT} and hence the set of real zeros of $\Delta_{1,k}$ is a discrete subset of $\R$ for every $k$. 
\item As in the case of multi-peakons, collisions in this part of the solution may happen only between neighbouring peakons and they may happen only in pairs.
\item  Solutions with initial data $(u_0,\mu_0)$ having an infinite multi-peakon form are completely and explicitly given by the formulas in~\eqref{eqnGISDel}.
  We will consider several particular examples in the later sections. 
  \item If $L$ in~\eqref{eqnLviaDelta} is finite for a given initial data $(u_0,\mu_0)$, then it remains finite for all times. However, the evolution of the part of $(u,\mu)$ supported on $[L,\infty)$ is a nontrivial issue and, in general, requires the knowledge of the whole spectral measure. However, some particular cases will be considered in the next section.
 \item  Applying Theorem~\ref{thm:MPevolt} with $K=0$, we conclude that $\omega(\ledot,0)|_{(-\infty,x_1)} = 0$  and $\dip(\ledot,0)|_{(-\infty,x_1)} = 0$ for some $x_1\in \R$ if and only if for all $t\in\R$ there exists $x_1(t)\in \R$ such that $\omega(\ledot,t)|_{(-\infty,x_1(t))} = 0$  and $\dip(\ledot,t)|_{(-\infty,x_1(t))} = 0$. Moreover, the maximal such $x_1$ is explicitly given by
  \begin{align}\label{eq:CHx_1vias0}
  x_1(t) =  -\log \left(  \int_{\R}  \E^{-\frac{t}{2\lambda}} \rho_0(d\lambda)\right).
  \end{align} 
Therefore, conservative solutions enjoy the finite propagation speed and~\eqref{eq:CHx_1vias0} provides an explicit expression for the corresponding wave front. Notice that for classical solutions the finite propagation speed was observed in~\cite{co05}.  
\end{enumerate} 
\end{remark}

  Because the explicit formulas for the peakon part of a solution as above are the same as for usual multi-peakons in~\cite{besasz00}, the dynamics are alike too. 
  In this respect, we are only going to prove that away from times of collision, the positions and weights of the peakons satisfy the same differential equations (whereas these are derived in~\cite{besasz00} using the weak formulation of the Camassa--Holm equation, we will use the evolution of the moments and the explicit expressions of the positions and weights). 
  The behavior at collisions for example can be analyzed in a similar way as for multi-peakons using the properties of the Hankel determinants $\Delta_{l,k}$. 

 \begin{proposition}\label{prop:dynamics}
  Suppose that the spectral measure $\rho_0$ at initial time zero is not supported on a finite set and its moments exist up to order $2K$ for some $K\in\N\cup\{0\}$. 
 If the determinants $\Delta_{1,0}(t),\ldots,\Delta_{1,K}(t)$ are non-zero at some time $t\in\R$, then for all $n\in\{1,\ldots,K\}$ one has 
  \begin{align}\label{eqnxomegadot}
   \dot{x}_n(t) & = u(x_n(t),t), & -\frac{\dot{\omega}_n(t)}{\omega_n(t)} & = \frac{u_x(x_n(t)-,t)+u_x(x_n(t)+,t)}{2}. 
 \end{align}
 \end{proposition}
 
 \begin{proof}
  We first compute that  
  \begin{align*}
     \dot{s}_l & = -\frac{1}{2}s_{l-1}, & \dot \Delta_{l,k}=-\frac{1}{2}\Gamma_{l-1,k}, 
 \end{align*}
 for $l\geq0$, where $\Gamma_{l,k}$ denotes the determinant 
   \begin{align}\label{eqnGammaDef}
\Gamma_{l,k}= \begin{vmatrix} s_{l} & s_{l+1} & \cdots & s_{l+k-1} \\ s_{l+2} & s_{l+3} & \cdots & s_{l+k+1} \\ s_{l+3} & s_{l+4} & \cdots & s_{l+k+2} \\ \vdots & \vdots & \ddots & \vdots \\ s_{l+k} & s_{l+k+1} & \cdots & s_{l+2k-1} \end{vmatrix}
 \end{align}
 with the convention that $\Gamma_{l,k}=0$ when $k$ is zero.
 These determinants satisfy the bilinear identities 
  \begin{align}
 \label{eqnDeltaGamma1} \Delta_{l-1,k}\Delta_{l+2,k-1}&=\Delta_{l+1,k-1}\Gamma_{l-1,k}-\Delta_{l,k}\Gamma_{l,k-1},\\
 \label{eqnDeltaGamma2}  \Delta_{l-1,k+1}\Delta_{l+2,k-1}&=\Delta_{l+1,k}\Gamma_{l-1,k}-\Delta_{l,k}\Gamma_{l,k},
   \end{align}
 which can be found as in~\cite[Lemma 3.3]{chang2014generalized} via Sylvester’s determinant identity.
  
   Since the determinants $\Delta_{l,k}$ depend continuously on time, our assumption implies that the determinants $\Delta_{1,0},\ldots,\Delta_{1,K}$ are non-zero in a neighborhood of $t$ and thus also $\kappa(n)=n-1$, so that 
   \begin{align}\label{eqnxomegaDelta}
                x_n & = \log\frac{\Delta_{2,n-1}}{\Delta_{0,n}}, &
        	  \omega_n & = \frac{\Delta_{0,n}\Delta_{2,n-1}}{\Delta_{1,n-1}\Delta_{1,n}}, 
   \end{align}
   for $n\in\{1,\ldots,K\}$ in a neighborhood of $t$. 
   Differentiating this expression, we get 
     \begin{align*}
 \dot{x}_n(t)  &= \frac{1}{2}\biggl(\frac{\Gamma_{-1,n}(t)}{\Delta_{0,n}(t)}-\frac{\Gamma_{1,n-1}(t)}{\Delta_{2,n-1}(t)}\biggr) \\
                    & = \frac{1}{2}\biggl(\frac{\Delta_{0,n}(t)\Delta_{3,n-1}(t)}{\Delta_{1,n}(t)\Delta_{2,n-1}(t)}+\frac{\Delta_{-1,n+1}(t)\Delta_{2,n-1}(t)}{\Delta_{0,n}(t)\Delta_{1,n}(t)}\biggr),
 \end{align*}
 where we used relation~\eqref{eqnDeltaGamma1} with $l=1$ and~\eqref{eqnDeltaGamma2} with $l=0$ to obtain the second line.  
 In order to verify the first equation in~\eqref{eqnxomegadot}, it thus suffices to prove that 
  \begin{align}\label{eqnuxnDelta}
   u(x_n(t),t) & = \frac{1}{2}\biggl(\frac{\Delta_{0,n}(t)\Delta_{3,n-1}(t)}{\Delta_{1,n}(t)\Delta_{2,n-1}(t)}+\frac{\Delta_{-1,n+1}(t)\Delta_{2,n-1}(t)}{\Delta_{0,n}(t)\Delta_{1,n}(t)}\biggr).
 \end{align} 
 However, this follows readily from~\eqref{eqnunearminusinf} when $n=1$ and from the recursion in~\eqref{eqnuxnRecursion} for larger $n$, while also using~\eqref{eqnDeltaRel} with $l=2$ and $l=0$. 
  For the second equation in~\eqref{eqnxomegadot}, we first differentiate the expression in~\eqref{eqnxomegaDelta} to get 
  \begin{align*}
    -\frac{\dot{\omega}_n(t) }{\omega_n(t)} &=\frac{1}{2}\biggl(\frac{\Gamma_{-1,n}(t)}{\Delta_{0,n}(t)} + \frac{\Gamma_{1,n-1}(t)}{\Delta_{2,n-1}(t)} - \frac{\Gamma_{0,n-1}(t)}{\Delta_{1,n-1}(t)} - \frac{\Gamma_{0,n}(t)}{\Delta_{1,n}(t)}\biggr) \\
 &= \frac{1}{2}\biggl(\frac{\Delta_{-1,n}(t)\Delta_{2,n-1}(t)}{\Delta_{0,n}(t)\Delta_{1,n-1}(t)}-\frac{\Delta_{0,n}(t)\Delta_{3,n-1}(t)}{\Delta_{1,n}(t)\Delta_{2,n-1}(t)}\biggr),
 \end{align*}
 where we used relation~\eqref{eqnDeltaGamma1} with $l=0$ and $l=1$. 
 That this is equal to the right-hand side in~\eqref{eqnxomegadot} follows from plugging~\eqref{eqnuxnDelta},~\eqref{eqnu+u'Delta} and~\eqref{eqnxomegaDelta} into
 \begin{align*}
   \frac{u_x(x_n(t)-,t)+u_x(x_n(t)+,t)}{2} & = - u(x_n(t),t) + u(x_n(t),t) +  u_x(x_n(t)-,t) - \frac{\omega_n(t)}{2}
 \end{align*}
 and then using relation~\eqref{eqnDeltaRel} with $l=0$. 
 \end{proof}

 
 \begin{corollary}\label{cor:ODEforCH}
Let $(u_0,\mu_0)\in \CHdom^+$ be such that the corresponding spectral measure $\rho_0$ is not supported on a finite set and all its moments exist. If the corresponding moment problem is determinate, then for each time $t\in\R$, the solution $(u(\ledot,t),\mu(\ledot,t))$ has the form
    \begin{align}\label{eqnPurePeakonU}
      u(x,t) & =  \frac{1}{2}\sum_{n=1}^{\infty} \omega_n(t) \E^{-|x-x_n(t)|}, & \dip(\ledot,t) & = 0,
    \end{align}
  where the increasing points $x_1(t),x_2(t),\ldots$ in $\R$ with $x_n(t)\rightarrow\infty$ and the positive weights $\omega_1(t),\omega_2(t),\ldots$  satisfy the system 
 \begin{align}\label{eqnxomegadotMPsX}
   \dot{x}_n(t) & = \frac{1}{2}\sum_{k=1}^{\infty} \omega_k(t) \E^{-|x_n(t)-x_k(t)|}, \\
   \dot{\omega}_n(t) & = \frac{1}{2} \omega_n(t)\sum_{k=1}^\infty \sgn(x_n(t)-x_k(t)) \E^{-|x_n(t)-x_k(t)|}\omega_k(t).\label{eqnxomegadotMPsW}
   \end{align}
 \end{corollary}

\begin{proof}
The first claim is an immediate consequence of Theorem~\ref{thm:MPevolt} and the fact that $(u(\ledot,t),\mu(\ledot,t))\in \CHdom^+$ for all $t\in\R$. In particular, the latter means that the solution is determined solely by $u(\ledot,t)$. Moreover, by Corollary~\ref{corStrinfpure+},  $u(\ledot,t)$ has the form~\eqref{eqnPurePeakon}
\begin{align}\label{eq:Uinfpeak}
u(x,t) = \frac{1}{2}\sum_{n=1}^\infty \omega_n(t)\E^{-|x-x_n(t)|},
\end{align}
with $\omega_n\colon\R\to(0,\infty)$ and $x_n\colon \R\to\R$ such that $x_1(t)<x_2(t)<\dots$ for all $t\in\R$. 
Moreover, the assumption $(u_0,\mu_0)\in \CHdom^+$ implies that peakons do not experience any collisions (since the Hankel determinants $\Delta_{1,k}(t)$ stay positive for all $t\in\R$) and hence~\eqref{eqnxomegadot} holds true for all $t\in\R$ and all $n\in\N$. It remains to plug the form of $u(\ledot,t)$ into~\eqref{eqnxomegadot} and take into account that the convergence is locally uniform.     
\end{proof}
 
 \begin{remark}
The assumption $(u_0,\mu_0)\in \CHdom^+$ is superfluous and it is possible to state the result in a wider generality, however, we decided to state Corollary~\ref{cor:ODEforCH} in the above form for transparence reasons.  
 \end{remark}
 
   \begin{remark}
The situation when the corresponding moment problem is indeterminate (this, in particular, includes the case considered in~\cite{li09}) is more involved and will be discussed in the next section.  
 \end{remark}
 
  \begin{remark} \label{rem:num}
In the following sections we are going to present some numerical illustrations of particular examples. For convenience, we only consider the cases of positive measures $\rho_0$. In this case, the explicit expressions of positions and momentum are given by \eqref{eqnxomegaDelta}. And, in the interval $[x_n,x_{n+1}]$, the function $u$ is explicitly given according to \eqref{form_u_interval}-\eqref{form_S}. However,  in practice, in order to depict the graph of $u$ with sufficiently many peaks, one has to employ numerics and also to address the issue of sensible approximations caused by the calculation of high-order determinants and the machine error.

The numerics for the figures with sufficiently many peaks are produced by the following two main steps.

\underline{Step~1:} By using \textsc{maple}, we compute $x_n$ and $\omega_n$ with $n\in \{1,\ldots,N\}$ for a given $N\in\N$. More precisely, in the case of the Laguerre (see Section \ref{sec:LaguerrePeaks}) and the Jacobi (see Section \ref{sec:JacobiPeaks}) weights, we first calculate the moments $s_0$, $s_1$ and $s_2$ with the precision of up to 200 of digits (or even more) by calling the \textsc{int} function.  By observing a four-term linear recurrence satisfied by the moments of the Laguerre and Jacobi weights, we then recursively compute further moments $s_k$ with $k$ large enough so that the Hankel determinants $\Delta_{0,n}$, $\Delta_{1,n}$ and $\Delta_{2,n}$ with $n\in \{1,\ldots,N\}$ can be derived by calling the function \textsc{determinant}. Subsequently, we obtain  $x_n$ and $\omega_n$ with high-precision by using~\eqref{eqnxomegaDelta}. Furthermore, $u(x_n)$ and $u_x(x_n-)$ can be recursively calculated by employing the recurrence relations~\eqref{eqnuu'rec}--\eqref{form_S} once  $u(x_1)$ and $u_x(x_1-)$ are evaluated. Let us also mention that, in the Al-Salam--Carlitz  case (see Section \ref{sec:DiscrSpec}), we truncate the summation with the first 300 terms to approximate all of the required moments. 

\underline{Step~2:}
We obtain high-precision values of the function $u$ in the interval $(x_n,x_{n+1})$ by employing~\eqref{form_u_interval}. This step and the graphs are implemented in  \textsc{matlab}.  
 \end{remark}

\begin{remark}\label{rem:graph}
In order to plot a multi-peakon profile with a larger number, it might be beneficial to employ a trick based on Remark \ref{rem:scaling}. In fact, it is further observed that, if we scale the moments by
$$
\tilde s_k= cd^k\,s_k(\rho)$$
with some positive constants $c,\,d>0$,
then we have
$$\tilde\Delta_{l,k}= c^kd^{k(k-1)+lk}\, \Delta_{n,k}(\rho)$$
and
\begin{align*}
\tilde x_k& = x_k(\rho) - \log c, & \tilde \omega_{k}& = d^{-1}\,\omega_{k}(\rho).
\end{align*}
\end{remark}

\section{Purely discrete spectrum}\label{sec:DiscrSpec}
\subsection{The global conservative solution and long-time asymptotics} 
As a first application of the explicit formulas derived in the previous sections, we are now going to establish long-time asymptotics for the positions and heights of the individual  peakons under the presumption that the support of the spectral measure $\rho_0$ at initial time zero is a positive discrete set $\sigma$.
 As before, we will consider the solution  
\begin{align}
  (u(\ledot,t),\mu(\ledot,t)) = \Phi^t(u_0,\mu_0),
\end{align}
where the initial data $(u_0,\mu_0)$ is the pair in $\CHdom^+$ corresponding to the measure $\rho_{0}$. 
 If all moments of $\rho_0$ exist, then it follows from Corollary~\ref{corStrinf} and Corollary~\ref{corIP+} that this solution has the form
    \begin{align}
      \omega(\ledot,t)|_{(-\infty,L(t))} & =  \sum_{n=1}^{\infty} \omega_n(t) \delta_{x_n(t)}, & \dip(\ledot,t) & = 0,  & L(t) & = \lim_{n\rightarrow\infty} x_n(t). 
    \end{align}
 with coefficients given by~\eqref{eq:3.9positive}.

\begin{theorem}\label{thm:longt_dis}
Suppose that the spectral measure $\rho_0$ at initial time zero is not supported on a finite set and all its moments exist. 
If the support of $\rho_0$ is a positive discrete set so that $\supp(\rho_0)=\{\lambda_n\,|\,n\in\N\}$ for an unbounded and strictly increasing sequence $(\lambda_n)$, then for every fixed $n\in\N$ the following asymptotics hold:
 \begin{enumerate}[label=(\roman*), ref=(\roman*), leftmargin=*, widest=ii]
\item\label{itmLongti} As $t\rightarrow\infty$ one has 
 \begin{align}\label{eqnLongtiDiscr}
   x_n(t) & \rightarrow \infty, &  x_n(t) &=o(t),& \omega_n(t)&  \rightarrow 0.
 \end{align}
\item\label{itmLongtii} As $t\rightarrow-\infty$ one has 
 \begin{align}\label{asym_neg}
   x_n(t) & = \frac{t}{2\lambda_n} -\log\rho_0(\{\lambda_n\}) - 2\sum_{i=1}^{n-1} \log\biggl(\frac{\lambda_n}{\lambda_i}-1\biggr) + \oo(1), & \omega_n(t) & \rightarrow \frac{1}{\lambda_n}.
 \end{align}
\end{enumerate}
\end{theorem}
 
 \begin{remark}\label{rem:PeakDiscr+}
 Before proving the above statements, let us mention that the asymptotics as $t\rightarrow\infty$ are somewhat misleading about what happens to the function $u$ as it appears that all peakons fade away and no peakons are present in the long-time picture (in particular, $u(\ledot,t)$ converges to zero pointwise as $t\to \infty$). 
 However, this is only due to the initial location of the peakons accumulating to the left of some point, which leads to an infinite number of interactions that each peakon experiences, successively driving its height to zero. 
 Looking at the function $u$ instead, one has the conservation law
  \begin{align}\label{eq:lawW1infty}
     \|(u+u_x)(\ledot,0)\|_{L^\infty(\R)} \leq 12\sqrt{2}\|(u+u_x)(\ledot,t)\|_{L^\infty(\R)}, 
 \end{align}
which holds true for all times $t\in\R$ (see~\cite[Remark~5.10]{Eplusminus}) and, moreover, one can indeed observe the emergence of an infinite train of peakons after long enough time; see~\cite[Section~9]{IsospecCH}, \cite[Section~4]{CouplingProblem} and~\cite[Section~6]{CHTrace}.  
 \end{remark}

\begin{remark}\label{rem:DSpec}
  Discreteness of the support of the spectral measure $\rho$ is equivalent to a certain decay condition on the pair $(u,\mu)$ at $+\infty$; see~\cite[Proposition~3.10]{Eplusminus}. 
 More specifically, the support of the spectral measure $\rho$ of a pair $(u,\mu)$ in $\CHdom$ is discrete if and only if 
  \begin{align}\label{eqnSinSinftyCH}
    \lim_{x\rightarrow\infty} \E^{x}\biggl(\int_{x}^{\infty}\E^{-s}(u(s) + u'(s))^2ds + \int_{x}^{\infty}\E^{-s}\dip(ds)\biggr) = 0.
  \end{align}
 For instance, this condition is satisfied when $(u,\mu)$ is of the form~\eqref{eqnPurePeakon} with
 \begin{align}
   \sum_{n=1}^\infty |\omega_n| & < \infty, & \sum_{n=1}^\infty \dip_n & < \infty. 
 \end{align}
 Under the additional assumption that the support of the spectral measure is positive, that is, that $\rho\in\SM_0^+$, condition~\eqref{eqnSinSinftyCH} is equivalent to 
 \begin{align}\label{eqnSinSinftyCH+}
    \lim_{x\rightarrow\infty} u(x) + u'(x-) = 0;
  \end{align}
  see~\cite[Corollary~3.13]{Eplusminus}. 
In the case of an infinite multi-peakon profile of the form~\eqref{eqnPurePeakon}, the latter is clearly equivalent to (compare this with Remark~\ref{rem:Hardy})
 \begin{align}\label{eqnSinSinftyCH+MP}
    \lim_{n\rightarrow\infty} \E^{x_n}\sum_{k\ge n} \frac{\omega_k}{\E^{x_k}}= \lim_{n\rightarrow\infty} \sum_{x_k \in [n,n+1)} \omega_k = 0.
  \end{align} 
 \end{remark}
 
 \begin{proof}[Proof of Theorem~\ref{thm:longt_dis}]
Since the spectral measure $\rho_0$ is of the form 
  \begin{align*}
    \rho_0 = \sum_{n\in\N} \gamma_n \delta_{\lambda_n}
  \end{align*}
  for some positive constants $\gamma_n$, the formula in Remark~\ref{remDeltas} turns into 
  \begin{align}\label{eqnDeltaRepLT}
    \Delta_{l,k}(t)=\sum_{J\in \mathcal{J}_k}\mgam_{J}\mlam_{J}^{l}\wand_k(\mlam_J)\exp\biggl(-\frac{t}{2}\sum\limits_{j \in J} \frac{1}{\lambda_j}\biggr), 
      \end{align}
  where the set $\mathcal{J}_k$ consists of all subsets of $\N$ with precisely $k$ elements and
  \begin{equation*}
    \mgam_J=\prod_{j\in J}\gamma_{j},
    \qquad
    \mlam_J^l=\prod_{j\in J}\lambda_j^l,
    \qquad
    \wand_k(\mlam_J)=\prod_{i,j\in J\atop i<j}|\lambda_i-\lambda_j|^2,
  \end{equation*}
  for $J\in\mathcal{J}_k$. 
  In addition, we also have a similar representation
  \begin{align*}
 \Gamma_{l-1,k}(t) &=  \sum_{J\in \mathcal{J}_k} \sum_{i\in J}\frac{1}{\lambda_i} \mgam_{J}\mlam_{J}^{l}\wand_k(\mlam_J)\exp\biggl(-\frac{t}{2}\sum\limits_{j \in J} \frac{1}{\lambda_j}\biggr)
 \end{align*}
  for the determinants defined in~\eqref{eqnGammaDef}.
 Now the expression for $\omega_n$ in~\eqref{eq:3.9positive} can be rewritten as
\begin{align*}
 \omega_n(t) & = \frac{\Delta_{0,n}(t)\Delta_{2,n-1}(t)}{\Delta_{1,n-1}(t)\Delta_{1,n}(t)}=\frac{\Gamma_{0,n}(t)}{\Delta_{1,n}(t)}-\frac{\Gamma_{0,n-1}(t)}{\Delta_{1,n-1}(t)}, 
\end{align*}
where we have used the identity (see~\cite[Lemma 3.3]{chang2014generalized} for example)
$$ \Delta_{2,n-1}\Delta_{0,n}=\Delta_{1,n-1}\Gamma_{0,n}-\Gamma_{0,n-1}\Delta_{1,n}.$$ 
In order to prove that $\omega_n(t)\rightarrow0$ as $t\rightarrow\infty$, it is therefore sufficient to show that 
\begin{align*}
  \frac{\Gamma_{l-1,k}(t)}{\Delta_{l,k}(t)} \rightarrow 0.
\end{align*}
 For a given $\eps>0$, choose $K\in\N$ such that $\frac{1}{\lambda_K}\leq\eps$. 
 We then have the inequality
     \begin{align*}
  \sum_{J\in \mathcal{J}_k\atop \min(J)< K}\sum \limits_{i\in J}\frac{1}{\lambda_i}  \mgam_{J}\mlam_{J}^{l}\wand_k(\mlam_J)\exp\biggl(-\frac{t}{2}\sum\limits_{j \in J} \frac{1}{\lambda_j}\biggr)&  \leq  \sum_{J\in \mathcal{J}_k\atop \min(J)< K}\sum \limits_{i\in J}\frac{1}{\lambda_i}  \mgam_{J}\mlam_{J}^{l}\wand_k(\mlam_J)\E^{-\frac{t}{2\lambda_K}}\\
  & = C_K \E^{-\frac{t}{2\lambda_K}} 
 \end{align*}
 for some positive constant $C_K$, from which we get 
  \begin{align*}
\frac{ \Gamma_{l-1,k}(t)}{\Delta_{l,k}(t)}&=\frac{\sum\limits_{J\in \mathcal{J}_k} \sum\limits_{i\in J}\frac{1}{\lambda_i} \mgam_{J}\mlam_{J}^{l}\wand_k(\mlam_J)\exp\Big(-\frac{t}{2}\sum\limits_{j \in J} \frac{1}{\lambda_j}\Big)}{\sum\limits_{J\in \mathcal{J}_k}\mgam_{J}\mlam_{J}^{l}\wand_k(\mlam_J)\exp\Big(-\frac{t}{2}\sum\limits_{j \in J} \frac{1}{\lambda_j}\Big)} \\
&\leq\frac{\sum\limits_{J\in \mathcal{J}_k\atop \min(J)\geq K}\sum \limits_{i\in J}\frac{1}{\lambda_i}  \mgam_{J}\mlam_{J}^{l} \wand_k(\mlam_J) \exp\Big(-\frac{t}{2}\sum\limits_{j \in J} \frac{1}{\lambda_j}\Big)}{\sum\limits_{J\in \mathcal{J}_k}\mgam_{J}\mlam_{J}^{l} \wand_k(\mlam_J) \exp\Big(-\frac{t}{2}\sum\limits_{j \in J} \frac{1}{\lambda_j}\Big)} \\
 & \qquad \qquad\qquad\qquad+ \frac{C_K \E^{-\frac{t}{2\lambda_K}}}{\sum\limits_{J\in \mathcal{J}_k}\mgam_{J}\mlam_{J}^{l} \wand_k(\mlam_J) \exp\Big(-\frac{t}{2}\sum\limits_{j \in J} \frac{1}{\lambda_j}\Big)}.
 \end{align*}
Notice that the last summand 
 is bounded by 
   \begin{align*}
  \frac{C_K \E^{-\frac{t}{2\lambda_K}}}{\sum\limits_{J\in \mathcal{J}_k}\mgam_{J}\mlam_{J}^{l} \wand_k(\mlam_J) \exp\Big(-\frac{t}{2}\sum\limits_{j \in J} \frac{1}{\lambda_j}\Big)} \leq \frac{C_K}{\mgam_{J}\mlam_{J}^{l} \wand_k(\mlam_J)}\exp\biggl(-\frac{t}{2} \biggl(\frac{1}{\lambda_K}-\sum\limits_{j \in J} \frac{1}{\lambda_j}\biggr)\biggr)
 \end{align*}
 for every $J\in \mathcal{J}_k$. 
 Since $J$ can be chosen such that the right-hand side tends to zero as $t\rightarrow\infty$, we get 
   \begin{align*}
 \limsup_{t\rightarrow\infty} \frac{ \Gamma_{l-1,k}(t)}{\Delta_{l,k}(t)} 
 &\leq \limsup_{t\rightarrow\infty} \frac{\sum\limits_{J\in \mathcal{J}_k\atop \min(J)\geq K}\sum \limits_{i\in J}\frac{1}{\lambda_i}  \mgam_{J}\mlam_{J}^{l} \wand_k(\mlam_J) \exp\Big(-\frac{t}{2}\sum\limits_{j \in J} \frac{1}{\lambda_j}\Big)}{\sum\limits_{J\in \mathcal{J}_k}\mgam_{J}\mlam_{J}^{l} \wand_k(\mlam_J) \exp\Big(-\frac{t}{2}\sum\limits_{j \in J} \frac{1}{\lambda_j}\Big)} \\
&\leq \limsup_{t\rightarrow\infty} \frac{\sum\limits_{J\in \mathcal{J}_k\atop \min(J)\geq K} \frac{k}{\lambda_K}  \mgam_{J}\mlam_{J}^{l} \wand_k(\mlam_J) \exp\Big(-\frac{t}{2}\sum\limits_{j \in J} \frac{1}{\lambda_j}\Big)}{\sum\limits_{J\in \mathcal{J}_k\atop \min(J)\geq K}\mgam_{J}\mlam_{J}^{l} \wand_k(\mlam_J) \exp\Big(-\frac{t}{2}\sum\limits_{j \in J} \frac{1}{\lambda_j}\Big)}\\ & = \frac{k}{\lambda_K} \leq k\varepsilon,
 \end{align*}
which implies that $\omega_n(t)\rightarrow0$ as $t\rightarrow\infty$. 
 In order to prove that $x_n(t)\rightarrow\infty$, it clearly suffices to show that $x_1(t)\rightarrow\infty$. 
However, this follows immediately by using the formula
    \begin{align*}
      x_1(t)  = \log\frac{\Delta_{2,0}(t)}{\Delta_{0,1}(t)}= - \log \Biggl(\sum_{n\in\N}\gamma_n\E^{-\frac{t}{2\lambda_n}}\Biggr)
    \end{align*}
 and noting that the sum converges to zero by dominated convergence. 
 For the remaining claim in~\ref{itmLongti}, we simply note that  
 \begin{align*} 
  \dot{x}_n(t)  = \frac{1}{2}\biggl(\frac{\Gamma_{-1,n}(t)}{\Delta_{0,n}(t)}-\frac{\Gamma_{1,n-1}(t)}{\Delta_{2,n-1}(t)}\biggr) \rightarrow0
 \end{align*}
 as $t\rightarrow\infty$, which implies that $x_n(t) =\oo(t)$.
 
 As for the asymptotics when $t\rightarrow-\infty$, it suffices to observe that from the representation in~\eqref{eqnDeltaRepLT} one gets 
\begin{align}\label{eqnDeltalk+}
  \Delta_{l,k}(t) \sim \mgam_{J_k} \mlam_{J_k}^l \wand_k(\mlam_{J_k}) \exp\biggl(-\frac{t}{2} \sum\limits_{i\in J_k}\frac{1}{\lambda_i}\biggr),
\end{align}
 where $J_k=\{1,\ldots,k\}$. 
 The asymptotics in~\ref{itmLongtii} then follow readily by using the formulas in~\eqref{eq:3.9positive}. 
 \end{proof}

\begin{remark}
  The convergence in the asymptotics in~\eqref{asym_neg} as $t\rightarrow-\infty$ is actually exponential. 
  In fact, this follows because~\eqref{eqnDeltalk+} can easily be improved to 
  \begin{align}
     \frac{\Delta_{l,k}(t)}{\mgam_{J_k} \mlam_{J_k}^l \wand_k(\mlam_{J_k})} \exp\biggl(-\frac{t}{2} \sum\limits_{i\in J_k}\frac{1}{\lambda_i}\biggr) = 1+\oo\Bigl(\E^{-\varepsilon_k|t|}\Bigr)
  \end{align}
  for every positive $\varepsilon_k$ with 
  \begin{align}
    \varepsilon_k < \frac{1}{2\lambda_{k}} - \frac{1}{2\lambda_{k+1}}. 
  \end{align} 
\end{remark}

\begin{remark}\label{rem:DiscrAsymp2K}
A careful inspection of the proof of Theorem~\ref{thm:longt_dis} shows that the asymptotic formulas remain true if the spectral measure $\rho_0$ at initial time has only finitely many moments. More specifically, if the moments of $\rho_0$ exist up to order $2K$ with some $K\in \N\cup\{0,\infty\}$, then it follows from 
Theorem~\ref{thmMomK} and Corollary~\ref{corStrKForm} that the corresponding  solution has the form
    \begin{align}\label{eq:thm5.1finpeak}
      \omega(\ledot,t)|_{(-\infty,x_{K+1}(t))} & =  \sum_{n=1}^{K} \omega_n(t) \delta_{x_n(t)}, & \dip(\ledot,t) & = 0. 
    \end{align}
 with coefficients given by~\eqref{eq:3.9positive}. Moreover, the determinants $\Delta_{0,k}$, $\Delta_{1,k}$, $\Delta_{2,k}$ and $\Gamma_{0,k}$ exist for all $k\in \{0,\dots,2K\}$. Therefore, the asymptotic formulas~\eqref{eqnLongtiDiscr} and~\eqref{asym_neg} hold true, however, for $x_n$ with $n\in \{1,\dots,K+1\}$ and for $\omega_n$ with $n\in \{1,\dots,K\}$.
\end{remark}

\begin{remark}
Under the positivity assumption of Theorem~\ref{thm:longt_dis}, the support of the spectral measure $\rho_0$ satisfies
\begin{align}\label{eq:SpforCH}
\sum_{n\in \N}\lambda_n^{-p} <\infty
\end{align}
for some $p>\nicefrac{1}{2}$ if and only if $u_0+u_0'\in L^p(\R)$; see~\cite[Corollary~3.13~(ii)]{Eplusminus}. Clearly, the condition~\eqref{eq:SpforCH} is invariant under the conservative Camassa--Holm flow~\eqref{eqnSMEvo} and in fact, by~\cite[Remark~5.10 and Corollary~5.13]{Eplusminus}, for every $p\in (\nicefrac{1}{2},\infty]$ there is a positive constant $c_p>0$ such that  
\begin{align}\label{eq:lawW1p+}
\frac{1}{c_p}\|(u+u_x)(\ledot,0)\|_{L^p(\R)} \le \|(u+u_x)(\ledot,t)\|_{L^p(\R)} \le c_p\|(u+u_x)(\ledot,0)\|_{L^p(\R)} 
\end{align}
for all times $t\in \R$.  
\end{remark}

Our next goal is to apply the above results to a particular subclass of profiles considered in Remark~\ref{rem:InftyMinToPlus}. 

\begin{definition}
For a given $L\in\R$, let $\Peakons_{L}$ be the collection of all pairs $(u,\mu)$ in $\CHdom$ having the form~\eqref{eqnstrinf2side} with $x_n$-s accumulating at $L$ from the left and 
 \begin{align}\label{eqnstrinf2sideN+}
      \omega|_{(L,\infty)} & =  \sum_{j=1}^{\kappa} \omega_{n_j} \delta_{x_{n_j}}, & \dip|_{(L,\infty)} & =  \sum_{j=1}^{\kappa} \dip_{n_j} \delta_{x_{n_j}},
    \end{align}
    for some $\kappa\in \N\cup\{0\}$. 
    \end{definition}

The importance of the above class of profiles stems from the fact that they are closely connected with $n$-canonical solutions of indeterminate moment problems. 

\begin{theorem}\label{thm:MPevoltIndet}
    If a pair $(u_0,\mu_0)$ belongs to $\Peakons_{L_0}$ for some $L_0\in\R$, then for each time $t\in\R$ there is $L(t)\in\R$  such that the solution $(u(\ledot,t),\mu(\ledot,t))$ belongs to $\Peakons_{L(t)}$.
\end{theorem}

\begin{proof}
First of all, observe that the moment problem associated with the corresponding spectral measure $\rho_0$ is indeterminate. Next, consider the generalized indefinite string $(\infty,\tilde{\omega}_0,\tilde{\dip}_0)$ associated with the given pair $(u_0,\mu_0)$ via~\eqref{eqnDefa} and~\eqref{eqnDefbeta}. It is immediate to observe that this generalized indefinite string is such that the triple $(\E^{L_0},\tilde{\omega}_0|_{[0,\E^{L_0})},\tilde{\dip}_0|_{[0,\E^{L_0})})$ is a Krein--Langer string and the truncated string $(\infty,\tilde{\omega}_{0,\E^{L_0}},\tilde{\dip}_{0,\E^{L_0}})$ (see~\eqref{eq:StringTrunc}) is also a Krein--Langer string with $N=\kappa$ in~\eqref{eqnKL}. According to Remark~\ref{rem:m-canon}~\ref{itmremNextremii}, the spectral measure of the pair $(u_0,\mu_0)$ is an $n$-canonical solution of the associated indeterminate moment problem for some finite $n\in\N\cup\{0\}$. The latter in particular means that the closure of polynomials in $L^2(\R;\rho_0)$ is a subspace of codimension $n$ in $L^2(\R;\rho_0)$. Taking into account that $0\notin \supp(\rho_0)$,  the same is true for the measure 
   \begin{align*}
\rho(t; d\lambda) = \E^{-\frac{t}{2\lambda}} \rho_0(d\lambda)
  \end{align*}
  with any $t\in \R$, 
which completes the proof by applying Remark~\ref{rem:m-canon}~\ref{itmremNextremii} together with the map~\eqref{eqnuitofa}. 
\end{proof}

\begin{remark}
Clearly, $L(t)$ does depend on $t$ (for instance, under the assumptions of Theorem~\ref{thm:longt_dis}, all peakons go to infinity as $t\rightarrow\infty$). The non-negative integer $n$ in the proof of Theorem~\ref{thm:MPevoltIndet} can be expressed (or at least estimated) in terms of the corresponding coefficients in~\eqref{eqnstrinf2sideN+} at initial time $t=0$. 
 \end{remark}
 
   If a pair $(u,\mu)$ in $\CHdom$ is of multi-peakon form to the right of some point $\ell\in\R$, that is, such that 
    \begin{align}
   \omega|_{(\ell,\infty)} & = \sum_{n=1}^N \omega_n \delta_{x_n}, & \dip|_{(\ell,\infty)} & = \sum_{n=1}^N \dip_n \delta_{x_n}, 
 \end{align}
  for some $N\in\N\cup\{0\}$, decreasing points $x_1,\ldots,x_N$ in $(\ell,\infty)$, real weights $\omega_1,\ldots,\omega_N$ and non-negative weights $\dip_1,\ldots,\dip_N$, then it also satisfies a strong decay restriction as in~\eqref{eqnMdef-} at $+\infty$ (notice that the set of pairs $\Peakons_L$ is obviously of this form).
   First of all, this entails that the spectral measure $\rho$ is supported on a discrete set $\sigma$ (this can be seen, for example, from~\eqref{eqnSinSinftyCH}) so that  
   \begin{align}\label{eq:rho_discr}
     \rho = \sum_{\lambda\in\sigma} \gamma_\lambda \delta_\lambda
   \end{align}
   for some positive constants $\gamma_\lambda$, which are usually expressed via the norms of the corresponding eigenfunctions. 
   Moreover, the infinite product
   \begin{align}
     W(z) = \lim_{R\rightarrow\infty} \prod_{\lambda\in\sigma \atop |\lambda|\leq R} \biggl(1-\frac{z}{\lambda}\biggr)
   \end{align}
   converges locally uniformly to an entire function $W$ in this case. 
   Due to the strong decay of the coefficients at $+\infty$, we can furthermore introduce a fundamental system of solutions with prescribed asymptotics at $+\infty$ and an associated spectral measure $\rho_+$ as in Section~\ref{secPre}. 
   This measure can be recovered from $\rho$ via
   \begin{align}\label{eq:rho+}
     \rho_+ = \sum_{\lambda\in\sigma} \frac{1}{\lambda^2\dot{W}(\lambda)^2\gamma_\lambda} \delta_\lambda.
   \end{align}
   In fact, all this is known because in this case the pair $(u,\mu)$ belongs to the class considered in~\cite{ConservCH}; see~\cite[Lemma~2.4]{ConservCH} for the relation between $\rho_+$ and $\rho$.
   It is now possible to apply all the results of this section to the measure $\rho_+$ to obtain characterizations of spectral measures $\rho$ for pairs $(u,\mu)$ in $\CHdom$ that are of multi-peakon form near $+\infty$.
   Of course, one then again obtains explicit formulas for the respective positions and weights of peakons. 
Let us demonstrate this by considering a particular case of pairs from the set $\Peakons_L$. 

\begin{lemma}\label{lem:LiCMP}
Assume that a pair $(u_0,\mu_0)$ belongs to $\Peakons_{L_0}$ for some $L_0\in\R$ and 
 \begin{align}
      \omega_0|_{(L_0,\infty)} & =  0, & \dip_0|_{[L_0,\infty)} & =  0.
    \end{align}
Then the solution $(u(\ledot,t),\mu(\ledot,t))$ belongs to $\Peakons_{L(t)}$ and 
 \begin{align}\label{eqnstrinf2side0+}
      \omega(\ledot,t)|_{(L(t),\infty)} & =  0, & \dip(\ledot,t)|_{[L(t),\infty)} & =  0
    \end{align}
    for all $t\in\R$. Moreover, 
    \begin{align}\label{eq:LevolviaRho+}
   L(t) = \log\left(\sum_{\lambda\in\sigma} \frac{\E^{\frac{t}{2\lambda}}}{\lambda^2\dot{W}(\lambda)^2\gamma_\lambda}\right), 
    \end{align}
and the coefficients of the solution in $(-\infty,L(t))$ can be recovered recursively as described in Corollary~\ref{corStrKForm} and Corollary~\ref{cor:UviaMoments}.
\end{lemma}

\begin{proof}
Since $(u_0,\mu_0)\in\Peakons_{L_0}$ has the form~\eqref{eqnstrinf2sideN+} with $\kappa=0$ and $\dip_0(\{L_0\}) = 0$, we conclude that the corresponding spectral measure $\rho_0$ is an N-extremal solution of the associated moment problem (see Proposition~\ref{prop:N-extrem}). The latter is equivalent to the fact that the polynomials are dense in $L^2(\R,\rho_0)$. Therefore, arguing as in the proof of Theorem~\ref{thm:MPevoltIndet}, one easily shows that~\eqref{eqnstrinf2side0+} takes place for all  $t\in\R$. 

It remains to notice that the formula for $L(t)$ follows from~\eqref{eq:x_1vias0} applied to the measure $\rho_+$; we only need to mention that the evolution of the spectral measure $\rho_+$ under the conservative Camassa--Holm flow is governed by the same formula~\eqref{eqnSMEvo}, however, the sign in the exponent must be different. 
\end{proof}

\begin{remark}
If $(u_0,\mu_0)\in\Peakons_L$, then applying Theorem~\ref{thm:longt_dis} to the measure $\rho_+$ one can obtain the asymptotics for the multi-peakon part of the solution $(u(\ledot,t),\mu(\ledot,t))$ lying to the right of $L(t)$. We will demonstrate this below in one particular situation.  
\end{remark}

The results of this section extend and complement the main results of~\cite{li09}. Let us state them explicitly.

\begin{corollary}\label{cor:Li}
Assume that the pair $(u_0,\mu_0)\in \CHdom^+$ is such that 
\begin{align}\label{eq:Lidata}
u_0(x) = \frac{1}{2}\sum_{n\in\N}\omega_n \E^{-|x-x_n|},
\end{align}
where $x_1<x_2<\dots$ is a strictly increasing bounded sequence with $\sup_{n\in\N}x_n = L_0 $ and  $\omega_1,\omega_2,\dots$ are positive weights such that
\begin{align}
\omega_0(\R) = \sum_{n\in\N}\omega_n<\infty.
\end{align}
Then:
 \begin{enumerate}[label=(\roman*), ref=(\roman*), leftmargin=*, widest=ii]
 \item The corresponding solution $(u(\ledot,t),\mu(\ledot,t))$ belongs to $\CHdom^+$ for all $t\in\R$ and $u$ has the form
\begin{align}
u(x,t) = \frac{1}{2}\sum_{n\in\N}\omega_n(t) \E^{-|x-x_n(t)|},
\end{align}
where the coefficients satisfy the system of equations~\eqref{eqnxomegadotMPsX}, \eqref{eqnxomegadotMPsW}. 
\item For all $t\in\R$
\begin{align}\label{eq:Wconserv}
\sum_{n\in\N}\omega_n(t)  = \omega_0(\R), 
\end{align}
and $x_1(t)<x_2(t)<\dots$ with $\sup_{n\in\N} x_n(t) = L(t)<\infty$. 
\item As $|t|\to\infty$, the asymptotic behavior of positions and heights is given by~\eqref{eqnLongtiDiscr} and~\eqref{asym_neg}. Moreover,
\begin{align}\label{eq:LasympLi}
L(t) = \begin{cases} \frac{t}{2\lambda_1} + \oo(t), & t\to\infty, \\ \oo(t), & t\to -\infty, \end{cases}
\end{align}
where $\lambda_1$ is the smallest eigenvalue of the corresponding spectral problem.
\end{enumerate}
\end{corollary}

\begin{proof}
By using Lemma~\ref{lem:LiCMP} and following the proof of Corollary~\ref{cor:ODEforCH}, in order to prove (i) it only suffices to show that $\omega(\ledot,t)$ has no point mass at $L(t)$. However, changing the roles of $-\infty$ and $+\infty$ by using the change of variables $(x,t)\mapsto (-x,-t)$ and considering the spectral measure $\rho_+$ defined by~\eqref{eq:rho+}, the latter follows from Lemma~\ref{lem:s1}. Indeed, since $u_0 - u_0''$ has no point mass at $L_0$, the corresponding spectral measure $\rho_{+}(\ledot,0)$ at initial time has infinite first moment according to Lemma~\ref{lem:s1}. However, 
\begin{align*}
\rho_+(d\lambda;t) = \E^{\frac{t}{2\lambda}}\rho_{+}(d\lambda;0)
\end{align*}
and hence so does the measure $\rho_+(\ledot;t)$ for all $t\in\R$. It remains to use  Lemma~\ref{lem:s1} once again.

(ii) We only need to mention that~\eqref{eq:Wconserv} follows from~\cite[Proposition~4.1]{ConservCH} for example. 

(iii) The asymptotic behavior of positions and heights follow from Theorem~\ref{thm:longt_dis} and Remark~\ref{rem:DiscrAsymp2K}. Indeed, one simply needs to apply these results with $K=0$ to the spectral measure $\rho_+$ and take into account the fact that the sign of $t$ in the exponent must be changed.
\end{proof}

\begin{remark}\label{rem:Li01}
A few remarks are in order.
\begin{enumerate}[label=(\alph*), ref=(\alph*), leftmargin=*, widest=e]
\item Notice that~\eqref{eq:LasympLi} complements Theorem~\ref{thm:longt_dis}~\ref{itmLongti}. More specifically, in case if $u_0$ has the form~\eqref{eqnPurePeakonU} with positions accumulating at some finite point $L_0$,  \eqref{eq:LasympLi} shows that as $t\to\infty$ this accumulation point moves asymptotically at constant speed proportional to the reciprocal of the smallest eigenvalue, however, at the same time the speed of each single peakon tends to zero.  
\item If in the setting of Corollary~\ref{cor:Li} $u_0$ has an additional positive peakon at $L_0$, that is, 
\begin{align}
u_0(x) = \tilde{u}_0(x) + \frac{1}{2}\omega_{L_0}\E^{-|x-L_0|},
\end{align}
with $\tilde{u}_0$ given by the right-hand side in~\eqref{eq:Lidata} and $\omega_{L_0}>0$, then, by Lemma~\ref{lem:LiCMP}, the corresponding solution $u(\ledot,t)$ will have the same form for all $t\in\R$. Moreover, by Corollary~\ref{cor:Li}~(i) and Lemma~\ref{lem:s1}, $\omega(\ledot,t)$ has a positive mass at $L(t)$ for all $t\in\R$, which is given explicitly by 
\begin{align}
\omega(\{L(t)\},t) = \frac{s_0(\rho_+(t))}{s_1(\rho_+(t))}.
\end{align}
Using~\eqref{eq:LevolviaRho+}, it is not difficult to show that  
\begin{align}\label{eq:WasympLi}
\omega(\{L(t)\},t) = \begin{cases} \frac{1}{\lambda_1} + \oo(1), & t\to\infty, \\ \oo(1), & t\to -\infty. \end{cases}
\end{align}
\item
By using a different approach based on Toda flows, Corollary~\ref{cor:Li} (except~\eqref{eq:LasympLi}) was obtained by L.-C.~Li in~\cite{li09} under the additional assumption
\begin{align}\label{eq:LIcond}
\sum_{n=1}^\infty \omega_n n^2 <\infty.
\end{align}
Let us mention that a solution to the Camassa--Holm equation~\eqref{eqnCH} with the initial data~\eqref{eq:Lidata} was constructed in~\cite{li09} as a solution to the infinite dimensional ODE system~\eqref{eqnxomegadotMPsX}, \eqref{eqnxomegadotMPsW}.\footnote{Let us mention that the class $S_+$ considered in~\cite{li09} can be obtained from $S_-$ by using the usual symmetry of the Camassa--Holm equation $u(x,t)\to u(-x,-t)$.} However, taking into account Proposition~\ref{prop:dynamics} and using~\cite{Eplusminus}, 
it is not difficult to show that weak solutions to the Camassa--Holm equation constructed in~\cite{li09} coincide with weak solutions constructed with the help of the conservative Camassa--Holm flow in Theorem~\ref{thm:consCHweak}. On the other hand, it does not seem clear to us how to extend the approach of~\cite{li09} to the class of pairs $\Peakons_L$. In particular, it is not clear to us whether one could get~\eqref{eq:LasympLi} by using the approach of~\cite{li09}.
\end{enumerate}  
 \end{remark}
 
 \begin{remark}\label{rem:2sidedPeakons}
Using the results of this section, it is also possible to characterize the spectral measures of pairs in $\CHdom$ that are of multi-peakon form near $-\infty$ and near $+\infty$.  
   For instance, in this way one can characterize the spectral measures $\rho$ of all pairs $(u,\mu)$ in $\CHdom$ that are of the form
       \begin{align}
      \omega|_{\R\backslash\{L\}} & =  \sum_{n=1}^{\infty} \omega_n \delta_{x_n}, & \dip|_{\R\backslash\{L\}} & =  \sum_{n=1}^{\infty} \dip_n \delta_{x_n},
    \end{align}
     for some points $x_1,x_2,\ldots$ in $\R\backslash\{L\}$ that only accumulate at $L$, real weights $\omega_1,\omega_2,\ldots$ and non-negative weights $\dip_1,\dip_2,\ldots$, where $L$ is a given point in $\R$. 
  More precisely, this can be done by applying Corollary~\ref{corStrinf} also to the measure $\rho_+$, which is related to $\rho$ via~\eqref{eq:rho+}. 
 \end{remark}

\subsection{Al-Salam--Carlitz peakons}\label{ss:AlSalamCarlitz}
As it was mentioned earlier, cases of indeterminate moment problems are present within the $q$-analog of the Askey-scheme.  The first example of an N-extremal measure was pointed in the 1960s  by T.~S.~Chihara~\cite{chi68} and the corresponding construction involves the Al-Salam--Carlitz polynomials of type II. We are going to use it to construct an explicitly solvable pair from $\Peakons_{L}$, which also belongs to the class considered in~\cite{li09}.

\begin{proposition}\label{prop:ASCpoln}
Let $a$, $q\in (0,\infty)$ satisfy $1<a<\nicefrac{1}{q}$.  
 Then the spectral measure $\rho_0$ of the pair $(u,\mu)\in \CHdom^+$ with $u$ having the form~\eqref{eq:Lidata} and the coefficients given by
 \begin{align}\label{eq:ASCcoef}
x_n & = \log \left(\sum_{k=0}^{n-1}\frac{(aq)^{k}}{(q;q)_{k}}\right), & \omega_n & =  \frac{(q;q)_{n-1}}{a^{n}} \sum_{k=0}^{n-1}\frac{(aq)^{k}}{(q;q)_{k}},
\end{align}
for all $n\in\N$, is
\begin{align}\label{eq:ASCmeas}
\rho_0 = (\nicefrac{q}{a};q)_\infty \sum_{n=0}^\infty  \frac{a^{-n} q^{n^2}}{(\nicefrac{q}{a};q)_n(q;q)_n} \delta_{aq^{-n}-1}.
\end{align}
\end{proposition}

\begin{proof}
Consider the probability measure\footnote{The fact that $\rho$ is a probability measure is due to M.\ E.\ H.\ Ismail~\cite[Theorem~5.1]{ism85}.}
\begin{align}\label{eq:ASCmeas0}
\rho = (aq;q)_\infty \sum_{n= 0}^\infty \frac{a^n q^{n^2}}{(aq;q)_n(q;q)_n} \delta_{q^{-n}-1}.
\end{align} 
It is known that the corresponding moment problem is indeterminate whenever $1<a<\nicefrac{1}{q}$ and, moreover, $\rho$ is an N-extremal solution~\cite{chi68,chr04}. The latter also means that polynomials are dense in $L^2(\R;\rho)$. The associated orthogonal polynomials of the first kind are expressed via the Al-Salam--Carlitz polynomials of type II, 
\begin{align*}
V_n^{(a)}(x;q) 
                       & = (-a)^n q^{-\binom{n}{2}} \sum_{k =0}^n \frac{(q^{-n};q)_k(x;q)_k}{(q;q)_k}q^{-\binom{k}{2}}\frac{q^{nk}}{(-a)^k}.
\end{align*}
Indeed, the orthogonality relations
\begin{align*}
(aq;q)_\infty\sum_{k= 0}^\infty \frac{a^k q^{k^2}}{(q;q)_k(aq;q)_k}V_m^{(a)}(q^{-k};q)V_n^{(a)}(q^{-k};q) = \frac{a^n(q;q)_n}{q^{n^2}}\delta_{n,m},
\end{align*}
and the recurrence relations
\begin{align*}
xV_n^{(a)}(x;q) = V_{n+1}^{(a)}(x;q) + (a+1)q^{-n}V_n^{(a)}(x;q) + aq^{1-2n}(1-q^n)V_{n-1}^{(a)}(x;q),
\end{align*}
together imply that the Jacobi parameters corresponding to the measure~\eqref{eq:ASCmeas0} are
\begin{align*}
a_n & = (a+1)q^{-n} - 1, & b_n & = \frac{\sqrt{a(1-q^{n+1})}}{q^{n+\nicefrac{1}{2}}},
\end{align*} 
and  the corresponding normalized orthogonal polynomials are explicitly  given by
\begin{align*}
P_n(x) & = \frac{1}{h_n}V_n^{(a)}(x+1;q), & h_n & = \sqrt{a^n(q;q)_n\, q^{-n^2}}\,.
\end{align*}
By taking into account that
\begin{align*}
V_n^{(a)}(1;q) = (-1)^n a^n q^{-\binom{n}{2}},
\end{align*}
we find the Krein string parameters (see~\cite[Section~5]{beva95}, \cite[f-la~(3.3)]{IndMoment}):
\begin{align*}
\wt{\omega}_{n+1}  & = \frac{-1}{b_{n} P_{n}(0)P_{n+1}(0)}  = - \frac{a^n(q;q)_n}{q^{n^2}V_n^{(a)}(1;q)V_{n+1}^{(a)}(1;q)} = \frac{(q;q)_n}{a^{n+1}},\\
 \wt{\ell}_{n+1} & = P_{n}(0)^2  = \frac{q^{n^2}}{a^n(q;q)_n}|V_n^{(a)}(1;q)|^2 = \frac{(aq)^n}{(q;q)_n},
\end{align*}
for all $n\in\N\cup\{0\}$. Therefore, the generalized indefinite string corresponding to the measure~\eqref{eq:ASCmeas0} is the Krein--Stieltjes string $(\tilde{L},\tilde{\omega})$ given by
\begin{align*}
\tilde{x}_n & = \sum_{k=1}^n \tilde{\ell}_k = \sum_{k=0}^{n-1}\frac{(aq)^{k}}{(q;q)_{k}}, & \tilde{\omega} & = \sum_{n=1}^\infty \tilde{\omega}_n\delta_{\tilde{x}_n},
\end{align*} 
with
\begin{align*}
\tilde{L} & = \lim_{n\to\infty}\tilde{x}_n = \sum_{k=0}^\infty\frac{(aq)^{k}}{(q;q)_{k}} = \frac{1}{(aq;q)_\infty}.
\end{align*}
If we extend this Krein string to $[0,\infty)$ by assigning zero mass to the interval $[\tilde{L},\infty)$, that is, we consider the new Krein string having infinite length and whose mass distribution on $[0,\infty)$ coincides with $\tilde{\omega}$, then, by Lemma~\ref{lem:GIStrunc}, its Weyl--Titchmarsh function is given by~\eqref{eq:WTtrunc}. In view of   
Theorem~\ref{th:KreinParam}, this gives rise to another N-extremal solution of the moment problem. Comparing the parameterizations in Theorem~\ref{th:KreinParam} and in~\cite[Proposition~1.7]{chr04},\footnote{Let us emphasize that in~\cite[Proposition~1.7]{chr04} the measures $\nu_0$ and $\nu_\infty$ are obtained via the Krein parameterization, which differs from the standard Nevanlinna parameterization, see~\cite[pp.~7-8]{chr04}.} we conclude that the corresponding spectral measure is given by~\eqref{eq:ASCmeas}. It remains to use the map~\eqref{eqnuitofa}, which transforms the new Krein string into the pair $(u,\mu)\in \CHdom^+$ given by~\eqref{eqnstrinfpureU} and \eqref{eq:ASCcoef}. 
\end{proof}

\begin{remark}
According to the above considerations, the peakon positions accumulate at a finite point from the left,
\begin{align}
\lim_{n\to \infty} x_n = L = - \log\big( (aq;q)_\infty\big).
\end{align}
Moreover, by construction, $\omega|_{[L,\infty)} = 0$. 
Notice that  
\begin{align}
\sum_{n\in\N}\lambda_n^{-p} = \sum_{n\in\N} \frac{1}{(aq^{1-n}-1)^p} <\infty
\end{align}
for all $p>0$ since $1<a<\nicefrac{1}{q}$.  
The above pair $(u,\mu)$ also satisfies~\eqref{eq:LIcond} since
\begin{align}
\sum_{n\in\N}n^2\omega_n = \sum_{n\in\N}n^2\frac{(q;q)_{n-1}}{a^{n}} \sum_{k=0}^{n-1}\frac{(aq)^{k}}{(q;q)_{k}}\le 
\frac{1}{(aq;q)_\infty} \sum_{n\in\N}\frac{n^2}{ a^{n}} <\infty,
\end{align}
and thus provides an explicitly solvable example for the class considered in~\cite{li09}.    
 \end{remark}

Our next goal is to present numerics for the  solution with initial data determined by~\eqref{eq:ASCcoef}. 
By taking $a=\nicefrac{6}{5}$, $q=\nicefrac{4}{5}$, we can obtain an explicit solution $u$ with Al-Salam--Carlitz infinite peakons according to the formulas~\eqref{eqnxomegaDelta} and~\eqref{form_u_interval}-\eqref{form_S}. 
 For convenience, we consider the corresponding Al-Salam--Carlitz infinite peakons with the rescaled spectral measure
\begin{align}
\tilde \rho_0(\ledot,t) =  \sum_{n=0}^\infty  \frac{a^{-n} q^{n^2}}{(\nicefrac{q}{a};q)_n(q;q)_n} e^{-\frac{t}{2(aq^{-n}-1)}}\delta_{aq^{-n}-1}
\end{align}
instead of $\rho_0$.
It follows from Remark \ref{rem:scaling} that the only distinction is that the positions of all peaks undergo the same shift $\log\big((\nicefrac{q}{a};q)_\infty\big)$.

The graphs of $u$ with the first several peaks are presented in Figure~\ref{fig:ASC3}; see Remark \ref{rem:num} and Remark \ref{rem:graph} for some key points of graphs.
  
\begin{figure}[h]
  \centering
  \includegraphics[width=1\textwidth]{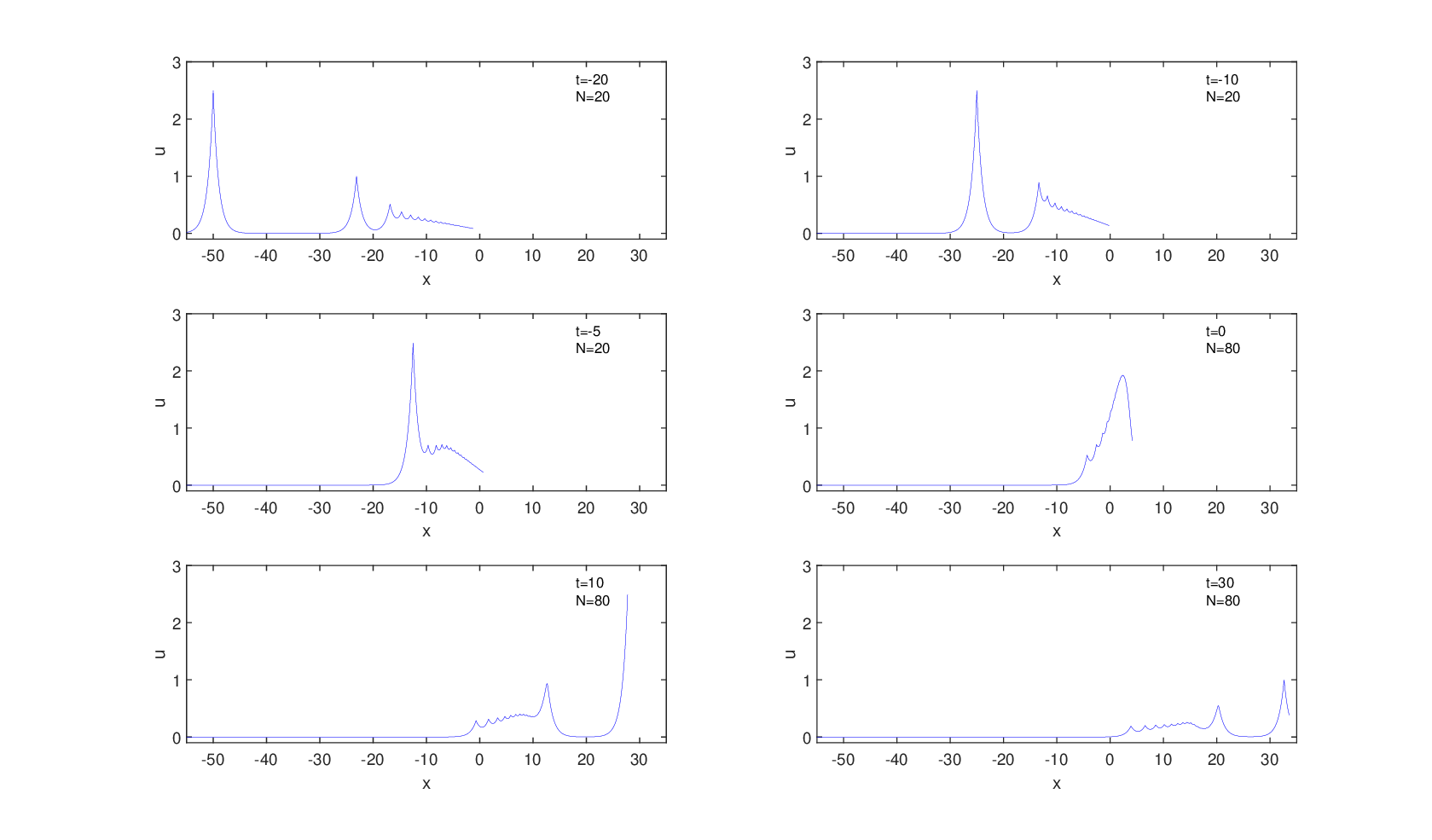}
  \caption{{\small Al-Salam--Carlitz infinite peakons with $a=\nicefrac{6}{5}$ and $q=\nicefrac{4}{5}$: Snapshots of $u(x,t)$ in the interval $[-55,x_{N+1}]$.
  }}\label{fig:ASC3} 
  \end{figure}

\section{Laguerre peakons}\label{sec:LaguerrePeaks}

\subsection{The Laguerre weight}\label{ss:Laguerre}

Consider the probability measure $\rho_{\gamma,\alpha}$ on $\R$ given by
\begin{align}\label{eq:LaguerreMeasure}
   \rho_{\gamma,\alpha}(d\lambda) =  \frac{1}{\Gamma(\gamma+1)}\id_{(\alpha,\infty)}(\lambda)(\lambda-\alpha)^\gamma\E^{-(\lambda-\alpha)} d\lambda,
\end{align}
where $\gamma>-1$ and $\alpha>0$ are parameters. 
The corresponding (orthonormal) orthogonal polynomials are given explicitly by\footnote{Here we use the standard notation for the generalized binomial coefficient \[ \binom{n+\gamma}{n} = \frac{(\gamma+1)_n}{\Gamma(n+1)}= \frac{\Gamma(n+\gamma+1)}{\Gamma(\gamma+1)\Gamma(n+1)},\] where $\Gamma$ is the classical gamma function.}
\begin{align}\label{eq:OPRLviaLag}
P_n(z) &= \frac{(-1)^n}{h_n} L_n^{(\gamma)}(z-\alpha), & h_n& = \binom{n+\gamma}{n}^{\nicefrac12},
\end{align}
for all $n\in\N\cup\{0\}$, 
where $L_n^{(\gamma)}$ are the Laguerre--Sonin polynomials~\cite[Section~5.1]{sze};
\begin{align}\label{eq:lagsonpol}
L_n^{(\gamma)}(z) 
=\sum_{k=0}^n\binom{n+\gamma}{n-k}\,\frac{(-z)^k}{k!}.
\end{align}
The Laguerre--Sonin polynomials satisfy the recurrence relations
\begin{align}
(n+1)L_{n+1}^{(\gamma)}(z) = (2n+\gamma+1-z)L_{n}^{(\gamma)}(z) - (n+\gamma)L_{n-1}^{(\gamma)}(z),
\end{align}
for all $n\in\N$ (see~\cite[Equation~(5.1.10)]{sze}), and hence the corresponding Jacobi parameters are given by
\begin{align}\label{eq:LagJacParam}
a_n & = 2n+1+\gamma + \alpha, & b_n & = \sqrt{(n+1)(n+\gamma+1)}.
\end{align}
On the other side, the generalized indefinite string with spectral measure $\rho_{\gamma,\alpha}$ is a Krein--Stieltjes string $(\infty,\wt\omega)$ with coefficients given by (see \cite[(3.3)]{IndMoment})
\begin{align}\label{eq:LaguerreCoeff}
\begin{split}
\wt{\omega}_{n+1} & = \frac{-1}{b_{n} P_{n}(0)P_{n+1}(0)} = \frac{(\gamma+1)_n}{(n+1)!L_n^{(\gamma)}(-\alpha)L_{n+1}^{(\gamma)}(-\alpha)}, \\
 \wt{\ell}_{n+1} & = P_{n}(0)^2 = \binom{n+\gamma}{n}^{-1}L_n^{(\gamma)}(-\alpha)^2,
\end{split}
\end{align}
for $n\in \N\cup\{0\}$. 
Finally, the pair $(u,\mu)$ in $\CHdom$ corresponding to the Laguerre spectral measure~$\rho_{\gamma,\alpha}$ is a pure infinite multi-peakon profile of the form~\eqref{eqnstrinfpure} with $x_n\rightarrow\infty$ and coefficients given explicitly by 
\begin{align}\label{eq:uLagatt=0}
x_n & = \log \left(\sum_{k=0}^{n-1} P_{k}(0)^2\right), & \omega_n & = - \frac{\sum_{k=0}^{n-1} P_{k}(0)^2}{b_{n-1}P_{n-1}(0)P_{n}(0)}, & \dip_n & = 0.
\end{align}
Using the asymptotic behavior of the Laguerre--Sonin polynomials, we can obtain the large $n$ asymptotics for these peakon parameters.

\begin{lemma}\label{lem:asympLaginn}
If $x_n$ and $\omega_n$ are given by~\eqref{eq:uLagatt=0} and~\eqref{eq:OPRLviaLag}, then
  \begin{align} \label{eqnMPLagCoeff02}
   x_n & = 4\sqrt{\alpha n}  + \log\biggl(\frac{\Gamma(\gamma+1)}{8\pi\alpha^{\gamma+1}\E^\alpha}\biggr) + \OO(n^{\nicefrac{-1}{2}}), & \omega_n & = \frac{1}{2\sqrt{\alpha n}} +\OO(n^{-1}), 
   \end{align} 
   as $n\to \infty$.
\end{lemma}

\begin{proof}
The Christoffel--Darboux formula~\eqref{eq:CDflaGen} for the Laguerre--Sonin polynomials reads~\dlmf{18.2.13}, \cite[Equation~(5.1.11)]{sze}:
\begin{align*}
\sum_{k=0}^{n-1} P_k(0)^2 = \frac{n!\Gamma(\gamma+1)}{\Gamma(n+\gamma)}\Big(L_{n}^{(\gamma)}(-\alpha)L_{n-1}^{(\gamma)'}(-\alpha) - L_{n}^{(\gamma)'}(-\alpha)L_{n-1}^{(\gamma)}(-\alpha)\Big).
\end{align*}
  Using the asymptotic behavior of the Laguerre polynomials (see~\cite[Theorem~8.22.3]{sze}), 
we readily compute by taking into account~\eqref{eq:uLagatt=0} that 
 \begin{align*}
\E^{x_n} & = \frac{n!\Gamma(\gamma+1)}{\Gamma(n+\gamma)}\Big(L_{n}^{(\gamma)}(-\alpha)L_{n-1}^{(\gamma)'}(-\alpha) - L_{n}^{(\gamma)'}(-\alpha)L_{n-1}^{(\gamma)}(-\alpha)\Big) \\
&  = \frac{\Gamma(\gamma+1)}{8\pi\alpha^{\gamma+1}\E^\alpha} \E^{4\sqrt{\alpha n}}(1+\OO(n^{\nicefrac{-1}{2}})),
\end{align*}
which yields the claimed asymptotics for $x_n$, and that 
\begin{align*}
\omega_n  & 
 = \frac{L_{n-1}^{(\gamma)'}(-\alpha) }{L_{n-1}^{(\gamma)}(-\alpha) } - \frac{L_{n}^{(\gamma)'}(-\alpha) }{L_{n}^{(\gamma)}(-\alpha) }
 =  \frac{1}{2\sqrt{\alpha n}} +\OO(n^{-1}).\qedhere
 \end{align*}
 \end{proof}

It follows from~\cite[Theorem~13.5]{ISPforCH} that $u-(4\alpha)^{-1}\in H^1[0,\infty)$, that is,
\begin{align}\label{eq:H1killipsimon}
\int_0^\infty \left(u(x)-\frac{1}{4\alpha}\right)^2 + u'(x)^2\, dx <\infty.
\end{align}
However, it is also possible to show that asymptotics like~\eqref{eqnMPLagCoeff02} imply~\eqref{eq:H1killipsimon}.

\begin{lemma}\label{lem:killipsimon}
 If $u$ has the form~\eqref{eqnstrinfpureU} with positions and heights satisfying
  \begin{align} \label{eqnMPLagCoeff03}
   x_n & = 4\sqrt{\alpha n} + c+ \OO(n^{-\eps}), & \omega_n & = \frac{1}{2\sqrt{\alpha n}}(1+ \OO(n^{-\eps})), 
   \end{align} 
  for some constants $c\in\R$ and $\eps>\nicefrac{1}{3}$, then $u-(4\alpha)^{-1}\in H^1[0,\infty)$.
\end{lemma}

\begin{proof}
By~\cite[Corollary~12.5]{ISPforCH} and also in view of~\eqref{eq:GISviaMP}, a function $u$ given by~\eqref{eqnstrinfpureU} satisfies~\eqref{eq:H1killipsimon} if and only if
 \begin{align}\label{eq:KSstrThA}
 \sum_{n\in\N} \big(1 - \E^{x_{n-1}-x_n}\big)^3 + \sum_{n\in\N} \big(1 - \E^{x_{n-1}-x_n}\big)\biggl(\E^{x_n} \sum_{k\ge n} \frac{\omega_k}{\E^{x_k}} - \frac{1}{4\alpha}\biggr)^2 <\infty.
 \end{align}
Taking into account that $x_n > x_{n-1}$ for all $n$, it is not difficult to see that the first series converges exactly when 
 \begin{align*}
 \sum_{n\in\N} (x_n - x_{n-1})^3  <\infty.
 \end{align*}
The latter clearly holds if $x_n$ satisfy~\eqref{eqnMPLagCoeff03} since in this case $x_n - x_{n-1} = \OO(n^{-\eps})$ as $n\to \infty$.
Next, taking into account that
\[
\int_{n}^\infty \frac{dx}{\sqrt{ x}\E^{4\sqrt{\alpha x}}} \le \sum_{k\ge n} \frac{1}{\sqrt{k}\E^{4\sqrt{\alpha k}}} \le \int_{n-1}^\infty \frac{dx}{\sqrt{ x}\E^{4\sqrt{\alpha x}}} 
\]
holds true for all $n\in\N$, we get from~\eqref{eqnMPLagCoeff03} the bound 
\[
 \OO(n^{-\eps})\le \E^{x_n} \sum_{k\ge n} \frac{\omega_k}{\E^{x_k}} - \frac{1}{4\alpha} \le \frac{1}{4\alpha}\biggl(\E^{\frac{2\sqrt{\alpha}}{\sqrt{n-1}}}(1+\OO(n^{-\eps})) - 1\biggr).
\]
 This immediately implies convergence of the second series in~\eqref{eq:KSstrThA}.
\end{proof}

\begin{remark}
Lemma~\ref{lem:killipsimon} implies that the spectral measure corresponding to an infinite multi-peakon profile with asymptotics~\eqref{eqnMPLagCoeff03} satisfies certain Szeg\H{o}-type conditions (see~\cite[Section~13.2]{ISPforCH}). 
However, its more refined structure (like the presence of a singular continuous spectrum) under the asymptotics~\eqref{eqnMPLagCoeff03} seems to be an interesting problem, even in the special case when $u$ is given explicitly by
\begin{align}
u(x) = \sum_{n\in\N} \frac{1}{\sqrt{n}}\E^{-|x-\sqrt{n}|}.
\end{align}
\end{remark}

\subsection{The global conservative solution and long-time asymptotics}\label{ss:LaguerreCH}

Consider now the solution 
\begin{align}\label{eq:LagFlow}
  (u(\ledot,t),\mu(\ledot,t)) = \Phi^t(u_0,\mu_0),
\end{align}
where the initial data $(u_0,\mu_0)$ is the pair in $\CHdom$ corresponding to the Laguerre measure $\rho_{\gamma,\alpha}$ defined in~\eqref{eq:LaguerreMeasure}. 
Since the moment problem for the measure $\rho_{\gamma,\alpha}$ is determinate, 
 we infer from Theorem~\ref{thm:MPevolt} and Corollary~\ref{corIP+} that this solution has the form
    \begin{align}\label{eq:LagFlowForm}
      \omega(\ledot,t) & =  \sum_{n=1}^{\infty} \omega_n(t) \delta_{x_n(t)}, & \dip(\ledot,t) & = 0,
    \end{align}
 with coefficients given by~\eqref{eq:3.9positive} and such that the increasing sequence of points $x_1(t),x_2(t),\ldots$ accumulates at $\infty$ for each fixed $t\in\R$.  
 We are next going to use these explicit formulas to derive long-time asymptotics for the peakon parameters of this solution. 

\begin{theorem}\label{thm:LaguerreMPs}
For every fixed $n\in\N$ the following asymptotics hold:
\begin{enumerate}[label=(\roman*), ref=(\roman*), leftmargin=*, widest=ii]
\item As $t\rightarrow\infty$ one has 
\begin{align}
\begin{split}
 \omega_n(t) & = \sqrt\frac{2}{t} + \oo(t^{\nicefrac{-1}{2}}), \\
  x_n(t) & = \sqrt {2t}+(n-\gamma-\nicefrac{3}{2})\log\sqrt{2 t}+\log\frac{2^{\gamma +1}\Gamma(\gamma+1)}{\sqrt{2\pi}\E^{\alpha}(n-1)!} +\oo(1).
 \end{split}
\end{align}
\item As $t\rightarrow-\infty$ one has 
\begin{align}
\begin{split}
\omega_n(t)& = \frac{1}{\alpha} + \oo(1), \\
 x_n(t) & = \frac{t}{2\alpha}+(2n-1+\gamma)\log \frac{|t|}{2\alpha} - \log\frac{\alpha^{1+\gamma}\Gamma(n+\gamma)(n-1)!}{\Gamma(\gamma+1)} + \oo(1).
\end{split}
\end{align}
\end{enumerate}
\end{theorem}

\begin{figure}[h]
  \centering
  \includegraphics[width=1\textwidth]{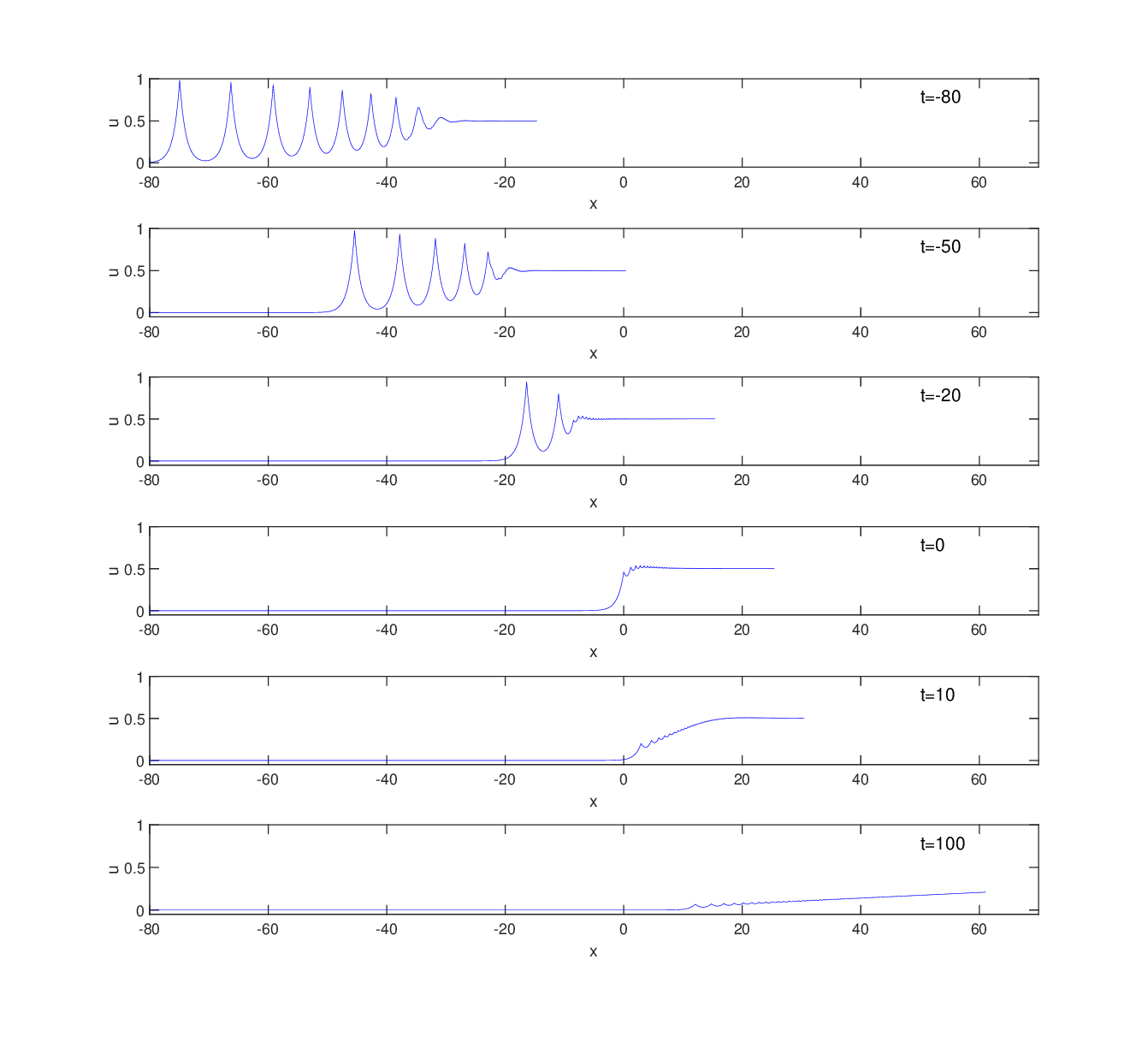}
  \caption{{\small Laguerre infinite peakons with  $\gamma=0$ and $\alpha=\nicefrac{1}{2}$:\\ Snapshots of $u(x,t)$ in the interval $[-80,x_{101}]$.
  }}\label{fig:lague7}
\end{figure}

\begin{remark}\label{rem:LagPeakons}
Despite the fact that every peakon has asymptotically the same height and moves at asymptotically the same speed as $t\to -\infty$, the distance between peaks grows logarithmically as $|t|\to \infty$. More precisely, for $n\in\N$ one has  
\begin{align}
|x_{n+1}(t) - x_n(t)|  = \log\sqrt{2t} - \log n + \oo(1),
\end{align}
as $t\to \infty$, and 
\begin{align}
|x_{n+1}(t)-x_n(t)|  = \log (t^2)  - \log (4\alpha^2n(n+\gamma)) + \oo(1),
\end{align}
as $t\to-\infty$. 

Let us also mention that the graphs of $u$ with the parameters $\gamma=0$ and $\alpha=\nicefrac{1}{2}$ are presented in Figure~\ref{fig:lague7}; see also~\cite{animations} for animations. We refer to Remark \ref{rem:num} and Remark \ref{rem:graph} for some key points regarding the corresponding numerics.
 \end{remark}

Our proof of this long-time behavior will be based on an analysis of the Hankel determinants $\Delta_{l,k}(t)$ introduced in Section~\ref{secMoment} with the moments explicitly given by
\begin{align}\label{eq:momentLag-t}
s_k(t)=\int_\R \lambda^k \E^{-\frac{t}{2\lambda}} \rho_{\gamma,\alpha}(d\lambda) 
= \frac{1}{\Gamma(\gamma+1)} \int_\alpha^\infty \lambda^k (\lambda-\alpha)^{\gamma} \E^{\alpha-\lambda-\frac{t}{2\lambda}} d\lambda,
\end{align}
for all $k\in\Z$. 
Since the measure $\rho_{\gamma,\alpha}$ is supported on $(0,\infty)$, the Hankel determinants $\Delta_{1,k}(t)$ are always positive  (see Remark~\ref{rem:formulas3.9}~\ref{rem:3.3b}) and hence $\kappa(n) = n-1$ for all $n\in\N$ and, moreover, 
\begin{align}\label{sol_xw}
x_n(t) & =\log \frac{\Delta_{2,n-1}(t)}{\Delta_{0,n}(t)}, & \omega_n(t) & =\frac{\Delta_{0,n}(t)\Delta_{2,n-1}(t)}{\Delta_{1,n}(t)\Delta_{1,n-1}(t)}.
\end{align}

\begin{remark} \label{rem:LaguerreRMT} 
Hankel determinants and random matrix models with the deformed measure
\begin{align}\label{Lag_meas_t}
\begin{split}
 \rho_{\gamma,\alpha}(d\lambda;t) &= \E^{-\frac{t}{2\lambda}}\rho_{\gamma,\alpha}(d\lambda) = \frac{1}{\Gamma(\gamma+1)}\id_{(\alpha,\infty)}(\lambda)(\lambda-\alpha)^{\gamma}\E^{\alpha-\lambda-\frac{t}{2\lambda}} d\lambda
\end{split}
\end{align}
 have attracted lots of interest recently. The corresponding problems include linear statistics of random variables \cite{chenits10}, eigenvalues of the Wigner--Smith time-delay matrix in the context of quantum
transport and electrical characteristics of chaotic cavities \cite{brou97,Mezzadri13,texier13}, and bosonic replica field theories \cite{osipov07}. Note that particular interest is given to the case of $\alpha=0$, which is referred to as the singularly perturbed Laguerre weight.  For studies of large $n$ and small $t$ asymptotics, see~\cite{Mezzadri13,lyu19,xu14cmp,xu15jat} for example.  
 \end{remark}

\begin{lemma}\label{lem:HankelLaguerre} 
For all $k$, $n\in\N\cup\{0\}$ the following asymptotics hold:
\begin{enumerate}[label=(\roman*), ref=(\roman*), leftmargin=*, widest=ii]
\item\label{itemHankLag+} As $t\rightarrow\infty$ one has 
\begin{align}\label{eq:DeltaLagAt+Infty}
\Delta_{k,n}(t) & \sim  \E^{-n\sqrt{2t}} \sqrt{2t}^{\frac{n(n+ 2k+2\gamma)}{2}} d_{k,n}^+, & 
d_{k,n}^+ & = \frac{(2\pi)^{\frac{n}{2}}\E^{\alpha n}}{2^{n(n+k+\gamma)}}\prod_{j=0}^{n-1} \frac{j!}{\Gamma(\gamma+1)}.
\end{align}
\item\label{itemHankLag-}
 As $t\rightarrow-\infty$ one has 
 \begin{align}\label{eq:DeltaLagAt-Infty}
\Delta_{k,n}(t) & \sim  \E^{n\frac{|t|}{2\alpha}} \biggl|\frac{t}{2\alpha}\biggr|^{-n(n+\gamma)} d_{k,n}^-, & d_{k,n}^- &= \alpha^{n(n+k+\gamma)} \prod_{j=0}^{n-1}(j!)^2\binom{\gamma+j}{j}.
\end{align}
\end{enumerate}
\end{lemma}

\begin{proof}
Recall the multiple integral expression for the Hankel determinants (see Remark~\ref{remDeltas}) 
\begin{align}\label{eq:HankLagIntegr}
\begin{split}
\Delta_{k,n}(t) & = \frac{1}{n!}\int_{\R_{>\alpha}^n} \mlam^k\wand_n(\mlam) \rho_{\gamma,\alpha}(d\mlam;t) \\
& = \frac{\E^{\alpha n}}{n!\, \Gamma(\gamma+1)^n} \int_{\R_{>\alpha}^n} \mlam^k {\rm R}_{\gamma}(\mlam-\mlam_\alpha)\wand_n(\mlam)  \E^{-\sum_{i=1}^n\lambda_i+\frac{t}{2\lambda_i}}\, d\mlam,
\end{split}\end{align}
where ${\rm R}_\gamma(\mlam) = \prod_{i=1}^n|\lambda_i|^\gamma$ and $\mlam_\alpha = (\alpha,\dots,\alpha)$.


{\em \ref{itemHankLag+} Asymptotics as $t\rightarrow\infty$.} First, instead of~\eqref{eq:HankLagIntegr}, we may consider the integral 
\begin{align}\label{eq:HankLagIntegr02}
I_{k,n}(t):= \frac{\E^{\alpha n}}{n!\, \Gamma(\gamma+1)^n}\int_{\R_{>0}^n} \mlam^k {\rm R}_{\gamma}(\mlam-\mlam_\alpha)\wand_n(\mlam) \E^{-\sum_{i=1}^n\lambda_i+\frac{t}{2\lambda_i}}\, d\mlam,
\end{align}
by noticing that $\Delta_{k,n}(t) = I_{k,n}(t) + \OO(\E^{-{t}/{2\alpha}})$ as $t\to\infty$ since
\begin{align*}
\int_{\Omega_n(\alpha)} \mlam^k {\rm R}_{\gamma}(\mlam-\mlam_\alpha)\wand_n(\mlam)  \E^{-\sum_{i=1}^n\lambda_i+\frac{t}{2\lambda_i}}\, d\mlam \le C_n\E^{-\frac{t}{2\alpha}}
\end{align*}
for all $t>0$, where $\Omega_n(\alpha) = \R_{>0}^n\backslash \R_{>\alpha}^n$ and
\begin{align*}
C_n = \int_{\R_{>0}^n} \mlam^k {\rm R}_{\gamma}(\mlam-\mlam_\alpha) \wand_n(\mlam)  \E^{-\sum_{i=1}^n\lambda_i}\, d\mlam .
\end{align*}
After a change of variables $\mlam\rightarrow \sqrt{t/2}\,\mlam$ in~\eqref{eq:HankLagIntegr02}, we get
\begin{align*}
I_{k,n}(2t^2)
& = t^{n^2 +n(k+\gamma)}\frac{\E^{\alpha n}}{n!\, \Gamma(\gamma+1)^n} \int_{\R_{>0}^n} \mlam^k {\rm R}_{\gamma}(\mlam-\mlam_{\alpha/t}) \wand_n(\mlam)\E^{-tg(\mlam)}\,d\mlam,
\end{align*}
where the function $g\colon \R_{> 0}^n \to \R_{\ge 0}$ is given by
\begin{align*}
g(\mlam) = \sum_{i=1}^n\lambda_i+\frac{1}{\lambda_i},\qquad \mlam \in \R_{> 0}^n.
\end{align*}
Clearly, the function $g$ attains a global minimum on $\R_{>0}^n$ at $\mlam_1 = (1,\dots,1)$ with $g(\mlam_1) = 2n$. Moreover, it admits the asymptotic expansion
\begin{align*}
g(\mlam) = 2n+ |\mlam - \mlam_1|^2 + \OO(|\mlam - \mlam_1|^3)
\end{align*}
near $\mlam_1$ and hence applying Laplace's method (see~\cite[Section~4.6]{deBr} or~\cite{mil} for example), we  get
\begin{align}\label{eq:LagEstGauss01}
I_{k,n}(2t^2)
& \sim \E^{-2nt}t^{n^2 +n(k+\gamma)}\frac{\E^{\alpha n}}{n!\, \Gamma(\gamma+1)^n} \int_{\R_{>0}^n} \wand_n(\mlam) \E^{- t|\mlam - \mlam_1|^2}\,d\mlam
\end{align}
as $t\to\infty$. 
Indeed, let $B_\delta(\mlam_1)$ be the ball of radius $\delta \in (0,1)$ centred at $\mlam_1$. 
Then, assuming $\gamma>-1$ fixed and  $t>\alpha$ sufficiently large, one has   
\begin{align}\label{eq:LagEstGauss02}
\E^{2nt}\int_{\R_{>0}^n\setminus B_\delta(\mlam_1)}\mlam^{\mulk} {\rm R}_{\gamma}(\mlam-\mlam_{\alpha/t}) \wand_n(\mlam)\E^{-t g(\mlam)}\,d\mlam \lesssim_\delta \E^{-t\eta(\delta)}
\end{align}
for all $\eta(\delta) \in (0,\frac{\delta^2}{1+\delta})$ since 
$\inf_{\lambda\in \R_{>0}\setminus (1-\delta,1+\delta)} \lambda +\frac{1}{\lambda}  \ge 2+ \frac{\delta^2}{1+\delta}$. 
On the other hand, for any $\eps>0$, we can find $\delta>0$ such that $|g(\mlam) -2n - |\mlam - \mlam_1|^2| < \eps|\mlam - \mlam_1|^2$ for all $\mlam\in B_\delta(\mlam_1)$. Moreover, taking into account that
\begin{align}\label{eq:estRlamviadelta}
(1-2\delta)^{k+|\gamma|}\le \lambda^k|\lambda-\alpha/t|^\gamma \le \frac{(1+\delta)^k}{(1-2\delta)^{|\gamma|}}
\end{align}
whenever $|\lambda - 1|\le \delta$ and $t>\alpha/\delta$, we get that 
\begin{align*}
c_1^n\int_{B_\delta(\mlam_1)}\mlam^{\mulk} {\rm R}_{\gamma}(\mlam-\mlam_{\alpha/t}) & \wand_n(\mlam)\E^{-t|\mlam - \mlam_1|^2 (1+\eps)} \,d\mlam
   \le \E^{2nt} \int_{B_\delta(\mlam_1)} \wand_n(\mlam)\E^{-t g(\mlam)}\,d\mlam \\
&   \le c_2^n\int_{B_\delta(\mlam_1)}\mlam^{\mulk} {\rm R}_{\gamma}(\mlam-\mlam_{\alpha/t}) \wand_n(\mlam)\E^{-t|\mlam - \mlam_1|^2 (1-\eps)}\,d\mlam
\end{align*}
holds true for all sufficiently large $t>\alpha/\delta$. Here $c_1$ and $c_2$ are the constants on the left-hand side and, respectively, the right-hand side of~\eqref{eq:estRlamviadelta}. 
Taking into account the estimate~\eqref{eq:LagEstGauss02}, we can approximate all three integrals by replacing the domain of integration by $\R_{>0}^n$ and the order of approximation is exponentially small (actually, it is of order $\OO(\E^{-t\tilde{\eta}})$ with some $\tilde{\eta} = \tilde{\eta}(\delta)>0$). Therefore, sending $\eps$ to zero we end up with~\eqref{eq:LagEstGauss01}.
Finally, it is immediate to see that 
\begin{align*}
 \int_{\R^n\setminus\R_{>0}^n} \wand_n(\mlam) \E^{- t |\mlam - \mlam_1|^2} \,d\mlam \lesssim_n \E^{-t(1-\epsilon)}
\end{align*}
for all $t>0$ and any $\epsilon\in (0,1)$. 
It remains to employ the known result for the Hankel determinant with the Gaussian weight (see~\cite[Equation~(17.6.7)]{mehta} for example)
\begin{equation*}
\int_{\R^n} \wand_n(\mlam) \E^{-t|\mlam|^2}d\mlam 
 =  t^{-\frac{n^2}{2}} \pi^{\frac{n}{2}} 2^{\frac{n(1-n)}{2}}\prod_{j=1}^{n}j!\, . 
\end{equation*}


{\em \ref{itemHankLag-}  Asymptotics as $t\rightarrow-\infty$.}
 Taking into account the sign change of the variable $t$, the main contribution to the asymptotics of the integral~\eqref{eq:HankLagIntegr} now arises from a neighborhood of $\lambda = \alpha$ since the function $\lambda\mapsto \frac{1}{\lambda}$ has its global maximum on $[\alpha,\infty)$ at the left endpoint. Indeed, one may write~\eqref{eq:HankLagIntegr} as 
 \begin{align*}
 \Delta_{k,n}(2t) &= \frac{\E^{\alpha n}}{n!\, \Gamma(\gamma+1)^n} \int_{\R_{>\alpha}^n} \mlam^k {\rm R}_{\gamma}(\mlam-\mlam_\alpha) \wand_n(\mlam)\E^{-\sum_{i=1}^n\lambda_i+\frac{t}{\lambda_i}}\,d\mlam\\ 
 &= \E^{\frac{n}{\alpha }|t|}\frac{\E^{\alpha n}}{n!\, \Gamma(\gamma+1)^n} \int_{\R_{>\alpha}^n} \Phi(\mlam)\E^{-|t|g(\mlam)}\,d\mlam,
  \end{align*}
 where the functions $\Phi$ and $g$ are given by 
  \begin{align*}
 \Phi(\mlam) & = \mlam^k {\rm R}_{\gamma}(\mlam-\mlam_\alpha) \wand_n(\mlam)\E^{-\sum_{i=1}^n {\lambda_i}}, & g(\mlam) & = \frac{n}{\alpha} - \sum_{i=1}^n \frac{1}{\lambda_i}.
  \end{align*}
  The function $g\colon [\alpha,\infty)^n \mapsto [0,\infty)$ attains its global minimum  at $\mlam_\alpha = (\alpha,\dots,\alpha)$ with $g(\mlam_\alpha) = 0$. Moreover, one has 
    \begin{align*}
  g(\mlam)  = \sum_{i=1}^n \frac{\lambda_i - \alpha}{\alpha^2} + \OO(|\mlam - \mlam_\alpha|^2)
  \end{align*}
  near $\mlam_\alpha$. Arguing again as in~\cite[Chapter~4]{deBr}, \cite{mil}, we conclude that 
   \begin{align*}
  \int_{\R_{>\alpha}^n} \Phi(\mlam)\E^{-|t|g(\mlam)}\,d\mlam \sim \int_{\R_{>\alpha}^n} \Phi(\mlam)\E^{-|t|\sum_{i=1}^n \frac{\lambda_i - \alpha}{\alpha^2}}\,d\mlam,
  \end{align*}
  and thus it remains to estimate the last integral. However, we have 
  \begin{align*}
  \int_{\R_{>\alpha}^n}  \Phi(\mlam) &\E^{-|t|\sum_{i=1}^n \frac{\lambda_i - \alpha}{\alpha^2}}\,d\mlam 
   =   \int_{\R_{>\alpha}^n} \mlam^k {\rm R}_{\gamma}(\mlam-\mlam_\alpha) \wand_n(\mlam)\E^{-\sum_{i=1}^n |t|\frac{\lambda_i-\alpha}{\alpha^2} + \lambda_i} \,d\mlam\\
  & =  \E^{-\alpha n} \int_{\R_{>0}^n} (\mlam+\mlam_\alpha)^k \mlam^\gamma \wand_n(\mlam)\E^{-\sum_{i=1}^n |t|\frac{\lambda_i}{\alpha^2} + \lambda_i} \,d\mlam\\ 
  & =  \E^{-\alpha n}\biggl(\frac{\alpha^2}{|t|}\biggr)^{n^2 + n\gamma}\int_{\R_{>0}^n} \biggl(\frac{\alpha^2}{|t|}\mlam+\mlam_\alpha\biggr)^k \mlam^\gamma \wand_n(\mlam)\E^{ -\sum_{i=1}^n \lambda_i + \frac{\alpha^2}{|t|}\lambda_i} \,d\mlam\\
  & \sim  |t|^{-n(n+\gamma)}\frac{\alpha^{2n^2+n(k+2\gamma)}}{\E^{\alpha n}}\int_{\R_{>0}^n} \mlam^\gamma \wand_n(\mlam)\E^{ -\sum_{i=1}^n \lambda_i} \,d\mlam.
  \end{align*}
  Here, in the third equality, we simply changed the variables $\mlam \to \frac{\alpha^2}{|t|}\mlam$ and the last equality follows from dominated convergence. 
It remains to use the formula 
(see~\cite[Equation~(17.6.5)]{mehta} for example)
\begin{equation}\label{eq:LaguerreHankelEXPL}
\int_{\R_{>0}^n} \mlam^\gamma \wand_n(\mlam)\E^{-\sum_{i=1}^n\lambda_i}d\mlam= \prod_{j=1}^{n}j!\,\Gamma(\gamma+j).\qedhere
\end{equation}
\end{proof}

Now we are in position to prove the main result of this subsection.

\begin{proof}[Proof of Theorem~\ref{thm:LaguerreMPs}]
 As $t\to\infty$, we get from~\eqref{sol_xw} and~\eqref{eq:DeltaLagAt+Infty}  that 
\begin{align*}
\E^{x_n(t)} = \frac{\Delta_{2,n-1}(t)}{\Delta_{0,n}(t)} & =  \frac{\E^{-(n-1)\sqrt{2t}} \sqrt{2t}^{\frac{(n-1)^2}{2} +(n-1)(2+\gamma)}}{\E^{-n\sqrt{2t}} \sqrt{2t}^{\frac{n^2}{2} +n\gamma}} \frac{d_{2,n-1}^+}{d_{0,n}^+} (1+\oo(1)) \\
& = \E^{\sqrt{2t}} \sqrt{2t}^{\frac{2n-3-2\gamma}{2}} \frac{d_{2,n-1}^+}{d_{0,n}^+} (1 +\oo(1)) 
\end{align*}
as well as 
\begin{align*}
\omega_n(t) & =\frac{\Delta_{0,n}(t)\Delta_{2,n-1}(t)}{\Delta_{1,n}(t)\Delta_{1,n-1}(t)} 
 =  \frac{1}{\sqrt{2t}}\frac{d_{0,n}^+d_{2,n-1}^+}{d_{1,n}^+d_{1,n-1}^+} (1+\oo(1)).
\end{align*}
The asymptotics as $t\to-\infty$ follow in a similar way.
%
\end{proof}


\begin{remark}\label{rem:longtimeU}
 Since the function $u$ can be recovered in terms of the Hankel determinants as well (recall Remark~\ref{rem:UatInfty}), one can also use Lemma~\ref{lem:HankelLaguerre} to obtain long-time asymptotics for $u(x,t)$ in regions with $x\leq x_n(t)$ for a fixed $n\in\N$.   
 \end{remark}

\subsection{Perturbed Laguerre weights}\label{ss:LaguerrePrtb}

In fact, the arguments used above enable us to extend the long-time asymptotic results to a wider setting.
 More specifically, let $\rho_{\gamma,\alpha}$ be given by~\eqref{Lag_meas_t} with $\gamma>-1$ and $\alpha>0$. Consider the measures 
\begin{equation} \label{geLag_meas_t}
\rho(d\lambda;t) =  h(\lambda)\rho_{\gamma,\alpha}(d\lambda;t),
\end{equation}
where $h\colon [\alpha,\infty)\to (0,\infty)$ is a continuous strictly positive function such that $h(\infty):=\lim_{\lambda\rightarrow\infty}h(\lambda)>0$. Let $\Phi^t(u_0,\mu_0)$ denote the integral curve corresponding to \eqref{geLag_meas_t}, that is, consider the solution defined by~\eqref{eq:LagFlow}, where the initial data $(u_0,\mu_0)$ is the pair in $\CHdom$ corresponding to the measure~\eqref{geLag_meas_t} with $t=0$. 
Since the moment problem for the measure $\rho_{\gamma,\alpha}$ is determinate, 
by Theorem~\ref{thm:MPevolt} and Corollary~\ref{corIP+} this solution has the form~\eqref{eq:LagFlowForm}
 with coefficients given by~\eqref{eq:3.9positive} and such that the increasing sequence of points $x_1(t),x_2(t),\ldots$ accumulates at $\infty$ for each fixed $t\in\R$. Also, due to our assumptions on $h$, \cite[Theorem~13.5]{ISPforCH} implies that $u-(4\alpha)^{-1}\in H^1[0,\infty)$ for all $t\in\R$. Regarding the long-time behavior, one can prove the following claim.

\begin{theorem}\label{th:LagPertrbd}
For every fixed $n\in\N$ the following asymptotics hold:
\begin{enumerate}[label=(\roman*), ref=(\roman*), leftmargin=*, widest=ii]
\item As $t\rightarrow\infty$ one has 
\begin{align}
\begin{split}
 \omega_n(t) & = \sqrt\frac{2}{t}  + \oo(t^{\nicefrac{-1}{2}}), \\
  x_n(t) &  = \sqrt {2t} + (n-\gamma-\nicefrac{3}{2})\log\sqrt{2 t} + \log\frac{2^{\gamma +1}\Gamma(\gamma+1)}{\sqrt{2\pi}\E^{\alpha} h(\infty)(n-1)!} + \oo(1), 
  \end{split}
\end{align}
\item As $t\rightarrow-\infty$ one has 
\begin{align}
\begin{split}
\omega_n(t) & = \frac{1}{\alpha} + \oo(1), \\
 x_n(t) & = \frac{t}{2\alpha}+(2n-1+\gamma)\log \frac{|t|}{2\alpha} - \log\frac{\alpha^{1+\gamma}h(\alpha)\Gamma(n+\gamma)(n-1)!}{\Gamma(\gamma+1)} +\oo(1).
 \end{split}
\end{align}
\end{enumerate}
\end{theorem}

\begin{remark}
A few remarks are in order.
\begin{enumerate}[label=(\roman*), ref=(\roman*), leftmargin=*, widest=iii]
\item
   Comparing these long-time asymptotics with the ones in Theorem~\ref{thm:LaguerreMPs} for the unperturbed Laguerre weight, one can notice that the function $h$ only contributes as a shift in the spatial variable by $-\log h(\infty)$ as $t\to\infty$ and by $-\log h(\alpha)$ as $t\to-\infty$.
\item 
   As before, the proof is based on an analysis of the Hankel determinants and a careful inspection of the proof of Theorem~\ref{thm:LaguerreMPs} and Laplace's method immediately shows that only the behavior of $h$ near the endpoints of the interval $[\alpha,\infty)$ matters. In particular, one may replace the continuity assumption by a less restrictive boundedness-type assumptions on $h$. 
\item 
   One can use the results of~\cite[Section~12]{ISPforCH} to obtain a complete characterization of spectral measures of pairs in $\CHdom^+$ such that $u-(4\alpha)^{-1}\in H^1[0,\infty)$. Moreover, this class of pairs is invariant under the action of the conservative Camassa--Holm flow (see, for example,~\cite[Section~13.3]{ISPforCH} and also Remark~13.23~(c) there). It seems to be a difficult but at the same time interesting problem to analyse the long-time behavior of solutions with initial data given by such pairs. Indeed, these spectral measures allow arbitrary singular spectrum on $[\alpha,\infty)$ as well as infinitely many eigenvalues in the interval $(0,\alpha)$, which may accumulate only at $\alpha$.   
\end{enumerate}
\end{remark}

\section{Jacobi peakons}\label{sec:JacobiPeaks}

\subsection{The Jacobi weight}\label{ss:Jacobi}

Another particular example is related to the classical Jacobi polynomials, for which we again need to perform a shift in the spectral parameter in order to ensure that zero does not belong to the support of the spectral measure.
Namely,  for constants $a$, $b>-1$ and $\alpha > 0$, set
\begin{align}\label{eq:JacobWeight}
w_\alpha^{(a,b)}(\lambda)= (\lambda-\alpha)^b (1+\alpha- \lambda)^a\id_{(\alpha,1+\alpha)}(\lambda),\qquad \lambda\in\R,
\end{align} 
and define the measure on $\R$ by setting
\begin{align}\label{eq:JacobSpMeas}
\rho_\alpha^{(a,b)}(d\lambda) = w_\alpha^{(a,b)}(\lambda)d\lambda.
\end{align} 

\begin{remark}
Using the transformations in Remark~\ref{rem:scaling}, one can easily extend the results to more general weights of the form 
\begin{align}\label{eq:JacobWeightGen}
w(\lambda) = C(\lambda-\tilde{\alpha})^b (\tilde{\beta} - \lambda)^a\id_{(\tilde{\alpha},\tilde{\beta})}(\lambda),\qquad \lambda\in\R,
\end{align} 
where $C>0$, $a,b> -1$ and $0<\tilde{\alpha}<\tilde{\beta}$. Indeed, applying first the rule described in (b) with $c = (\tilde{\beta} - \tilde{\alpha})^{-1}$ to the  weight~\eqref{eq:JacobWeight} and then choosing $\alpha$ such that $\tilde{\alpha} = \alpha (\tilde{\beta} - \tilde{\alpha})$ it remains to apply (a) with a suitable constant. 
 \end{remark}

The corresponding orthogonal polynomials $p_n$,
normalized by
\begin{align}\label{eq:jacobi-normal}
p_n(1+\alpha) = \binom{n+a}{n} = \frac{(a+1)_n}{n!}
\end{align}
for all $n\in \N\cup\{0\}$ 
are related to the classical  {\em Jacobi polynomials} and can be expressed as a terminating hypergeometric series:
\begin{align}
p_n(z) & = P_n^{(a,b)}(2(z-\alpha) - 1) = \frac{(a+1)_n}{n!}\hyp21{-n,n+a+b+1}{a+1}{1+\alpha -z }.
\end{align}
The Jacobi polynomials are given by Rodrigues' formula \cite[(4.3.1), (4.3.2)]{sze}
\begin{align}
P_n^{(a,b)}(z) &=\sum_{k=0}^n \binom{n+a}{n-k} \binom{n+b}{k} \left(\frac{z-1}{2}\right)^k\left(\frac{z+1}{2}\right)^{n-k},\label{K5}
\end{align}
which immediately implies
\begin{align}\label{eq:JPsymmetry}
P_n^{(a,b)}(-z) = (-1)^nP_n^{(b,a)}(z),
\end{align}
and hence, in particular, that 
\begin{align}
P_n^{(a,b)} (-1) =(-1)^n \binom{n+b}{n} = (-1)^n\,\frac{(b+1)_n}{n!}\,.
\end{align}
The $L^2$ norm of $p_n$ is given by (see~\cite[(4.3.3)]{sze})
\begin{align}\begin{split}
k_n^2:= \int_{\alpha}^{1+\alpha} |p_n(\lambda)|^2 w_\alpha^{(a,b)}(\lambda)d\lambda & = \frac{1}{2^{a+b+1}}\int_{-1}^1 |P_n^{(a,b)}(\lambda)| (1+\lambda)^a (1-\lambda)^b d\lambda\\
& = \frac{1}{2n+a+b+1}\, \frac{\Gamma(n+a+1) \Gamma(n+b+1)}{n!\,\Gamma(n+a+b+1)}.
\end{split}\end{align}
Using the recurrence formula for the Jacobi polynomials~\cite[Equation~(4.5.1)]{sze}, we get 
\begin{align}\begin{split}
 & 2n(n+a+b)(2n+a+b-2)P_n^{(a,b)}(z)  \\
  & \qquad =  - 2(n+a-1)(n+b-1)(2n+a+b)P_{n-2}^{(a,b)}(z) \\
 & \qquad\qquad + (2n+a+b-1)\big(z(2n+a+b)(2n+a+b-2) + a^2 - b^2\big)P_{n-1}^{(a,b)}(z),
\end{split}\end{align}
which holds true for all $n\in\N$ with $n\geq 2$, we find the Jacobi parameters
\begin{align}\label{eq:JacobiBn2}
\begin{split}
b_{n-1} 
& =  \frac{(n+a)(n+b)}{(2n+a+b)(2n+a+b+1)}. 
\end{split}
\end{align}

\begin{remark}
Jacobi polynomials include the Chebyshev polynomials (where the case $a=b=-1/2$ corresponds to the first kind; $a=b=1/2$ to the second kind, etc.),  the Legendre polynomials (when $a=b=0$), and the ultraspherical (Gegenbauer) polynomials (when $a=b$); see \cite{dlmf}, \cite{sze} for further details. 
 \end{remark}


\begin{lemma}\label{lem:asympJacobinn2}
The pair $(u,\mu)$ in $\CHdom$ corresponding to the spectral measure $\rho_\alpha^{(a,b)}$ is a pure infinite multi-peakon profile of the form~\eqref{eqnstrinfpure} with $\dip\equiv 0$ and 
  \begin{align} \label{eq:MPJacobiW}
    \omega_n & = \frac{1}{\sqrt{\alpha(1+\alpha)}} + \oo(1),\\ 
    x_n & = (2n+a+b) \log\bigl(1+2\alpha +\sqrt{\alpha(1+\alpha)}\bigl) - \log\bigl(8\pi\alpha^{b+1}(1+\alpha)^{a+1}\bigl) +\oo(1),  \label{eq:MPJacobiX}
   \end{align} 
   as $n\to \infty$.
\end{lemma}

\begin{proof}
Since the spectral measure $\rho_\alpha^{(a,b)}$ has compact support, the corresponding moment problem is determinate, which implies the first claim. Moreover, the support is contained in $(0,\infty)$ and hence $\dip\equiv 0$ (see Remark~\ref{rem:formulas3.9}~\ref{rem:3.3b}). The asymptotic behavior follows from the Christoffel--Darboux formula~\eqref{eq:CDflaGen} (see also~\dlmf{18.2.13}, \cite[Equation~(I.2.5)]{akh}),
by employing the asymptotic behavior of the Jacobi polynomials (see~\cite[Theorem~8.21.7]{sze})
 \begin{align*}
 P_n^{(a,b)}(z)  =\frac{1}{\sqrt{2\pi n}} \frac{ \bigl(\sqrt{z-1}+\sqrt{z+1} \bigr)^{a+b}}{(z-1)^{a/2}(z+1)^{b/2}}\frac{\bigl(z+\sqrt{z^2-1}\bigr)^{n+1/2}}{(z^2-1)^{1/4}}(1+\oo(1)) ,
 \end{align*}
as $n\to \infty$, which holds true locally uniformly in $\C\backslash[-1,1]$. Indeed, we then readily compute by using~\eqref{eq:XnviaCD} that 
 \begin{align*}
\E^{x_n}   & = \frac{2b_{n-1}}{k_{n-1}k_n} 
                   \bigl(P_{n}^{(a,b)'}(z) P_{n-1}^{(a,b)}(z) - P_{n-1}^{(a,b)'}(z) P_{n}^{(a,b)}(z)\bigr)\big|_{z=-2\alpha-1}\\
                   & = \frac{2b_{n-1}}{k_{n-1}k_n} 
                   \bigl(P_{n}^{(b,a)'}(z) P_{n-1}^{(b,a)}(z) - P_{n-1}^{(b,a)'}(z) P_{n}^{(b,a)}(z)\bigr)\big|_{z=2\alpha+1}\\
               & \sim  \frac{2b_{n-1}}{k_{n-1}k_n} \frac{2^{a+b}}{2\pi\sqrt{n(n-1)}}\frac{ \bigl(z+\sqrt{z^2-1}\bigr)^{2n+a+b}}{(z-1)^{b+1}(z+1)^{a+1}}\bigg|_{z=2\alpha+1} \\
               & = \frac{2b_{n-1}}{k_{n-1}k_n}\frac{1}{8\pi\sqrt{n(n-1)} }\frac{ \big(2\alpha+1+2\sqrt{\alpha(1+\alpha)} \big)^{2n+a+b}}{\alpha^{b+1}(1+\alpha)^{a+1}},
                \end{align*}
                and it remains to notice that 
         \begin{align*}
         \frac{2b_{n-1}}{k_{n-1}k_n}  = n+\oo(n)
         \end{align*}       
         as $n\to \infty$ (see~\dlmf{5.11.12}). 
                Finally, using~\eqref{eq:omegaviaCD}, we get
 \begin{align*} 
\omega_n &  =  2\frac{P_{n}^{(a,b)'}(-1-2\alpha) }{P_{n}^{(a,b)}(-1-2\alpha) } - 2\frac{P_{n-1}^{(a,b)'}(-1-2\alpha) }{P_{n-1}^{(a,b)}(-1-2\alpha) } 
         =  \frac{1}{\sqrt{\alpha(1+\alpha) }} +\oo(1). \qedhere
 \end{align*}
\end{proof}

\begin{remark}
Notice that the pair $(u,\mu)$ corresponding to the Jacobi weight behaves asymptotically like a half-periodic multi-peakon profile
\begin{align}\label{eq:u-1period}
u_0(x) = \frac{1}{2}\sum_{n\in\N}\frac{1}{\sqrt{\alpha(1+\alpha)}} \E^{-|x-cn-\ell|}, 
\end{align}
with the constants $c$, $\ell$ given by 
\begin{align}
 c & = 2\log\bigl(1+2\alpha +\sqrt{\alpha(1+\alpha)}\bigr), & \ell & = \log\frac{\bigl(1+2\alpha +\sqrt{\alpha(1+\alpha)}\bigr)^{a+b}}{8\pi\alpha^{b+1}(1+\alpha)^{a+1}}.
\end{align}
To a certain extent this might not be too surprising because the spectral measure of the pair $(u_0,\mu_0)$ with $u_0$ given by~\eqref{eq:u-1period} and $\dip_0\equiv 0$ is explicitly given by (one may use~\cite{InvPeriodMP} for this purpose)
 \begin{align}\label{eq:rho0-periodic}
\rho(d\lambda) =  \frac{\sinh(c/2)}{\pi\E^{\ell+\nicefrac{c}{2}}}\id_{[\lambda_1,\lambda_2]}(\lambda)\frac{\sqrt{(\lambda - \lambda_1)(\lambda_2-\lambda)}}{\lambda^2 }d\lambda,
\end{align}
where 
\begin{align}
\lambda_1 & =\tanh(c/4)\sqrt{\alpha(1+\alpha)}, & \lambda_2 & =\coth(c/4)\sqrt{\alpha(1+\alpha)}.
\end{align}
\eat{{\tiny 
In the case when $\frac{1}{\sqrt{\alpha(1+\alpha)}} = 1$ and if summation is over $\Z_{\ge 0}$, 
\begin{align}\label{eq:rho-example1}
\rho(d\lambda) =  \frac{\E^{\frac{c}{2}-\ell}\sinh(c/2)}{\pi}\id_{[\lambda_1,\lambda_2]}(\lambda)\frac{\sqrt{(\lambda - \lambda_1)(\lambda_2-\lambda)}}{\lambda^2 }d\lambda
\end{align}
where 
\begin{align*}
\lambda_1 & = \tanh(c/4), & \lambda_2 & = \coth(c/4). 
\end{align*}
}}
 Clearly, this measure is of the form 
\begin{align}
\rho(d\lambda) = h(\lambda)w(\lambda)d\lambda,
\end{align}
where $h$ is smooth and strictly positive on $[\lambda_1,\lambda_2]$ and $w$ is given by~\eqref{eq:JacobWeightGen} with an obvious choice of parameters. 
 \end{remark}


\subsection{The global conservative solution and long-time asymptotics}\label{ss:JacobiCH}

Let us now consider the infinite peakon solution for the initial data given by the Jacobi measure in~\eqref{eq:JacobSpMeas}. 
More specifically, the pair $(u(\ledot,t),\mu(\ledot,t))$ corresponds to the measure given by 
\begin{align}\label{Jac_meas_t}
\rho_{\alpha}^{(a,b)}(d\lambda;t) &=  \E^{-\frac{t}{2\lambda}} w_\alpha^{(a,b)}(\lambda) d\lambda,
\end{align}
which obviously gives the measure in~\eqref{eq:JacobSpMeas} when $t=0$. 
Since the moment problem for the measure in~\eqref{eq:JacobSpMeas} is determinate, we infer from Theorem~\ref{thm:MPevolt} and Corollary~\ref{corIP+} that this solution has the form
    \begin{align}
      \omega(\ledot,t) & =  \sum_{n=1}^{\infty} \omega_n(t) \delta_{x_n(t)}, & \dip(\ledot,t) & = 0,
    \end{align}
 with coefficients given by~\eqref{eq:3.9positive} for each time $t\in\R$.  Moreover, the increasing sequence points $x_1(t),x_2(t),\ldots$ accumulates at $\infty$ as $n\to\infty$.
 Our main result again provides the peakonwise long-time behavior. 

\begin{theorem}\label{th:JacobiPeakons}
For every fixed $n\in\N$ the following asymptotics hold:
\begin{enumerate}[label=(\roman*), ref=(\roman*), leftmargin=*, widest=ii]
\item\label{eq:MPjacobi+Infty}
 As $t\rightarrow\infty$ one has
\begin{align}
\begin{split}
 \omega_n(t)  & = \frac{1}{1+\alpha} + \oo(1), \\ 
 x_n(t) & = \frac{t}{2(1+\alpha)} + (2n+a-1)\log \frac{t}{2(1+\alpha)} - \log j_{n}^+ + \oo(1), 
 \end{split}
 \end{align}
 where 
 \begin{align}
  j_{n}^+ & =  (1+\alpha)^{a+1}(n-1)!\, \Gamma(n+a).
  \end{align}
\item\label{eq:MPjacobi-Infty}
 As $t\rightarrow-\infty$ one has
\begin{align}
\begin{split}
 \omega_n(t) & = \frac{1}{\alpha} + \oo(1), \\
x_n(t) & = \frac{t}{2\alpha} + (2n+b-1)\log \frac{|t|}{2\alpha} - \log j_{n}^-+ \oo(1),
\end{split}
\end{align}
where 
\begin{align}
j_{n}^- = \alpha^{b+1}(n-1)!\, \Gamma(n+b).
\end{align}
\end{enumerate}
\end{theorem}

\begin{figure}[h]
  \centering
  \includegraphics[width=1\textwidth]{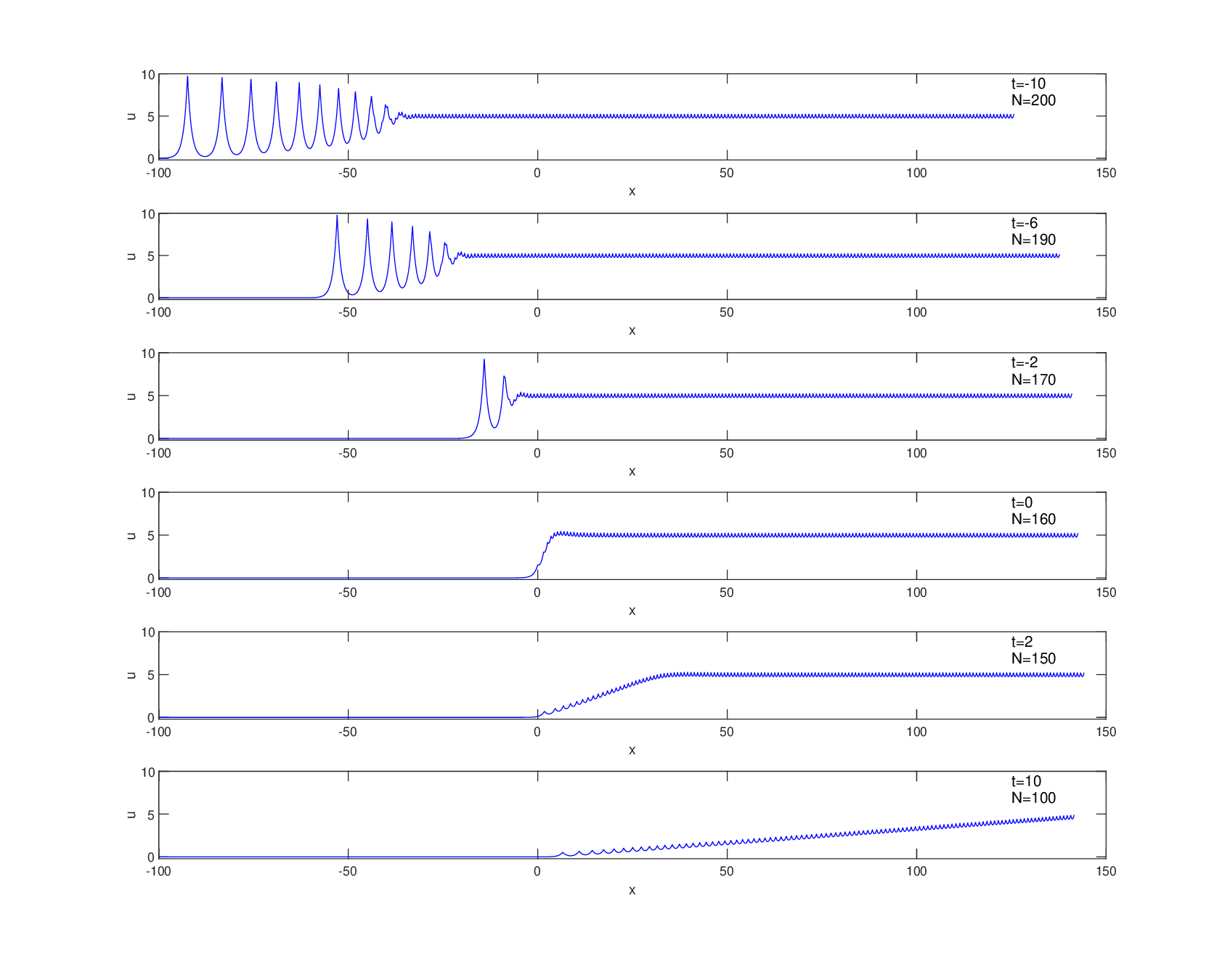}
  \caption{{\small Jacobi--Legendre infinite peakons ($a=b=0$) with $\alpha=\nicefrac{1}{20}$:\\ Snapshots of $u(x,t)$ in the interval $[-100,x_{N+1}]$.
  }}\label{fig:jacNew}
\end{figure}

\begin{remark}\label{rem:JacobiPeakons}
Let us remark that despite the fact that every peakon has asymptotically the same height and moves at asymptotically the same constant speed, the distance between peaks grows logarithmically, that is, for $n\in\N$ one has   
\begin{align}
|x_{n+1}(t)-x_n(t)| & =  \log(t^2) - \log\big(4(1+\alpha)^2(n^2+an)\big)+\oo(1),
\end{align}
 as $t\rightarrow\infty$, and 
\begin{align}
 |x_{n+1}(t)-x_n(t)| & = \log(t^2)  - \log\big(4\alpha^2(n^2+bn)\big) + \oo(1),
\end{align}
as $t\rightarrow-\infty$. 
Let us also mention that the graphs of $u$ with the parameters $a=b=0$ and $\alpha=\nicefrac{1}{20}$ are presented in Figure~\ref{fig:jacNew}; see also~\cite{animations} for animations. We refer to Remark \ref{rem:num} and Remark \ref{rem:graph} for some key points regarding the corresponding numerics.
 \end{remark}

The proof of the above asymptotics relies on an analysis of Hankel determinants. 
 
\begin{lemma}\label{lem:JacobiHankt} 
For all $k$, $n\in\N\cup\{0\}$ the following asymptotics hold:
\begin{enumerate}[label=(\roman*), ref=(\roman*), leftmargin=*, widest=ii]
\item\label{eq:jacobiDelta+Infty}
As $t\rightarrow\infty$ one has 
\begin{align}\label{eq:JacobiHankt+}
\begin{split}
\Delta_{k,n}(t) & \sim  \E^{-\frac{nt}{2(1+\alpha)}}t^{-n(n+a)}J^+_{k,n},\\
J^+_{k,n} &=  2^{n(n+a)}(1+\alpha)^{n(2n+2a+k)}\prod_{j=0}^{n-1}j!\, \Gamma(j+a+1).
\end{split}
\end{align}
\item\label{eq:jacobiDelta-Infty}
  As $t\rightarrow-\infty$ one has 
\begin{align}\label{eq:JacobiHankt-}
\begin{split}
\Delta_{k,n}(t) &\sim   \E^{-\frac{nt}{2\alpha}}|t|^{-n(n+b)}J^-_{k,n},\\
J^-_{k,n} &= 2^{n(n+b)} \alpha^{n(2n+2b+k)}\prod_{j=0}^{n-1}j!\, \Gamma(j+b+1).
\end{split}
\end{align}
\end{enumerate}
\end{lemma}

\begin{proof}
 The multiple integral expression~\eqref{eq:HankelviaIntegral} yields 
\begin{align*}
\Delta_{k,n}(t)&=\frac{1}{n!}\int_{(\alpha,1+\alpha)^n} \mlam^k \wand_n(\mlam) w_\alpha^{(a,b)}(\mlam) \E^{-\frac{t}{2}\sum_{i=1}^n\frac{1}{\lambda_i}}d\mlam \\
&=\frac{\E^{-\frac{nt}{2(1+\alpha)}}}{n!}\int_{(\alpha,1+\alpha)^n} \mlam^k \wand_n(\mlam) w_\alpha^{(a,b)}(\mlam) \E^{-tg(\mlam)}d\mlam,
\end{align*}
where $g\colon [\alpha,1+\alpha]^n \to \R_{\ge 0}$ is given by 
\begin{align*}
g(\mlam) = \frac{1}{2}\sum_{i=1}^n\biggl(\frac{1}{\lambda_i} - \frac{1}{1+\alpha}\biggr).
\end{align*}
Clearly, the function $g$ is non-negative on $[\alpha,1+\alpha]^n$ and attains its minimum at $\mlam_{1+\alpha}$ with $g(\mlam_{1+\alpha})=0$. The asymptotic behavior of the integral 
\begin{align*}
I_{k,n}(t)&=\int_{(\alpha,1+\alpha)^n} \mlam^k \wand_n(\mlam) w_\alpha^{(a,b)}(\mlam) \E^{-tg(\mlam)} d\mlam
\end{align*}
can be obtained by applying Laplace's method. More specifically, the Taylor expansion of $g$ near $\mlam_{1+\alpha}$ is simply given by
\begin{align*}
g(\mlam) = \sum_{i=1}^n \frac{1 + \alpha - \lambda_i}{2(1+\alpha)^2} +\OO(|\mlam-\mlam_{1+\alpha}|^2).
\end{align*}
Arguing as in~\cite[Chapter~4]{deBr}, \cite{mil}, it is not difficult to conclude that  
\begin{align*}
  I_{k,n}(t) \sim \wt{I}_{k,n}(t)&:=\int_{(\alpha,1+\alpha)^n} \mlam^k \wand_n(\mlam) w_\alpha^{(a,b)}(\mlam) \E^{-\frac{t}{2(1+\alpha)^2}\sum_{i=1}^n (1 + \alpha - \lambda_i)} d\mlam.
\end{align*}
Next, from the change of variables $\lambda_i \mapsto 1+\alpha - \frac{2(1+\alpha)^2}{t}\lambda_i$ we get
\begin{align*}
 &  \biggl(\frac{2(1+\alpha)^{2}}{t}\biggr)^{-n^2 - an} \wt{I}_{k,n}(t) \\
& \quad = \int_{\bigl(0,\frac{t}{2(1+\alpha)^2}\bigr)^n}  \biggl(\mlam_{1+\alpha} - \frac{2(1+\alpha)^2}{t}\mlam\biggr)^k \biggl(\mlam_{1} - \frac{2(1+\alpha)^2}{t}\mlam\biggr)^b  \frac{\mlam^a \wand_n(\mlam)}{\E^{\sum_{i=1}^n  \lambda_i}}d\mlam.
\end{align*}
By dominated convergence, the last integral tends to 
 \begin{align*}
(1+\alpha)^{kn} \int_{\R_{>0}^n} \mlam^a \wand_n(\mlam) \E^{-\sum_{i=1}^n  \lambda_i} d\mlam 
\end{align*}
as $t\to \infty$ and it remains to use~\eqref{eq:LaguerreHankelEXPL}.

The justification of the asymptotics when $t\to-\infty$ is verbatim with the only change that now the main contribution in the corresponding multiple integral comes from the left spectral edge, and we omit the details.
\end{proof}

\begin{proof}[Proof of Theorem~\ref{th:JacobiPeakons}]
All the claims follow from straightforward calculations. Indeed, as $t\to\infty$, we easily get from~\eqref{sol_xw} and~\eqref{eq:JacobiHankt+} that 
\begin{align*}
\E^{x_n(t)}  = \frac{\Delta_{2,n-1}(t)}{\Delta_{0,n}(t)} & = \frac{\E^{-\frac{(n-1)t}{2(1+\alpha)}}t^{-(n-1)(n-1+a)}}{\E^{-\frac{nt}{2(1+\alpha)}}t^{-n(n+a)}} \frac{J^+_{2,n-1}}{J^+_{0,n}} (1+\oo(1)) \\
& = \E^{\frac{t}{2(1+\alpha)}}t^{2n+a-1} \frac{J^+_{2,n-1}}{J^+_{0,n}}\left(1 +\oo(1)\right) 
\end{align*}
as well as 
\begin{align*}
\omega_n(t) & =\frac{\Delta_{0,n}(t)\Delta_{2,n-1}(t)}{\Delta_{1,n}(t)\Delta_{1,n-1}(t)}   =  \frac{J^+_{0,n} J^+_{2,n-1}}{J^+_{1,n}J^+_{1,n-1} } + \oo(1).
\end{align*}
The asymptotics as $t\to-\infty$ follow in a similar way.
\end{proof}

\begin{remark}
The above analysis extends to a wider class of spectral measures. Namely, with minor modifications and also taking into account Remark~\ref{rem:scaling}, it is possible to extend the analysis to solutions associated with 
\begin{align}\label{geJac_meas_t}
\rho(d\lambda;t) &= \E^{-\frac{t}{2\lambda}} h(\lambda)w(\lambda)   d\lambda,
\end{align}
where $w$ is the general Jacobi weight in~\eqref{eq:JacobWeightGen} and $h\colon [\tilde{\alpha},\tilde{\beta}]\to (0,\infty)$ is a continuous function. Notice that the latter includes the measure given by~\eqref{eq:rho0-periodic}, which corresponds to the initial data having the form of a half-periodic multi-peakon profile~\eqref{eq:u-1period}. 
We refrain from stating the corresponding results here and leave them to the interested reader. 
 \end{remark}

\appendix

\section{Generalized indefinite strings and the classical moment problem}\label{appMoment}

The major goal in this appendix is to demonstrate how the moment problem can be included into the spectral theory of generalized indefinite strings and also to describe its solutions in the indeterminate case. In fact, a rather detailed account on the former can be found in~\cite{IndMoment}. 

\subsection{Generalized indefinite strings}\label{appA:GIS}
We begin by recalling from~\cite{IndefiniteString} that a {\em generalized indefinite string} is a triple $(L,\tilde{\omega},\tilde{\dip})$ such that $L\in(0,\infty]$, $\tilde{\omega}$ is a real distribution in $H^{-1}_{\loc}[0,L)$ and $\tilde{\dip}$ is a positive Borel measure on the interval $[0,L)$.
The unique function $\tilde{\Wr}$ in $L^2_{\loc}[0,L)$ such that 
 \begin{align}
  \tilde{\omega}(h) & = - \int_0^L \tilde{\Wr}(x)h'(x)dx 
 \end{align}
 for all functions $h\in H^1_{\cc}[0,L)$ is called the {\em normalized anti-derivative} of $\tilde{\omega}$. 
Associated with such a generalized indefinite string $(L,\tilde{\omega},\tilde{\dip})$ is the ordinary differential equation  
  \begin{align}\label{eqnGISODEapp}
  -y'' = z\, \tilde{\omega}y + z^2 \tilde{\dip}y
 \end{align}
 on the interval $[0,L)$, where $z$ is a complex spectral parameter. Of course, this differential equation has to be understood in a weak sense:   
  A solution of~\eqref{eqnGISODEapp} is a function $f\in H^1_{\loc}[0,L)$ such that 
 \begin{align}\label{eqnIntDEho}
 f'\NL h(0) + \int_{0}^L f'(x) h'(x) dx = z\, \tilde{\omega}(fh) + z^2 \int_{0}^L f(x)h(x)\tilde{\dip}(dx)
 \end{align}
 for all functions $h\in H^1_{\cc}[0,L) $ and a (unique) constant $f'\NL \in\C$. 
 With this notion of solution, one can introduce the fundamental system of solutions $\theta(z,\redot)$ and $\phi(z,\redot)$ satisfying the initial conditions
 \begin{align}
  \theta(z,0)& = \phi'\NLz =1, &  \theta'\NLz & = \phi(z,0) =0,
 \end{align}
 for every $z\in\C$; see \cite[Lemma~3.2]{IndefiniteString}. 
 Even though the derivatives of these functions are only locally square integrable in general, there are unique left-continuous functions $\theta^\qd(z,\redot)$ and $\phi^\qd(z,\redot)$ on $[0,L)$ such that 
 \begin{align}
   \theta^\qd(z,x) & = \theta'(z,x) + z\Wr(x)\theta(z,x), & \phi^\qd(z,x) & = \phi'(z,x) + z\Wr(x)\phi(z,x),
 \end{align}
 for almost all $x\in[0,L)$;  see \cite[Equation~(4.12)]{IndefiniteString}.
 These functions will henceforth be referred to as {\em quasi-derivatives} of the solutions $\theta(z,\redot)$ and $\phi(z,\redot)$. 
 As functions of the spectral parameter $z$, the solutions as well as their quasi-derivatives are entire; see~\cite[Corollary~3.5]{IndefiniteString} for example.  
 
 One of the main objects in spectral theory for a generalized indefinite string $(L,\omega,\dip)$ is the Weyl--Titchmarsh function $m$ defined on $\C\backslash\R$ by 
\begin{align}
  m(z) = \lim_{x\to L}-\frac{\theta(z,x)}{z\phi(z,x)}.
\end{align}
This Weyl--Titchmarsh function $m$ turns out to be a {\em Herglotz--Nevanlinna function}, that is, it is analytic, maps the upper complex half-plane $\C_+$ into the closure of the upper complex half-plane and satisfies the symmetry relation
 \begin{align}
  m(z^\ast)^\ast = m(z)
 \end{align}
 for all $z\in\C\backslash\R$. 
 Even more, we established in~\cite[Section~6]{IndefiniteString} that the map $(L,\omega,\dip)\mapsto m$ is a homeomorphism between the set of all generalized indefinite strings (equipped with a reasonable topology) and the set of all Herglotz--Nevanlinna functions (equipped with the topology of locally uniform convergence). 

In addition, we will also need the following simple fact:
 For a given generalized indefinite string $(L,\tilde{\omega},\tilde{\dip})$ and any $\ell\in (0,L)$, let us denote by $S_{\ell}$ the generalized indefinite string obtained by cutting out from the left a piece of length $\ell$; we shall call it a {\em truncated generalized indefinite string}. More specifically, this new string $S_{\ell}$ has length $L-\ell$ and coefficients $\tilde{\omega}_{\ell}$ and $\tilde{\dip}_{\ell}$ defined by 
\begin{align}\label{eq:StringTrunc}
\tilde{\Wr}_{\ell} (x) & = \tilde{\Wr}(x+\ell), &
\int_{0}^xd\tilde{\dip}_{\ell} & = \int_{\ell}^{\ell+x}d\tilde{\dip}.
\end{align}
 
\begin{lemma}\label{lem:GIStrunc}
Let $(L,\tilde{\omega},\tilde{\dip})$ be a generalized indefinite string and let $(L-\ell,\tilde{\omega}_\ell,\tilde{\dip}_\ell)$ be a truncated generalized indefinite string with $\ell\in (0,L)$.  Then the corresponding Weyl--Titchmarsh functions $m$ and $m_\ell$ are connected by
\begin{align}\label{eq:WTtrunc}
m(z) = -\frac{\theta(z,\ell) m_\ell(z)  - \frac{1}{z}\theta^\qd_\ell(z,\ell)}{z\phi(z,\ell)m_\ell(z)-\phi^\qd_\ell(z,\ell)},
\end{align}
where $f^\qd_\ell(x) = f'(x) + z\Wr_\ell(x-\ell)f(x)$.
\end{lemma}

\begin{remark}\label{rem:mcontfracApp}
Let us write down explicitly one particular situation, which will be of interest to us in Section~\ref{secMoment}.
If $(L,\tilde{\omega},\tilde{\dip})$ and $\ell>0$ are such that $\tilde{\omega}|_{[0,\ell)} = \omega_0 \delta_0$ and $\tilde{\dip}|_{[0,\ell)} = 0$, then, taking into account that 
\begin{align*}
\theta(z,x) & = 1-z\omega_0 x, 
& \phi(z,x) & = x, 
\end{align*}
for all $x\in [0,\ell]$, we get 
\begin{align}\label{eq:mcontfracApp}
m(z) = \omega_0 + \cfrac{1}{-z\ell + \cfrac{1}{m_\ell(z)}}.
\end{align}
%
%
%
\end{remark}

We are now going to focus on a subclass of generalized indefinite strings that is particularly relevant in connection with the moment problem. 

\begin{definition}[\cite{IndMoment}]\label{def:KLstring}
A {\em Krein--Langer string} is a generalized indefinite string $(L,\tilde{\omega},\tilde{\dip})$ such that the coefficients $\tilde{\omega}$ and $\tilde{\dip}$ are supported on a discrete set in $[0,L)$. 
\end{definition}

For a Krein--Langer string $(L,\tilde{\omega},\tilde{\dip})$, the distribution $\tilde{\omega}$ and the measure $\tilde{\dip}$ can be written in a unique way as  
    \begin{align}\label{eqnKL}
      \tilde{\omega} & =  \sum_{n=0}^{N} \tilde{\omega}_n \delta_{\tilde{x}_n}, & \tilde{\dip} & =  \sum_{n=0}^{N} \tilde{\dip}_n \delta_{\tilde{x}_n},
    \end{align}
 for some $N\in\N\cup\{0,\infty\}$, strictly increasing points $(\tilde{x}_n)_{n=0}^N$ in $[0,L)$ with $\tilde{x}_0=0$ and $\tilde{x}_n\rightarrow L$ if $N=\infty$, real weights $(\tilde{\omega}_n)_{n=0}^N$ and non-negative weights $(\tilde{\dip}_n)_{n=0}^N$ with $|\tilde{\omega}_n|+\tilde{\dip}_n>0$ for all $n\geq 1$ (notice that we do allow a simultaneous vanishing of $\tilde{\omega}_0$ and $\tilde{\dip}_0$). 
 Here we also use $\delta_x$ to denote the unit Dirac measure centered at $x$.  
 The distances between consecutive point masses are given by 
 \begin{align}\label{eqnKLell}
   \tilde{\ell}_n = \tilde{x}_n - \tilde{x}_{n-1}. 
 \end{align}

\begin{definition}\label{def:KSstring}
A generalized indefinite string $(L,\tilde{\omega},\tilde{\dip})$ such that the distribution $\tilde{\omega}$ is a positive Borel measure on $[0,L)$ and the measure $\tilde{\dip}$ is identically zero is called a {\em Krein string} and is denoted by $(L,\tilde{\omega})$. A {\em Krein--Stieltjes string} is a Krein string $(L,\tilde{\omega})$ such that the coefficient $\tilde{\omega}$ is supported on a discrete set in $[0,L)$. 
\end{definition}

\begin{remark}\label{rem:KreinSting}
Let us emphasize that a generalized indefinite string is a Krein string if and only if its Weyl--Titchmarsh function $m$ is a {\em Stieltjes function}, that is, it is a Herglotz--Nevanlinna function such that $z\mapsto zm(z)$ is also Herglotz--Nevanlinna. Equivalently, a function $m$ on $\C\backslash\R$ is a Stieltjes function if and only if it admits an integral representation
\begin{align}\label{eq:mStieltjes}
m(z) = c - \frac{1}{Lz} + \int_{(0,\infty)}\frac{\rho(d\lambda)}{\lambda-z},
\end{align}
where $c$ is a non-negative constant, $L$ is a constant in $(0,\infty]$ and $\rho$ is a positive Borel measure on $[0,\infty)$ having no point mass at zero and satisfying 
\begin{align}\label{eq:mStieltjesRho}
 \int_{(0,\infty)}\frac{\rho(d\lambda)}{\lambda+1}<\infty.
\end{align}
 For more details we refer to~\cite{kakr74a}. 
 \end{remark}

 Let us now consider the differential equation~\eqref{eqnGISODEapp} for a Krein--Langer string $(L,\tilde{\omega},\tilde{\dip})$. Since $\tilde{\omega}$ and $\tilde{\dip}$ are both measures in this case, the differential equation reduces to the integral equation
   \begin{align}\label{eq:IntEqn}
  f(x) = f(0) + f'(0-)x - z\int_{0}^{x} (x-s) f(s)\,\tilde{\omega}(ds) - z^2 \int_{0}^{x} (x-s) f(s)\,\tilde{\dip}(ds),
  \end{align}
 which in turn is nothing but  
\begin{align}\label{eqLintEqnKL}
f(x) = f(0) + f'(0-)x - \sum_{\tilde{x}_n<x} (x-\tilde{x}_n) (z\,\tilde{\omega}_n+z^2\tilde{\dip}_n) f(\tilde{x}_n).
\end{align}
Clearly, the solution $f$ is thus continuous and piece-wise linear. Moreover, evaluating $f$ at the points $\tilde{x}_n$, we get
\begin{align} \label{eq:recKL1}
 \begin{split}
f'(\tilde{x}_n+) - f'(\tilde{x}_n-) & = - (z\,\tilde{\omega}_n + z^2\tilde{\dip}_n) f(\tilde{x}_n), \\
f(\tilde{x}_n) - f(\tilde{x}_{n-1}) & = \tilde{\ell}_{n} f'(\tilde{x}_{n-1}+) = \tilde{\ell}_{n} f'(\tilde{x}_{n}-),
 \end{split}
\end{align}
 for $n\geq1$. 
 In particular, the representation~\eqref{eqLintEqnKL} shows that the functions $\theta(\ledot,x)$ and $\phi(\ledot,x)$ are polynomials for every $x\in[0,L)$. 
 Upon setting
\begin{align}
m_n(z) = -\frac{\theta(z,\tilde{x}_n)}{z\, \phi(z,\tilde{x}_n)},\quad z\in\C\backslash\R,
\end{align}
it is not difficult to see using \eqref{eq:recKL1} that $m_n$ admits the continued fraction expansion
\begin{align}\label{eq:cfracKL}
m_n(z) = \tilde{\omega}_0 + \tilde{\dip}_0 z+ \cfrac{1}{-\tilde{\ell}_1 z +  \cfrac{1}{\tilde{\omega}_1 + \tilde{\dip}_1 z + \cfrac{1}{\,\ddots\, + \cfrac{1}{\tilde{\omega}_{n-1}+ \tilde{\dip}_{n-1} z + \cfrac{1}{-\tilde{\ell}_n z}}}}}
\end{align}
for all $n\geq1$. In the setting of Krein strings, this connection was discovered by M.\ G.\ Krein, who showed that Krein--Stieltjes strings play a crucial role in the study of the Stieltjes moment problem~\cite{kr52} (see also~\cite[Appendix]{akh} and~\cite[Section~13]{kakr74}). 
 In a similar way, Krein--Langer strings are related to the indefinite moment problem~\cite{krla79} and to the Hamburger moment problem~\cite{IndMoment}. We are going to discuss these connections in the following subsection.
 
 \subsection{The classical moment problem}\label{appA:moment}
Let $(s_k)_{k=0}^\infty$ be a sequence of real numbers. 
 The classical {\em Hamburger moment problem} is to find a positive Borel measure $\rho$ on $\R$ such that the numbers $s_k$ are its moments of order $k$, that is, such that 
\begin{align}\label{eq:Hamburger}
s_k=\int_{\R} \lambda^k\, \rho(d\lambda)
\end{align}
for all $k\in \N\cup \{0\}$.
 Every positive Borel measure $\rho$ on $\R$ that satisfies \eqref{eq:Hamburger} is called a solution of the Hamburger moment problem with data $(s_k)_{k=0}^\infty$. 
 Similarly, the {\em Stieltjes moment problem} is to find a positive Borel measure $\rho$ supported on $[0,\infty)$ such that the numbers $s_k$ are its moments of order $k$, that is, such that 
\begin{align}\label{eq:Stieltjes}
s_k=\int_{[0,\infty)} \lambda^k\, \rho(d\lambda)
\end{align}
for all $k\in \N\cup \{0\}$.
There are two principal questions: 
\begin{enumerate}[label=(\roman*), ref=(\roman*), leftmargin=*, widest=iii]
\item
For which sequences $(s_k)_{k=0}^\infty$ are the moment problems solvable? 
\item
 Are solutions unique? If not, how to describe the set of all solutions?
\end{enumerate}

The answer to the first question is well known (see~\cite{akh, sc17}): {\em The Hamburger (respectively, Stieltjes) moment problem has a solution if and only if the moment sequence $(s_k)_{k=0}^\infty$ is positive (respectively, doubly positive, that is, both sequences $(s_{k})_{k=0}^\infty$ and $(s_{k+1})_{k=0}^\infty$ are positive).} Recall that a sequence $(s_k)_{k=0}^\infty$
is called positive (respectively, strictly positive) if the Hankel determinants\footnote{The Hankel determinants $\Delta_{l,k}$ here are defined by~\eqref{eqnHankel1,2} as in Section~\ref{secMoment}.} $\Delta_{0,n}$ are non-negative (respectively, positive) for all $n\in \N\cup \{0\}$. In particular, the sequence $(s_{k+1})_{k=0}^\infty$ is positive if $\Delta_{1,n}$ are non-negative for all 
$n\in \N\cup \{0\}$. 

Let us now briefly sketch an approach to the second question, which was pioneered by M.\ G.\ Krein~\cite{kr52, kakr74}. More specifically, he observed that the Stieltjes moment problem can be included into the spectral theory of Krein strings, where a major role is played by Krein--Stieltjes strings. Similarly, it was shown in~\cite{IndMoment} that the Hamburger moment problem can be included into the spectral theory of generalized indefinite strings, where the special role is played by Krein--Langer strings. 

From now on, let us suppose that $(s_k)_{k=0}^\infty$ is a strictly positive sequence, which is equivalent to the fact that the corresponding solution to the moment problem is supported on an infinite set.
We define the increasing function 
 \begin{align}
   \kappa\colon \N\rightarrow \N\cup\{0\}
 \end{align}
  such that $\Delta_{1,\kappa(1)},\Delta_{1,\kappa(2)},\ldots$ enumerates all non-zero members of the sequence $(\Delta_{1,k})_{k=0}^\infty$.  
 Since the latter sequence does not have any consecutive zeros, we could alternatively also define $\kappa$ recursively via 
  \begin{align} 
     \kappa(1) & = 0, &   \kappa(n+1) & = \begin{cases} \kappa(n)+1, & \Delta_{1,\kappa(n)+1}\not=0, \\ \kappa(n)+2, & \Delta_{1,\kappa(n)+1}=0. \end{cases}
  \end{align} 
  With this function $\kappa$,  we then set
  \begin{subequations}\label{eqnGISDelApp}
      \begin{align}
                \tilde{x}_n & = \frac{\Delta_{2,\kappa(n)}}{\Delta_{0,\kappa(n)+1}}, \qquad L  = \lim_{n\to \infty} \tilde{x}_n, \label{eqnGISDelxApp} \\
        	  \tilde{\omega}_n & = \frac{\Delta_{-1,\kappa(n)+1}}{\Delta_{1,\kappa(n)}} - \frac{\Delta_{-1,\kappa(n+1)+1}}{\Delta_{1,\kappa(n+1)}}, \label{eqnGISDelomegaApp} \\
	 \tilde{\dip}_n & = \frac{\Delta_{-2,\kappa(n)+2}}{\Delta_{0,\kappa(n)+1}} -  \frac{\Delta_{-2,\kappa(n+1)+1}}{\Delta_{0,\kappa(n+1)}}, \label{eqnGISDeldipApp}
      \end{align}
           \end{subequations}  
           where the determinants $\Delta_{-1,k}$ and $\Delta_{-2,k}$ are computed with $s_{-1}=s_{-2}=0$. 
          This gives rise to a Krein--Langer string $(L,\tilde{\omega},\tilde{\dip})$ with $\tilde{\omega}$ and $\tilde{\dip}$ defined by 
           \begin{align}
             \tilde{\omega} & = \sum_{n=1}^\infty \tilde{\omega}_n \delta_{\tilde{x}_n}, & \tilde{\dip} & = \sum_{n=1}^\infty \tilde{\dip}_n \delta_{\tilde{x}_n}.
           \end{align}
       In this way, we arrive at a map
\begin{align}\label{eq:mapPsi}
\Psi_{\mathcal{S}}\colon (s_k)_{k=0}^\infty\mapsto (L,\tilde{\omega},\tilde{\dip})
\end{align}
from the set of all strictly positive sequences to a subset of Krein--Langer strings. In fact (see~\cite[Theorem~5.3]{IndMoment}), this map establishes: 
\begin{enumerate}[label=(\roman*), ref=(\roman*), leftmargin=*, widest=iii]
\item
A one-to-one correspondence between the set of all strictly positive sequences and the set of all Krein--Langer strings with $|\tilde{\omega}_0|+\tilde{\dip}_0 = 0$ and $N=\infty$.
\item 
 A one-to-one correspondence between the set of all doubly strictly positive sequences and the set of all Krein--Stieltjes strings with $\tilde{\omega}_0 = 0$ and $N=\infty$.
 \end{enumerate}
The latter leads to the following determinacy criterion.

\begin{theorem}\label{th:determinacy}
  Let $(s_k)_{k=0}^\infty$ be a strictly positive sequence.
  Then the following conditions are equivalent:
\begin{enumerate}[label=(\roman*), ref=(\roman*), leftmargin=*, widest=iii]
\item The Hamburger moment problem \eqref{eq:Hamburger} is indeterminate.
\item The series below converges; 
 \begin{align}\label{eq:HambIndet}
 \sum_{n=1}^\infty \tilde{\ell}_{n} + \tilde{\ell}_{n} \tilde{\Wr}_{n}^2 + \tilde{\dip}_{n} & < \infty, & \tilde{\Wr}_n & :=\sum_{k=1}^n \tilde{\omega}_k.
 \end{align}
\end{enumerate}
If the sequence $(s_k)_{k=0}^\infty$ is doubly strictly positive, then the following conditions are equivalent:
\begin{enumerate}[label=(\roman*), ref=(\roman*), leftmargin=*, widest=iii]
\item The Stieltjes moment problem \eqref{eq:Stieltjes} is indeterminate.
\item The series below converges; 
 \begin{align}\label{eq:StilIndet}
 \sum_{n= 1}^\infty \tilde{\ell}_{n} + \tilde{\omega}_{n}  < \infty.
 \end{align}
\end{enumerate}
\end{theorem}

Notice that the above result implies that indeterminacy either in the sense of Stieltjes or in the sense of Hamburger implies that $L$ is finite.
In conclusion, let us provide a description of all solutions to the moment problem in the indeterminate case. It follows from the results of I.\ S.\ Kac~\cite{kac99} and the connection between canonical systems and generalized indefinite strings~\cite{IndefiniteString}, \cite[Appendix~A]{IndMoment}.

\begin{theorem}\label{th:KreinParam}
  Let $(s_k)_{k=0}^\infty$ be a strictly positive sequence such that the corresponding Hamburger moment problem is indeterminate. Then the map 
  \begin{align}\label{eq:mapMPsol01}
  \tilde{m} \mapsto \rho_{\tilde{m}}
  \end{align}
  defined by 
  \begin{align}\label{eq:mapMPsol02}
  \int_\R \frac{\rho_{\tilde{m}}(d\lambda)}{\lambda - z} = -\frac{\theta(z,L) \tilde{m}(z)  - \frac{1}{z}\theta^\qd(z,L)}{z\phi(z,L)\tilde{m}(z)-\phi^\qd(z,L)},
  \end{align}
  where 
  \begin{align}
  \begin{pmatrix} \theta(z,L) & z\phi(z,L) \\ \theta^\qd(z,L)/z & \phi^\qd(z,L)\end{pmatrix} = 
  \lim_{n\to \infty} \begin{pmatrix} \theta(z,\tilde{x}_n) & z\phi(z,\tilde{x}_n) \\ \theta^\qd(z,\tilde{x}_n)/z & \phi^\qd(z,\tilde{x}_n)\end{pmatrix}
  \end{align}
  and $\theta(z,\redot)$, $\phi(z,\redot)$ is the fundamental system of solutions associated with the Krein--Langer string corresponding to the given moment sequence via the map~\eqref{eq:mapPsi}, establishes a one-to-one correspondence between the set of all Herglotz--Nevanlinna functions (augmented by the constant $\tilde{m}=\infty$ function) and the set of all solutions to the Hamburger moment problem.  
  
If   $(s_k)_{k=0}^\infty$ is a doubly strictly positive sequence such that the corresponding Stieltjes moment problem is indeterminate, then the map~\eqref{eq:mapMPsol01}--\eqref{eq:mapMPsol02} establishes a one-to-one correspondence between the set of all Stieltjes functions (augmented by the constant $\tilde{m}=\infty$ function) and the set of all solutions to the Stieltjes moment problem.
\end{theorem}

Recall that a solution $\rho$ to the Hamburger (respectively, Stieltjes) moment problem is called {\em N-extremal} if polynomials are dense in $L^2(\R;\rho)$ (respectively,  in $L^2([0,\infty);\rho)$). A solution $\rho$ is called $n${\em -canonical of order} $n\in\N\cup\{0\}$ if $\tilde{m}$ in~\eqref{eq:mapMPsol02} is a rational function of order $n$. N-extremal solutions are canonical of order $0$. It is known that for $n$-canonical solutions, polynomials are dense in $L^p(\R;\rho)$ for any $p\in [1,2)$, but not for $p=2$ unless $n=0$. In fact, for an $n$-canonical solution $\rho$ of order $n\in \N\cup\{0\}$, the closure of polynomials in $L^2(\R;\rho)$ is a subspace of codimension $n$ in $L^2(\R;\rho)$. 
For more details we refer to~\cite{akh, sc17}.

\begin{remark}
The above parameterization~\eqref{eq:mapMPsol01}--\eqref{eq:mapMPsol02} of solutions to the indeterminate moment problem was obtained by M.\ G.\ Krein for the Stieltjes moment problem~\cite[Section~13]{kakr74} and it is usually called {\em Krein's parameterization}. In view of Lemma~\ref{lem:GIStrunc}, this parameterization admits a very transparent interpretation. Namely, with a given sequence of moments $(s_k)_{k=0}^\infty$ we can associate, via the map~\eqref{eq:mapPsi}, a generalized indefinite string $(L,\tilde{\omega},\tilde{\dip})$, which is a Krein--Langer string. Moreover, its Weyl--Titchmarsh function $m$ is obtained as a limit of rational Herglotz--Nevanlinna functions~\eqref{eq:cfracKL} as $n\to \infty$. Notice that this function $m$ corresponds to setting $\tilde{m}$ equal to infinity in~\eqref{eq:mapMPsol02}. In order to obtain all solutions, we need to consider Weyl--Titchmarsh functions of all the strings obtained from $(L,\tilde{\omega},\tilde{\dip})$ by attaching a new generalized indefinite string to the right endpoint $L$. 
Then $\tilde{m}$ in~\eqref{eq:mapMPsol02} is nothing but the Weyl--Titchmarsh function of the new string truncated at $x=L$.  
 \end{remark}

The procedure described above allows to characterize N-extremal and $n$-canonical solutions to indeterminate moment problems.

\begin{proposition}\label{prop:N-extrem}
A positive Borel measure $\rho$ on $\R$ is an N-extremal solution to an indeterminate Hamburger moment problem  if and only if the following conditions are satisfied:
\begin{enumerate}[label=(\roman*), ref=(\roman*), leftmargin=*, widest=iii]
\item The support of $\rho$ is an unbounded discrete subset of $\R$.
\item\label{itmNextremii} The Cauchy transform of $\rho$,
\begin{align}
m(z) = \int_\R \frac{\rho(d\lambda)}{\lambda - z},\quad z\in\C\backslash\R,
\end{align}
is the Weyl--Titchmarsh function either of a Krein--Langer string satisfying~\eqref{eq:HambIndet} or of a generalized indefinite string $(L,\tilde{\omega},\tilde{\dip})$ with $L=\infty$ such that there is an $\ell \in (0,\infty)$ such that $(\ell,\tilde{\omega}|_{[0,\ell)},\tilde{\dip}|_{[0,\ell)})$ is a Krein--Langer string and the truncated string $(L-\ell,\tilde{\omega}_\ell,\tilde{\dip}_\ell)$ is such that $\tilde{\omega}_\ell = c\delta_0$ with $c\in\R$ and $\tilde{\dip}_\ell$ is identically zero. 
\end{enumerate}
\end{proposition}

\begin{remark}\label{rem:m-canon}
A few remarks are in order:
\begin{enumerate}[label=(\roman*), ref=(\roman*), leftmargin=*, widest=iii]
\item Notice that if $\rho$ is an N-extremal solution to an indeterminate moment problem, then it corresponds to a Krein--Langer string satisfying~\eqref{eq:HambIndet} if and only if $\rho$ has a mass at zero, $\rho(\{0\})>0$. In this case, the mass is equal to the reciprocal of the length of the corresponding generalized indefinite string.  
\item\label{itmremNextremii} In order to get a description of $n$-canonical solutions, it suffices to allow $L$ to be finite in Proposition~\ref{prop:N-extrem}~\ref{itmNextremii} and also to allow truncated generalized indefinite strings to be Krein--Langer strings with finite $N$ in~\eqref{eqnKL}.
\item The description of N-extremal and $n$-canonical solutions to the Stieltjes moment problem is analogous and goes back to the work of Krein~\cite{kr52}.   
\end{enumerate}
\end{remark}

\section*{Acknowledgments}
We thank Jacob Christiansen and Harald Woracek for useful comments and hints with respect to the literature on the indeterminate moment problem and Guoce Xin for his suggestions on how to effectively graph more peaks. We are also grateful to Jacek Szmigielski for encouraging comments and for his interest in the infinite-peakon dynamics.



\begin{thebibliography}{XX}

\bibitem{akh}
N.\ I.\ Akhiezer, {\em The Classical Moment Problem and Some Related Questions in Analysis},
Oliver and Boyd Ltd., Edinburgh, London, 1965.

\bibitem{ba68}
E.\ H.\ Bareiss, {\em Sylvester's identity and multistep integer-preserving Gaussian elimination}, Math.\ Comp.\ {\bf 22} (1968), 565--578.

\bibitem{besasz98}
R.\ Beals, D.\ H.\ Sattinger and J.\ Szmigielski, {\em Acoustic scattering and the extended Korteweg--de Vries hierarchy}, Adv.\ Math.\ {\bf 140} (1998), no.~2, 190--206.

\bibitem{besasz00}
R.\ Beals, D.\ H.\ Sattinger and J.\ Szmigielski, {\em Multipeakons and the classical moment problem}, Adv.\ Math.\ {\bf 154} (2000), no.~2, 229--257.

\bibitem{besasz01}
R.\ Beals, D.\ H.\ Sattinger and J.\ Szmigielski, {\em Peakons, strings, and the finite Toda lattice}, Comm.\ Pure Appl.\ Math.\ {\bf 54} (2001), no.~1, 91--106.

\bibitem{bci02}
Ch.\ Berg, Y.\ Chen and M.\ E.\ H.\ Ismail, {\em Small eigenvalues of large Hankel matrices: the indeterminate case}, Math.\ Scand.\ {\bf 91} (2002), no.~1, 67--81.

\bibitem{beva95}
C.\ Berg and G.\ Valent, {\em Nevanlinna extremal measures for some orthogonal polynomials
related to birth and death processes}, J.\ Comp.\ Appl.\ Math.\ {\bf 57} (1995),  29--43.

\bibitem{bech20}
C.\ Berg and J.\ S.\ Christiansen, {\em The Moment Problem}, in: M.\ E.\ H.\ Ismail, {\em Encyclopedia of Special Functions: the Askey--Bateman Project, Vol.~1: Univariate Orthogonal Polynomials}, pp.~269--306, Cambridge University Press, 2020.

\bibitem{bre16}
A.\ Bressan,  {\em Uniqueness of conservative solutions for nonlinear wave equations via characteristics}, Bull.\ Braz.\ Math.\ Soc., New Ser.\ {\bf 47} (2016), no.~1, 157--169.

\bibitem{bcz15}
A.\ Bressan, G.\ Chen and Q.\ Zhang, {\em Uniqueness of conservative solutions to the Camassa--Holm equation via characteristics}, Discrete Contin.\ Dyn.\ Syst.\ {\bf 35} (2015), no.~1, 25--42.

\bibitem{brco07}
A.\ Bressan and A.\ Constantin, {\em Global conservative solutions of the Camassa--Holm equation}, Arch.\ Ration.\ Mech.\ Anal.\ {\bf 183} (2007), no.~2, 215--239.

\bibitem{brou97}
P.\ W.\ Brouwer, K.\ M.\ Frahm and C.\ W.\ J.\ Beenakker, {\em Quantum mechanical time-delay matrix in chaotic scattering}, Phys.\ Rev.\ Lett.\ {\bf 78} (1997), 4737--4740.

\bibitem{caho93}
R.\ Camassa and D.\ Holm, {\em An integrable shallow water equation with peaked solitons}, Phys.\ Rev.\ Lett.\ {\bf 71} (1993), no.~11, 1661--1664.

\bibitem{chang2014generalized}
X.~Chang, X.~Chen and X.~Hu, {\em A generalized nonisospectral Camassa--Holm equation and its multipeakon solutions}, Adv.\ Math.\ {\bf 263} (2014), 154--177.
 
\bibitem{animations}
X.\ Chang, J.\ Eckhardt and A.\ Kostenko, {\em Animations of Jacobi and Laguerre peakon solutions of the Camassa--Holm equation}, Loughborough University, Media, \doi{10.17028/rd.lboro.29110385}. 

\bibitem{chlizh06}
M.\ Chen, S.\ Liu and Y.\ Zhang, {\em A two-component generalization of the Camassa--Holm equation and its solutions}, Lett.\ Math.\ Phys.\ {\bf 75} (2006), no.~1, 1--15. 

\bibitem{chenits10}
Y.\ Chen and A.\ Its, {\em Painlev\'e III and a singular linear statistics in Hermitian random
matrix ensembles, I}, J.\ Approx.\ Theory {\bf 162} (2010), 270--297.

\bibitem{chi68}
T.\ S.\ Chihara, {\em On indeterminate Hamburger moment problems}, Pacific J.\ Math.\ {\bf 27} (1968),
475--484.

\bibitem{chi70}
T.\ S.\ Chihara, {\em A characterization and a class of distribution functions for the Stieltjes--Wigert polynomials}, Canad.\ Math.\ Bull.\ {\bf 13} (1970), 529--532.

\bibitem{chr03}
J.\ S.\ Christiansen, {\em The moment problem associated with the Stieltjes--Wigert polynomials}, J.\ Math.\ Anal.\ Appl.\ {\bf 277} (2003), 218--245..

\bibitem{chr04}
J.\ S.\ Christiansen, {\em Indeterminate Moment Problems within the Askey-scheme}, PhD Thesis, Univ.\ of Copenhagen, 2004.

\bibitem{co05}
A.\ Constantin, {\em Finite propagation speed for the Camassa--Holm equation}, J.\ Math.\ Phys.\  {\bf 46} (2005), 023506, 4~pp. 

\bibitem{coes98}
A.\ Constantin and J.\ Escher, {\em Global existence and blow-up for a shallow water equation}, Ann.\ Scuola Norm.\ Sup.\ Pisa Cl.\ Sci.\ (4) {\bf 26} (1998), no.~2, 303--328.

\bibitem{coes98b}
A.\ Constantin and J.\ Escher, {\em Wave breaking for nonlinear nonlocal shallow water equations}, Acta Math.\ {\bf 181} (1998), no.~2, 229--243.

\bibitem{coiv08}
A.\ Constantin and R.\ I.\ Ivanov, {\em On an integrable two-component Camassa--Holm shallow water system}, Phys.\ Lett.\ A {\bf 372} (2008), no.~48, 7129--7132. 

\bibitem{deBr}
N.\ G.\ de Bruijn, {\em Asymptotic Methods in Analysis}, North-Holland Publ., 1958.

\bibitem{dlt85}
P.\ Deift,  L.-C.\ Li and C.\ Tomei, {\em Toda flows with infinitely many variables}, J.\ Funct.\ Anal.\ {\bf 64} (1985),
358--402.

\bibitem{ConservCH}
J.\ Eckhardt, {\em The inverse spectral transform for the conservative Camassa--Holm flow with decaying initial data}, Arch.\ Ration.\ Mech.\ Anal.\ {\bf 224} (2017), no.~1, 21--52.

\bibitem{StieltjesType}
J.\ Eckhardt, {\em Continued fraction expansions of Herglotz--Nevanlinna functions and generalized indefinite strings of Stieltjes type}, Bull.\ Lond.\ Math.\ Soc.\  {\bf 54} (2022), no.~2, 737--759.

\bibitem{CHTrace}
J.\ Eckhardt, {\em On the inverse spectral method for solving the Camassa--Holm equation}, in preparation. 

\bibitem{ConservMP}
J.\ Eckhardt and A.\ Kostenko, {\em An isospectral problem for global conservative multi-peakon solutions of the Camassa--Holm equation}, Comm.\ Math.\ Phys.\ {\bf 329} (2014), no.~3, 893--918.

\bibitem{IndefiniteString}
J.\ Eckhardt and A.\ Kostenko, {\em The inverse spectral problem for indefinite strings}, Invent.\ Math.\ {\bf 204} (2016), no.~3, 939--977.

\bibitem{CHPencil}
J.\ Eckhardt and A.\ Kostenko, {\em Quadratic operator pencils associated with the conservative Camassa--Holm flow}, Bull.\ Soc.\ Math.\ France {\bf 145} (2017), no.~1, 47--95.

\bibitem{IndMoment}
J.\ Eckhardt and A.\ Kostenko, {\em The classical moment problem and generalized indefinite strings},  Integr.\ Equat.\ Oper.\ Theory {\bf 90} (2018), no.~2, Art.~23, 30~pp.

\bibitem{InvPeriodMP}
J.\ Eckhardt and A.\ Kostenko, {\em The inverse spectral problem for periodic conservative multi-peakon solutions of the Camassa--Holm equation}, Int.\ Math.\ Res.\ Not.\ IMRN {\bf 2020}, no.~16, 5126--5151.

\bibitem{DSpec}
J.\ Eckhardt and A.\ Kostenko, {\em Generalized indefinite strings with purely discrete spectrum}, in: ``From Complex Analysis to Operator Theory: A Panorama. In Memory of Sergey Naboko", M.\ Brown et al.\  (eds.), Oper.\ Theory: Adv.\ Appl.\ {\bf 291}, 435--474, 2023.

\bibitem{Eplusminus}
J.\ Eckhardt and A.\ Kostenko, {\em The conservative Camassa--Holm flow with step-like irregular initial data}, Proc.\ Lond.\ Math.\ Soc.\ (3) {\bf 130} (2025), no.~5, Paper No.\ e70050, 42~pp. 

\bibitem{ISPforCH}
J.\ Eckhardt and A.\ Kostenko, {\em Trace formulas and inverse spectral theory for generalized indefinite strings}, Invent.\ Math.\ {\bf 238} (2024), no.~2, 391--502.

\bibitem{AsymCS}
J.\ Eckhardt, A.\ Kostenko and G.\ Teschl, {\em Spectral asymptotics for canonical systems}, J.\ Reine Angew.\ Math.\ {\bf 736} (2018), 285--315.

\bibitem{IsospecCH}
J.\ Eckhardt and G.\ Teschl, {\em On the isospectral problem of the dispersionless Camassa--Holm equation}, Adv.\ Math.\ {\bf 235} (2013), 469--495.

\bibitem{CouplingProblem}
J.\ Eckhardt and G.\ Teschl, {\em A coupling problem for entire functions and its application to the long-time asymptotics of integrable wave equations}, Nonlinearity {\bf 29} (2016), no.~3, 1036--1046.

\bibitem{elka}
G.\ A.\ El and A.\ M.\ Kamchatnov, {\em Kinetic equation for a dense soliton gas}, Phys.\ Rev.\ Lett.\ {\bf 95}, 204101 (2005).

\bibitem{esleyi07}
J.\ Escher, O.\ Lechtenfeld and Z.\ Yin, {\em Well-posedness and blow-up phenomena for the 2-component Camassa--Holm equation}, Discrete Contin.\ Dyn.\ Syst.\ {\bf 19} (2007), no.~3, 493--513. 

\bibitem{fofu81}
B.\ Fuchssteiner and A.\ S.\ Fokas, {\em Symplectic structures, their B\"acklund transformations and hereditary symmetries}, Phys.\ D {\bf 4} (1981/82), no.~1, 47--66. 

\bibitem{ga}
F.\ R.\ Gantmacher, {\em The Theory of Matrices. Vol. 1}, translated by K.\ A.\ Hirsch, Amer.\ Math.\ Soc., Chelsea Publishing, New York, 1959.

\bibitem{gewe14}
F.\ Gesztesy and R.\ Weikard, {\em Some remarks on the spectral problem underlying the Camassa--Holm hierarchy}, in {\em Operator theory in harmonic and non-commutative analysis}, 137--188, Oper.\ Theory Adv.\ Appl., 240, Birkh\"auser/Springer, Cham, 2014.

\bibitem{gkz92}
F.\ Gesztesy, W.\ Karwowski and Z.\ Zhao, {\em Limits of soliton solutions}, Duke Math.\ J.\ {\bf 68} (1992), no.~1, 101--150.

\bibitem{grhora12}
K.\ Grunert, H.\ Holden and X.\ Raynaud, {\em Global solutions for the two-component Camassa--Holm system}, Comm.\ Partial Differential Equations {\bf 37} (2012), no.~12, 2245--2271.

\bibitem{grhora12b}
K.\ Grunert, H.\ Holden and X.\ Raynaud, {\em Global conservative solutions to the Camassa--Holm equation for initial data with nonvanishing asymptotics}, Discrete Contin.\ Dyn.\ Syst.\ {\bf 32} (2012), no.~12, 4209--4227.

\bibitem{hora07}
H.\ Holden and X.\ Raynaud, {\em Global conservative solutions of the Camassa--Holm equation---a Lagrangian point of view}, Comm.\ Partial Differential Equations {\bf 32} (2007), no.~10-12, 1511--1549.

\bibitem{hora07c}
H.\ Holden and X.\ Raynaud, {\em Global conservative solutions of the generalized hyperelastic-rod wave equation}, J.\ Differential Equations {\bf 233} (2007), no.~2, 
448--484.

\bibitem{hoiv11}
D.\ D.\ Holm and R.\ I.\ Ivanov, {\em Two-component CH system: inverse scattering, peakons and geometry}, Inverse Problems {\bf 27} (2011), no.~4, 045013, 19 pp.

\bibitem{ism85}
M.\ E.\ H.\ Ismail, {\em A queueing model and a set of orthogonal polynomials}, J.\ Math.\ Anal.\ Appl.\
{\bf 108} (1985), 575--594.

\bibitem{kac99}
I.\ S.\ Kac, {\em Inclusion of Hamburger's power moment problem in the spectral theory of canonical systems}, Zapiski Nauch.\ Seminarov POMI {\bf 262} (1999),  147--171; {\em English transl. in:} J.\  Math.\ Sciences (New York) {\bf 110} (2002), 2991--3004.

\bibitem{kakr74a}
I.\ S.\ Kac and M.\ G.\ Krein, {\em R--functions --- analytic functions mapping the upper half-plane into itself}, Amer.\ Math.\ Soc.\ Transl.\ Ser.\ 2, {\bf 103} (1974), 1--18.

\bibitem{kakr74}
I.\ S.\ Kac and M.\ G.\ Krein, {\em On the spectral functions of the string}, Amer.\ Math.\ Soc.\ Transl.\ Ser.\ 2 {\bf 103} (1974), 19--102.

\bibitem{kr52} 
M.\ G.\ Kre\u{\i}n, {\em On a generalization of investigation of Stieltjes}, Dokl.\ Akad.\ Nauk SSSR {\bf 87} (1952), no.~6, 881--884; (in Russian)  

\bibitem{krla79} 
M.\ G.\ Kre\u{\i}n and H.\ Langer, {\em On some extension problems which are closely connected with the theory of Hermitian operators in a space $\Pi_\kappa$. III. Indefinite analogues of the Hamburger and Stieltjes moment problems. Part I.} Beitr\"age Anal.\ No.\ 14 (1979), 25--40; {\em Part II.} Beitr\"age Anal.\ No.\ 15 (1980), 27--45.

\bibitem{la76}
H.\ Langer, {\em Spektralfunktionen einer Klasse von Differentialoperatoren zweiter Ordnung mit nichtlinearem Eigenwertparameter}, Ann.\ Acad.\ Sci.\ Fenn.\ Ser.\ A I Math.\ {\bf 2} (1976), 269--301. 

\bibitem{li08}
L.-C.\ Li, {\em Factorization problem on the Hilbert--Schmidt group and the Camassa--Holm equation}, Comm.\ Pure Appl.\ Math.\ {\bf 61} (2008), 186--209.

\bibitem{li09}
L.-C.\ Li, {\em Long time behaviour for a class of low-regularity solutions of the Camassa--Holm equation}, Commun.\ Math.\ Phys.\ {\bf 285}  (2009), no.~1, 265--291.

\bibitem{lyu19}
S.\ Lyu, J.\ Griffin and Y.\ Chen, {\em The Hankel determinant associated with a singularly perturbed Laguerre unitary ensemble}, J.\ Nonlinear Math.\ Phys.\ {\bf 26} (2019), 24--53.

\bibitem{mar91}
V.\ A.\ Marchenko, {\em The Cauchy problem for the KdV equation with nondecreasing initial data}, 
in: V.\ E.\  Zakharov (ed.), {\em  What is Integrability?}, Springer Series in Nonlinear Dynamics, pp.~273--318, 1990.

\bibitem{mc04}
H.\ P.\ McKean, {\em Breakdown of the Camassa--Holm equation}, Comm.\ Pure Appl.\ Math.\ {\bf 57} (2004), no.~3, 416--418. 

\bibitem{mehta}
M.\ L.\ Mehta, {\em Random Matrices}, 3rd edn., Acad.\ Press, Elsevier Ltd, 2004.

\bibitem{Mezzadri13}
F.\ Mezzadri and N.\ J.\ Simm, {\em Tau-function theory of chaotic quantum transport with $\beta= 1, 2, 4$}, Comm.\ Math.\ Phys.\ {\bf 324} (2013), 465--513.

\bibitem{mil}
P.\ D.\ Miller, {\em Applied Asymptotic Analysis}, Grad.\ Texts in Math.\ {\bf 75}, Amer.\ Math.\ Soc., Providence, RI, 2006.

\bibitem{dlmf}
F.\ W.\ J.\ Olver et al.,
{\em NIST Handbook of Mathematical Functions},
Cambridge University Press, Cambridge, 2010;
\url{http://dlmf.nist.gov} .

\bibitem{osipov07}
V.\ A.\ Osipov and E.\ Kanzieper, {\em Are bosonic replicas faulty?}, Phys.\ Rev.\ Lett.\ {\bf 99} (2007), 050602.

\bibitem{sash03}
A.\ M.\ Savchuk and A.\ A.\ Shkalikov, {\em Sturm--Liouville operators with distribution potentials}, Trans.\ Moscow Math.\ Soc.\ {\bf 2003} (2003), 143--190.

\bibitem{sc17}
K.\ Schm\"udgen, {\em The Moment Problem}, Springer, Cham, 2017. 

\bibitem{sim09}
B.\ Simon, {\em The Christoffel--Darboux kernel}, in: ``Perspectives in PDE, Harmonic Analysis and Applications", Proc.\ Sympos.\ Pure Math.\ {\bf 79}, pp.~295--335, Amer.\ Math.\ Soc., Providence, 2008.

\bibitem{sti}
T.\ J.\ Stieltjes, {\em Recherches sur les fractions continues}, Ann.\ Fac.\ Sci.\ Toulouse {\bf 8} (1894), 1--122;  Ann.\ Fac.\
Sci.\ Toulouse {\bf 9} (1895), 1--47.

\bibitem{sze}
G.\ Szeg\H{o}, {\em Orthogonal Polynomials}, 4th edn., Amer.\ Math.\ Soc., Providence, RI, 1975.

\bibitem{texier13}
C.\ Texier and S.\ N.\ Majumdar, {\em Wigner time-delay distribution in chaotic cavities and freezing transition,}
Phys.\ Rev.\ Lett.\ {\bf 110} (2013), 250602. 

\bibitem{ww06}
Z.\ Wang and R.\ Wong, {\em Uniform asymptotics of the Stieltjes--Wigert polynomials via the Riemann--Hilbert approach}, J.\ Math.\ Pures Appl.\ {\bf 85} (2006), no.~5,  698--718.

\bibitem{won18}
R.\ Wong, {\em Asymptotics of orthogonal polynomials}, Intern.\ J.\ Numer.\ Anal.\ Model.\ {\bf 15} (2018), no.~1-2,  193--212.

\bibitem{xizh00}
Z.\ Xin and P.\ Zhang, {\em On the weak solutions to a shallow water equation}, Comm.\ Pure Appl.\ Math.\ {\bf 53} (2000), no.~11, 1411--1433.

\bibitem{xu14cmp}
 S.\ X.\ Xu, D.\ Dai and Y.\ Q.\ Zhao, {\em Critical edge behavior and the Bessel to Airy transition
in the singularly perturbed Laguerre unitary ensemble}, Comm.\ Math.\ Phys.\ {\bf 332} (2014), 1257--1296.
 
\bibitem{xu15jat}
 S.\ X.\ Xu, D.\ Dai, and Y.\ Q.\ Zhao, {\em Painlev\'e III asymptotics of Hankel determinants for a
singularly perturbed Laguerre weight}, J.\ Approx.\ Theory {\bf 192} (2015), 1--18.

\bibitem{zak}
V.\ E.\ Zakharov, {\em Turbulence in integrable systems}, Stud.\ Appl.\ Math.\ {\bf 122} (2009), 219--234.

\end{thebibliography}
\end{document}